\documentclass[a4paper, 11pt]{article}

\usepackage[margin=1in]{geometry}
\usepackage[english]{babel}

\usepackage{amsthm,amssymb}
\usepackage[colorlinks=false, linktocpage=true]{hyperref}
\usepackage{amsmath,graphicx}
\usepackage{amscd}
\usepackage{mathtools}
\usepackage{enumerate}
\usepackage{url}
\usepackage{tikz}
\usetikzlibrary{positioning}
\usetikzlibrary{shapes.geometric}

\usepackage{paralist}

\date{14th May, 2021}
\title{Twisted quadratic foldings of root systems\\ and liftings of Schubert classes}
\author{Maiko Serizawa\\
\small Department of Mathematics and Statistics\\[-0.8ex]
\small University of Ottawa\\
\small\tt mseri065@uottawa.ca
}

\theoremstyle{plain}
\newtheorem{theorem}{Theorem}[section]
\newtheorem{lemma}[theorem]{Lemma}
\newtheorem{corollary}[theorem]{Corollary}
\newtheorem{proposition}[theorem]{Proposition}

\newtheorem{claim}[theorem]{Claim}

\theoremstyle{definition}
\newtheorem{definition}[theorem]{Definition}
\newtheorem{example}[theorem]{Example}

\newtheorem{question}[theorem]{Question}

\theoremstyle{remark}
\newtheorem{remark}[theorem]{Remark}

\newcommand{\ra}{\rightarrow}

\newcommand{\hra}{\hookrightarrow}

\newcommand{\longra}{\longrightarrow}

\newcommand{\La}{\Leftarrow}
\newcommand{\Ra}{\Rightarrow}

\newcommand{\ke}{{\mathcal E}}
\newcommand{\kf}{{\mathcal F}}
\newcommand{\kg}{{\mathcal G}}

\newcommand{\kl}{{\mathcal L}}

\newcommand{\kr}{{\mathcal R}}

\newcommand{\kt}{{\mathcal T}}
\newcommand{\ku}{{\mathcal U}}

\newcommand{\kz}{{\mathcal Z}}

\newcommand{\C}{{\mathbb C}}

\newcommand{\IE}{{\mathbb E}}

\newcommand{\Q}{{\mathbb Q}}
\newcommand{\R}{{\mathbb R}}

\newcommand{\Z}{{\mathbb Z}}

\newcommand{\gi}{{\mathfrak i}}

\newcommand{\Aut}{\operatorname{Aut}}
\newcommand{\GL}{\operatorname{GL}}

\newcommand{\Sym}{\operatorname{Sym}}

\newcommand{\inv}{^{-1}}

\newcommand{\dual}[1]{{#1}^{\vee}}
\newcommand{\blform}[2]{({#1}, {#2})}

\newcommand{\tauform}[2]{({#1}, {#2})_{\tau}}

\newcommand{\lbar}{\overline}

\newcommand{\Lattice}{\Lambda}
\newcommand{\RLattice}{\Lambda_r}

\newcommand{\Root}{\Phi}
\newcommand{\Simple}{\Delta}
\newcommand{\Weyl}{{W}}

\newcommand{\cd}{\Gamma}
\newcommand{\red}[1]{\kr({#1})}
\newcommand{\FM}[1]{F({#1})}

\newcommand{\sym}{\mathcal{S}}
\newcommand{\moment}{\kg}
\newcommand{\pmoment}{\kg^P}
\newcommand{\tmoment}{\moment_{\tau}}
\newcommand{\tpmoment}{\moment^P_{\tau}}
\newcommand{\sa}[1]{\kz({#1})}

\newcommand{\freemodule}[1]{\bigoplus_{#1}\sym}

\newcommand{\labelling}{\kl}

\newcommand{\cgen}[1]{s_{#1}}
\newcommand{\reflection}[1]{s_{#1}}
\newcommand{\Reflection}[1]{R_{#1}}

\newcommand{\basek}{K}
\newcommand{\basel}{L}
\newcommand{\Utau}{U_{\basel}^{(\tau)}}
\newcommand{\Usig}{U_{\basel}^{(\sigma)}}

\newcommand{\tauo}{\kt}
\newcommand{\tauproj}{\pi_{\tau}}
\newcommand{\TWeyl}{\Weyl_{\tau}}
\newcommand{\embedding}{\varepsilon}
\newcommand{\length}[1]{l(#1)}
\newcommand{\tlength}[1]{l_{\tau}(#1)}
\newcommand{\adm}{\varphi}
\newcommand{\admred}[1]{\kr_{\adm}(#1)}

\newcommand{\schubert}[1]{\sigma^{(#1)}}
\newcommand{\tschubert}[1]{\sigma_{\tau}^{(#1)}}

\newcommand{\foldingset}[1]{{#1}^{\kf}}

\newcommand{\Para}[1]{P_{{#1}}}
\newcommand{\TPara}[1]{(P_{\tau})_{{#1}}}

\newcommand{\coho}[2]{H^{#1}({#2})}
\newcommand{\tcoho}[2]{H^{#1}_T({#2})}

\DeclareMathOperator{\spn}{span}
\DeclareMathOperator{\pr}{pr}

\begin{document}
\maketitle

\begin{abstract}
Given a finite crystallographic root system $\Root$ whose Dynkin diagram has a non-trivial automorphism, it yields a new root system $\Root_{\tau}$ by a so-called classical folding. On the other hand, Lusztig's folding (1983) folds the root system of type $E_8$ to $H_4$ starting from an automorphism of the root lattice of type $E_8$. The notion of a \emph{twisted quadratic folding} of a root system was introduced by Lanini--Zainoulline (2018) to describe both the classical foldings and Lusztig's folding on the same footing. 
The \emph{structure algebra} $\sa{\moment}$ of the moment graph $\moment$ associated with a finite root system and its reflection group $\Weyl$ is an algebra  over a certain polynomial ring $\sym,$ whose underlying module is free with a distinguished basis $\{\schubert{w} \mid w \in \Weyl\}$ called \emph{combinatorial Schubert classes}.
By Lanini--Zainoulline (2018), a twisted quadratic folding $\Root \rightsquigarrow \Root_{\tau}$ induces an embedding of the respective Coxeter groups $\embedding: \TWeyl \hra \Weyl$ and a ring homomorphism $\embedding^*: \sa{\moment} \ra \sa{\tmoment}$ between the corresponding structure algebras.
This paper studies the $\embedding^*$-preimage of Schubert classes and in particular provides a combinatorial criterion for a Schubert class $\tschubert{u}$ of $\sa{\tmoment}$ to admit a Schubert class $\schubert{w}$ of $\sa{\moment}$ such that the relation $\embedding^*(\schubert{w}) = c \cdot \tschubert{u}$ holds for some nonzero scalar $c$. 
\end{abstract}

\section{Introduction}

This paper builds on the connection of two major mathematical constructions, \emph{foldings} of finite root systems and the \emph{structure algebras} of moment graphs associated with finite root systems. These two were recently connected  in the work of Lanini--Zainoulline \cite{LZ18}. 

A \emph{folding} of a root system is a procedure long known in the literature to exploit the symmetry of a given root system and yield a smaller root system of a different type (e.g. Steinberg \cite{St67}). There are four families of classical foldings of finite crystallographic root systems: foldings from type $A_{2n-1}$ to $C_n \, (n \geq 2),$ from type $D_{n+1}$ to $B_n \, (n \geq 3),$ from type $E_6$ to $F_4,$ and from type $D_4$ to $G_2$ (cf. Springer \cite[Section~10]{Springer}), in which each folding starts with an automorphism of the Dynkin diagram of the given root system. The other well-known folding is the Lusztig folding from type $E_8$ to $H_4$ discovered by Lusztig in \cite[Section~3.9 (b)]{Lusztig83}, in which the folding begins with an automorphism of the root lattice of type $E_8$ not induced by an automorphism of the Dynkin diagram (cf. Moody--Patera \cite{MP93}). Recently, Lanini--Zainoulline~\cite{LZ18} introduced the notion of a \emph{twisted quadratic folding} of a root system which explains the construction of both classical quadratic foldings (those listed above except the one from type $D_4$ to $G_2$) and the Lusztig folding on the same footing via the notion of a \emph{folded representation} of a root system. Let $\basek$ be a unital subring of $\R,$ $\basel$ be a free quadratic $\basek$-algebra given by $\basel = \basek[x]/(p(x))$ for some monic separable $\basek$-polynomial $p(x) = x^2 -c_1x +c_2,$ and $\tau$ be one of the two distinct roots of $p(x)$ in $\basel.$ A free $\basek$-module $U$ obtained from some free $\basel$-module $\ku$ by restriction of scalars is equipped with a $\basek$-linear map $\tauo$ given by multiplication by $\tau$ on $\ku.$ 
A folded representation realizes a finite crystallographic root system $\Root$ as a subset of $U$ so that its root lattice is stable under $\tauo$ and its simple roots admit a certain partition with respect to $\tauo$ (\cite[Definition~3.1]{LZ18}). Then the \emph{$\tau$-twisted folding} of $\Root$ is the procedure to pass from $\Root$ to a new finite root system $\Root_{\tau}$ under the projection onto the $\tau$-eigenspace of $\tauo.$ In this framework, the classical foldings are obtained by taking $p(x)=x^2-1,$ and the Lusztig folding is obtained by taking $p(x)=x^2-x-1$ (\cite[Examples~4.3--4.5]{LZ18}).

Our second objects of study, the notions of a \emph{structure algebra} and its \emph{Schubert classes}, stem from the study of equivariant cohomology rings of flag varieties and its generalization to arbitrary finite Coxeter groups.
Let $G$ be a connected semisimple linear algebraic group over $\C$ with a maximal torus $T,$ a Borel subgroup $B$  containing $T,$ and the Weyl group $\Weyl.$ Let $\Lattice$ be the character group of $T$ and $\sym:=\Sym_{\Z}\Lattice$ be the symmetric algebra. Then the structure algebra associated with $\Weyl$ is given by
\begin{equation*} 
\sa{\moment} = \{\xi \in \bigoplus_{v \in \Weyl} \sym \mid \xi(v)-\xi(v') \,\, \mbox{is divisible by} \,\, \alpha \in \Root^+ \, \mbox{if} \, v <_{\alpha} v'\},
\end{equation*}
where $\xi(v)$ denotes the $v$-coordinate of $\xi$ for each $v\in \Weyl,$ $<_{\alpha}$ denotes the Bruhat cover relation on $\Weyl$ and $\moment$ is a labelled graph called a \emph{moment graph} (Fiebig \cite[Sections~3]{F08coxeter}), which in this case coincides with the Bruhat graph of $\Weyl.$ The structure algebra $\sa{\moment}$ is an $\sym$-algebra under the coordinate-wise addition and multiplication, and by a famous result of Goresky--Kottwitz--MacPherson \cite[Theorem~7.2]{GKM97}, the $T$-equivariant cohomology ring $\tcoho{\bullet}{G/B, \Z}$ is isomorphic to $\sa{\moment}$ as $\sym$-algebras via the so-called localization map (see also \cite[Theorem~8.11]{CZZ15}). The notions of a moment graph and its structure algebra can be extended to any finite real reflection groups by an appropriate choice of $\sym.$

It is well-known that the ordinary cohomology ring $\coho{\bullet}{G/B, \R}$ has two descriptions, one via the additive basis called Schubert classes (cf. Borel \cite{Borel54}) and the other via the so-called coinvariant algebra $\sym/\sym^{\Weyl}_+$ (cf. Borel \cite{Borel53}), where $\sym^{\Weyl}_+$ is the ideal generated by $\Weyl$-invariant polynomials with no constant terms. In their independent work, Bernstein--Gelfand--Gelfand~\cite{BGG73} and Demazure~\cite{Demazure73} constructed polynomials in the coinvariant algebra which correspond to the Schubert classes of $\coho{\bullet}{G/B}$ under Borel's isomorphism. Hiller~\cite{Hiller82} proved that the polynomials constructed in \cite{BGG73} form an additive basis of the coinvariant algebra for any finite real reflection group $\Weyl.$ 
Kaji \cite{Ka11} extended Hiller's result to the equivariant setting, constructing additive basis elements in the double coinvariant algebra $\sym \otimes_{\sym^{\Weyl}} \sym$ for any finite real reflection group $\Weyl,$ which in case $\Weyl$ is crystallographic correspond to the Schubert classes of $\tcoho{\bullet}{G/B}$ under the Borel-type isomorphism from $\tcoho{\bullet}{G/B}$ to $\sym \otimes_{\sym^{\Weyl}} \sym$ (\cite[Proposition~4.1]{Ka11}). He further showed an isomorphism from the double coinvariant algebra to the structure algebra $\sa{\moment}$ (\cite[Corollary~4.10]{Ka11}), thereby giving an explicit formula \eqref{eq:schubert0} for the images $\schubert{w}$ ($w \in \Weyl$) of his double Schubert polynomials under this isomorphism: for $w, v \in \Weyl$ and $v=\reflection{1}\cdots\reflection{\ell}$ any reduced expression,
\begin{equation} \label{eq:schubert0}
\schubert{w}(v) = \sum \prod_{p=1}^{\length{w}} \mathbf{r}_{i_p}(\alpha_{i_p}),
\end{equation}
where  the sum runs over all $1 \leq i_1 < \cdots < i_{\length{w}} \leq \ell$ such that $\reflection{i_1}\cdots\reflection{i_{\length{w}}}=w,$ $\alpha_{j}$ is a simple root corresponding to $\reflection{j},$ and $\mathbf{r}_{i_p}(\alpha_{i_p})=\reflection{1}\reflection{2}\cdots\reflection{i_p-1}(\alpha_{i_p}).$

Let $\Root \rightsquigarrow \Root_{\tau}$ be a twisted quadratic folding of a finite crystallographic root system.
Let $\Weyl$ and $\TWeyl$ be the finite reflection groups of $\Root$ and $\Root_{\tau}$ respectively. The folding $\Root \rightsquigarrow \Root_{\tau}$ induces an embedding $\embedding: \TWeyl \hra \Weyl$ of Coxeter groups (cf. Lusztig  \cite[Corollary~3.3]{Lusztig83}). If we denote the moment graph associated with the original root system $\Root$ by $\moment$ and the one associated with the $\tau$-folded root system $\Root_{\tau}$ by $\tmoment,$ the main result of Lanini--Zainoulline \cite[Theorem~6.2]{LZ18} states that there is a ring homomorphism $\embedding^*: \sa{\moment} \ra \sa{\tmoment}$ induced from the embedding $\embedding.$ 
This paper studies the Schubert calculus of the folding map $\embedding^*.$ We take \eqref{eq:schubert0} as definition of a Schubert class of $\sa{\moment}$ (Definition~\ref{def:schubert}), and investigate how the Schubert classes of the original structure algebra $\sa{\moment}$ relate to those of the folded structure algebra $\sa{\tmoment}$ under $\embedding^*.$
As a first step to understand the $\embedding^*$-preimages of the Schubert classes of $\sa{\tmoment},$ we find an answer to the following question.

\begin{question} \label{q:lifting}
Which Schubert class $\tschubert{u} \in \sa{\tmoment}$ admits a Schubert class $\schubert{w} \in \sa{\moment}$ satisfying the following condition?
\begin{equation} \label{eq:liftingrel}
\embedding^*(\schubert{w})=c\cdot \tschubert{u} \,\,  \mbox{for some} \,\, c \in \basel^{\times}.
\end{equation}
\end{question}

\begin{definition} \label{def:schubertlifting}
We say that a Schubert class $\tschubert{u}$ \emph{is liftable} or \emph{admits a lifting} if there exists some $\schubert{w} \in \sa{\moment}$ satisfying \eqref{eq:liftingrel}.
Such a Schubert class $\schubert{w}$ is called a \emph{lifting} of $\tschubert{u},$ and we say that $\tschubert{u}$ \emph{lifts to} $\schubert{w}.$
\end{definition}
\noindent
Our interesting discovery is that  the condition for a given Schubert class $\tschubert{u}$ of $\sa{\tmoment}$ to admit a lifting in $\sa{\moment}$ involves only the existence of an element $w \in \Weyl$ satisfying certain combinatorial conditions in relation to $u \in \TWeyl.$ More precisely, we introduce a map $\adm: S \ra S_{\tau}$ on the Coxeter generating sets of $\Weyl$ and $\TWeyl$ (Section~\ref{sec:comfolding}), and define the \emph{folding subset} $\foldingset{\Weyl}$ of $\Weyl$ in terms of how $\adm$ behaves on the reduced expressions of each element $w \in \Weyl$ (Definition~\ref{def:foldingset}). Our main result is the following.

\begin{theorem}[Corollary~\ref{prop:liftingcondition2}, Proposition~\ref{prop:liftingcoefficient}] \label{Mthm1}
Let $u \in \TWeyl$ and $w \in \Weyl.$ Let $w=\reflection{i_1}\cdots\reflection{i_{\ell}}$ be any $\Weyl$-reduced expression. Then the Schubert class $\schubert{w} \in \sa{\moment}$ is a lifting of $\tschubert{u} \in \sa{\tmoment}$ if and only if $w \in \foldingset{\Weyl}$ and $u = \adm(\reflection{i_1})\cdots\adm(\reflection{i_{\ell}}).$ Moreover, the coefficient $c$ in \eqref{eq:liftingrel} equals $\tau^m$ for some unique nonnegative integer $m,$ where $\tau$ is the parameter appearing in the definition of the underlying twisted quadratic folding.
\end{theorem}

If every Schubert class of $\sa{\tmoment}$ admits a lifting, then we say that $\embedding^*$ has the \emph{lifting property}. In general, $\embedding^*$ does not have this property. Even in the smallest examples of a twisted quadratic folding from type $A_3$ to $C_2$ or from type $A_4$ to $H_2,$ the Schubert class of $\sa{\tmoment}$ corresponding to the longest element $u_0$ of $\TWeyl$ is not liftable (cf. Example~\ref{egg:typeA4lifting}). Our criterion for nonliftability (Corollary~\ref{prop:noliftingcondition}) suggests that as the size of the folded root system increases, there will be more and more nonliftable Schubert classes in $\sa{\tmoment}.$ However, by passing to the parabolic setting, in some cases, it is possible to choose a parabolic subgroup $\Weyl_P \subsetneq \Weyl$ so that the folding map $\embedding^*: \sa{\pmoment} \ra \sa{\tpmoment}$ does have the lifting property (Example~\ref{egg:typeD6lifting}). As an application of Theorem~\ref{Mthm1}, in Section~\ref{sec:liftingprop}, we provide a list of parabolic subgroups $\Weyl_P$ such that the folding map $\embedding^*: \sa{\pmoment} \ra \sa{\tpmoment}$ has the lifting property (Proposition~\ref{prop:liftingprop}).

The paper is organized as follows. In Section~\ref{sec:coxeter}, we recall the notions of a root datum and a root system and some facts about finite Coxeter groups essential to the later discussion. In Sections~\ref{sec:moment}--\ref{sec:schubert}, we introduce the notions of a \emph{parabolic moment graph} and its \emph{structure algebra}, then discuss a distinguished basis of the structure algebra called \emph{(combinatorial) Schubert classes}. In Section~\ref{sec:folding}, we briefly summarize the Lanini-Zainoulline construction of a \emph{twisted quadratic folding} of a root system and introduce our main object of study, the induced folding map $\embedding^*: \sa{\pmoment} \ra \sa{\tpmoment}$ between the structure algebras.
Sections~\ref{sec:comfolding}--\ref{sec:liftingprop} investigate the liftings of Schubert classes under the folding map $\embedding^*.$  
In Section~\ref{sec:comfolding}, we define a combinatorial folding map $\adm: S \ra S_{\tau}$ and prove its key properties. In Section~\ref{sec:lifting}, we prove Theorem~\ref{Mthm1} and describe some examples. Section~\ref{sec:nolifting} discusses how to detect nonliftable Schubert classes and provides a useful criterion (Corollary~\ref{prop:noliftingcondition}). Finally, Section~\ref{sec:liftingprop} applies the results of Sections~\ref{sec:lifting} and \ref{sec:nolifting} to each known twisted quadratic folding to provide a complete list of parabolic subsets $P \subset S$ such that the folding map $\embedding^*: \sa{\pmoment} \ra \sa{\tpmoment}$ has the lifting property.

\section{Coxeter groups and root systems} \label{sec:coxeter}
This section sets the notation and collects some facts about root data, finite root systems, and finite Coxeter groups that will be frequently used throughout the text.

\subsection{Root data and root systems} \label{ssec:roots}

This paper works with two approaches to a finite root system, that is, the notion of a root datum following \cite[Expos{\'e}~XXI]{SGA} and that of a finite root system following \cite[Section~1.2]{Humphreys}.

We denote a root datum by $\Root \hra \dual{\Lattice},$ where the convention is as follows: 
\begin{inparaenum}[1.)]
\item $\Lattice$ is a free $\Z$-module of finite rank,
\item $\Root \subset \Lattice$ is a nonempty finite subset,
\item $\dual{\Lattice}$ is the dual $\Z$-module,
\item $\dual{\alpha}$ denotes the image of $\alpha \in \Root,$
\item $\reflection{\alpha}(x) = x - \dual{\alpha}(x) \alpha$ for $x \in \Lattice,$
\item $\Weyl(\Root) := \langle \reflection{\alpha} \mid \alpha \in \Root \rangle \leq \GL_{\Z}(\Lattice)$ is the Weyl group of $\Root.$ 
\item $\Simple$ denotes a simple system,
\item $\RLattice := \spn_{\Z}(\Root)$ is the root lattice.
\end{inparaenum}
We assume that a root datum is reduced, that is, if $\Root \cap c \Root \neq \emptyset$ for some $c \in \Z$, then $c=\pm 1$.

\begin{definition}\label{def:georep}
Let $\Root \hra \dual{\Lattice}$ be a root datum and $\basek \subset \R$ be a (unital) subring. Let $U$ be a free $\basek$-module containing $\Lattice \otimes_{\Z} \basek$ as a $\basek$-submodule. Let $\blform{-}{-}$ denote a dot product on $U.$ If $\dual{\alpha}(x)=\frac{2\blform{\alpha}{x}}{\blform{\alpha}{\alpha}}$ for all $\alpha \in \Root$ and $x \in \Lattice$, we call the inclusion $\Root \subset \Lattice \hra \Lattice \otimes_{\Z} \basek \subset U$ a \emph{geometric $\basek$-representation} of the root datum and denote it by $(\Root, U)_{\basek}.$ By extending the formula of $\reflection{\alpha}$ $(\alpha \in \Root)$ over $U$, we obtain $\reflection{\alpha} \in \GL_{\basek}(U).$ We define the Weyl group of $(\Root, U)_{\basek}$ to be the subgroup $\langle \reflection{\alpha} \mid \alpha \in \Root \rangle$ of $\GL_{\basek}(U).$
\end{definition}

Let $\IE$ be a Euclidean space with a dot product $\blform{-}{-}.$
We denote a root system by $\Root$ and use the following convention:
\begin{inparaenum}[1.)]
\item $\Root \subset \IE \backslash \{0\}$ is a nonempty finite subset,
\item $\reflection{\alpha}(x) = x - \tfrac{2\blform{\alpha}{x}}{\blform{\alpha}{\alpha}}\alpha$ for  $x \in \IE,$
\item $\Weyl(\Root) := \langle \reflection{\alpha} \mid \alpha \in \Root \rangle \leq \GL_{\R}(\IE)$ is the Weyl group of $\Root.$ 
\end{inparaenum}
We assume that a root system is reduced, that is, if $\Root \cap c \Root \neq \emptyset$ for some $c \in \R$, then $c=\pm 1$.
The \emph{minimal coefficient ring} of $\Root$ is the minimal subring $\basek_{min} \subset \R$ containing all the values $\tfrac{2\blform{\alpha_i}{\alpha_j}}{\blform{\alpha_i}{\alpha_i}}$ where $\alpha_i, \alpha_j \in \Simple.$ If $\Root$ is crystallographic, $\basek_{min}=\Z.$

\begin{definition} \label{def:rootsystemK}
Let $\Root$ be a root system with the minimal coefficient ring $\basek_{min}$ and $\basek \subset \R$ be any subring containing $\basek_{min}.$ Let $U$ be a free $\basek$-module equipped with a dot product $\blform{-}{-}.$ Then $\Root$ admits a realization $\Root \subset U$ satisfying the defining conditions of a root system. The realization $\Root \subset U$ is called a \emph{$\basek$-realization} of the root system. We define the Weyl group of $\Root \subset U$ to be the subgroup $\langle \reflection{\alpha} \mid \alpha \in \Root \rangle$ of $\GL_{\basek}(U).$
\end{definition}
\noindent
The two distinct notions as defined in Definitions~\ref{def:georep} and \ref{def:rootsystemK} will be used in the construction of a twisted quadratic folding in Section~\ref{sec:folding}, where the input of a folding is a root datum and the output is a not necessarily crystallographic finite root system.

\subsection{Coxeter groups} \label{ssec:coxeter}
Let $(\Weyl, S)$ be a finite Coxeter system. We use the following notation:
\begin{inparaenum}[1)]
\item $m(s,t) \in \Z_{\geq 2}$ is the order of $st$ in $\Weyl$ for $s, t \in S,$
\item $l: \Weyl \ra \Z_{\geq 0}$ is the length function,
\item $(\Weyl, \leq)$ is the Bruhat poset,
\item $\FM{S}$ is the free monoid generated by $S,$
\item $[s, t]_{m}$ is the word $stst \cdots \in \FM{S}$ of length $m,$
\item $\red{w} \subset \FM{S}$ is the set of all reduced words representing $w \in \Weyl,$
\item $\red{\Weyl} := \bigcup_{w \in \Weyl} \red{w},$
\item $\Weyl_P := \langle s \mid s \in P \rangle \leq \Weyl$ for a (possibly empty) subset $P \subset S,$
\item $\Weyl^P := \{w \in \Weyl \mid \length{ws}>\length{w} \,\, \text{for all} \,\, s \in P\},$
\item $\lbar{w}$ is the unique element of $w\Weyl_P \cap \Weyl^P$ for $w \in \Weyl.$
\end{inparaenum}

The following list of well-known results will be frequently used in the main discussion. The reader is referred to \cite{Brenti} and \cite{Humphreys} for the details.

\begin{lemma}[{\cite{Humphreys}}] \label{prop:coxeter} 
\begin{enumerate}[\normalfont (a)] 
\item Let $w \in \Weyl,$ $\alpha \in \Root^+.$ Then $\length{w\reflection{\alpha}} > \length{w}$ if and only if $w(\alpha) \in \Root^+.$ Equivalently, $\length{\reflection{\alpha}w} > \length{w}$ if and only if $w\inv(\alpha) \in \Root^+.$
\item Let $w, w' \in \Weyl$. Let $w = \cgen{1} \cdots \cgen{k}$ be a reduced expression. Then $w' \leq w$ if and only if $w'$ admits a reduced expression which is a subexpression of $\cgen{1} \cdots \cgen{k}$. 
\end{enumerate}
\end{lemma}

\noindent
We identify $\Weyl$ with a finite real reflection group acting on a finite dimensional Euclidean space $U$ via the standard geometric representation.
Let $\theta \in U$ be a dominant vector with respect to a simple system $\Simple$ whose stabilizer subgroup is $W_P$.
The $W$-orbit $W\theta$ is equipped with a partial order given by taking the transitive closure of the following binary  relations for $x \in U:$
\begin{equation*} 
x < \reflection{\alpha}(x) \,\, \text{for all} \,\, \alpha \in \Root^+ \,\, \text{with} \,\, \blform{x}{\alpha} > 0.
\end{equation*}

\begin{lemma}[{\cite[Proposition~1.1]{Stembridge}}] \label{prop:orbit}
The orbit map $\Weyl \ra \Weyl\theta: w \mapsto w(\theta)$ restricts to a poset isomorphism $(\Weyl^P, \leq) \simeq (\Weyl \theta, \leq).$
\end{lemma}

\section{Moment graphs and structure algebras} \label{sec:moment}

In this paper, we work with a class of moment graphs called \emph{parabolic moment graphs} (cf. \cite[Example~4.3]{DLZ19}). 
Let $\Simple_P \subset \Simple$ be a (possibly empty) subset.
The Bruhat poset $(\Weyl, \leq)$ restricts to the subposet $(\Weyl^P, \leq).$ The following data defines a \emph{parabolic moment graph} $\pmoment$ associated with the pair $(\Root, \Root_P).$
\begin{enumerate}[(PM1)]
\item The vertex set is the poset $(\Weyl^P, \leq),$
\item The edge set is $\ke := \{w \ra \lbar{\reflection{\alpha}w} \mid w \in \Weyl^P, \alpha \in \Root^+, w < \lbar{\reflection{\alpha}w}\},$ 
\item The labelling function is $\kl(w \ra \lbar{\reflection{\alpha}w}) = \alpha$ for each $\alpha \in \Root^+.$
\end{enumerate}
Note that $w < \lbar{\reflection{\alpha}w}$ if and only if $w < \reflection{\alpha}w$ and $w \neq \lbar{\reflection{\alpha}w}$ (use Lemma~\ref{prop:orbit}).
A proof for the well-definedness of the labelling function $\kl$ can be found in \cite[Example~4.3]{DLZ19} for example.
We will denote the parabolic moment graph associated with $(\Root, \emptyset)$ by $\moment$ instead of $\moment^{\emptyset}$ to simplify notation. It coincides with the \emph{Bruhat graph} as defined in \cite{Brenti}.

\begin{example}[Type $A_2$] \label{egg:momentA2}
Let $\Root$ be the root system of type $A_2$ with simple roots denoted by $\{\alpha_1, \alpha_2\}.$ Let $\reflection{i}$ denote the simple reflection associated with $\alpha_i.$ Let $\Simple_P=\{\alpha_1\}.$ Below are the moment graphs associate with $(\Root, \emptyset)$ (left) and $(\Root, \Root_P)$ (right).

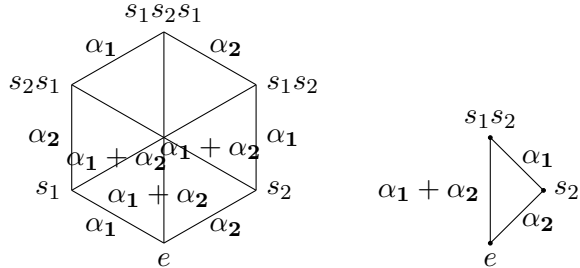
\begin{center}
\begin{tikzpicture}[scale=0.7]
\draw  (1.73205, -1) -- (0, 0) node[anchor=south] {$s_1s_2s_1$};
\draw (-1.73205, -1) -- (0, 0) ;
\draw (0, -4) -- (0, 0) ;
\draw (1.73205, -3) -- (1.73205, -1) node[anchor=west] {$s_1s_2$};
\draw (-1.73205, -3) -- (1.73205, -1);
\draw (-1.73205, -3) --  (-1.73205, -1) node[anchor=east] {$s_2s_1$};
\draw (1.73205, -3) --  (-1.73205, -1) ;
\draw (0, -4) -- (-1.73205, -3) node[anchor=east] {$s_1$};
\draw (0, -4) --  (1.73205, -3) node[anchor=west] {$s_2$};
\draw (0, -4) node[anchor=north] {$e$};
\draw (-0.1,-3.5) node[anchor=south] {$\mathbf{\alpha_1+\alpha_2}$};
\draw  (-0.666025, -3.7) node[anchor=east] {$\mathbf{\alpha_1}$};
\draw (0.666025, -3.7) node[anchor=west] {$\mathbf{\alpha_2}$};
\draw (-1.73205, -2) node[anchor=east] {$\mathbf{\alpha_2}$};
\draw (1.73205, -2) node[anchor=west] {$\mathbf{\alpha_1}$};
\draw  (-0.666025, -0.3) node[anchor=east] {$\mathbf{\alpha_1}$};
\draw (0.666025, -0.3) node[anchor=west] {$\mathbf{\alpha_2}$};
\draw (-0.88,-2.8) node[anchor=south] {$\mathbf{\alpha_1+\alpha_2}$};
\draw (0.88,-2.6) node[anchor=south] {$\mathbf{\alpha_1+\alpha_2}$};
\end{tikzpicture}
\quad
\begin{tikzpicture}[scale=0.7]
\filldraw[black] (0,0) circle (1pt) node[anchor=south] {$\reflection{1}\reflection{2}$};
\filldraw[black] (0,-2) circle (1pt) node[anchor=north] {$e$};
\filldraw[black] (1,-1) circle (1pt) node[anchor=west] {$\reflection{2}$};

\draw (0,0) -- (0,-2);
\draw (0,0) -- (1,-1);
\draw (1,-1) -- (0,-2);

\draw (0,-1) node[anchor=east] {$\mathbf{\alpha_1+\alpha_2}$};
\draw (0.4,-1.6) node[anchor=west] {$\mathbf{\alpha_2}$};
\draw (0.4,-0.4) node[anchor=west] {$\mathbf{\alpha_1}$};
\end{tikzpicture}
\end{center}
\end{example}

Let $\pmoment = ((\Weyl^P, \leq), \ke, \labelling)$ be the moment graph associated with $(\Root, \Root_P)$ defined above.
Let $\Root \subset U$ be a $\basek$-realization of the root system for some subring $\basek \subset \R$ containing $\basek_{min}$ (Definition \ref{def:rootsystemK}). 
Set $\sym := \Sym_{\basek}U,$ the symmetric algebra of $U$ over $\basek.$ 

\begin{definition}[{\cite[Section~3.2]{F08coxeter}}] \label{def:sa}
The \emph{structure algebra} of $\pmoment$ is given by
\begin{equation*}
\sa{\pmoment} := \{(z_w)_{w \in \Weyl^P} \in \freemodule{w \in \Weyl^P}  \mid z_w - z_{w'} \in \kl(w \ra w')\sym \, \, \mbox{for all } \, w \ra w' \in \ke \}.
\end{equation*}
\end{definition}

\noindent
By direct computation, one can verify that $\sa{\pmoment}$ is an $\sym$-algebra under the coordinate-wise addition, multiplication, and scalar multiplication.

\begin{remark}  \label{rmk:labeldegree}
\begin{enumerate}[\normalfont (a)]
\item An element $\xi \in \sa{\pmoment}$ can be viewed either as an $|\Weyl^P|$-tuple of $\sym$-elements or as a function $\Weyl^P \ra \sym.$ Using this second perspective, we will often denote the $w$-coordinate of $\xi$ by $\xi(w).$
\item Notice that the polynomial algebra $\sym$ depends on the $\basek$-realization of the root system $\Root \subset U.$ When the $\basek$-realization $\Root \subset U$ is coming from a geometric $\basek$-representation $(\Root, U)_{\basek}$ of a root datum, there is a canonical choice for $\sym$ as follows. 
Let $\Root \hra \dual{\Lattice}$ be a root datum with a simple system $\Simple$ and a (possibly empty) subset $\Simple_P \subset \Simple.$ Let $\basek \subset \R$ be a (unital) subring. Let $(\Root, U)_{\basek}$ be a geometric $\basek$-representation of the root datum. By definition of $(\Root, U)_{\basek},$ $U$ contains $U_{\Lattice}:=\Lattice \otimes_{\Z} \basek$ as a free $\basek$-submodule. We take $\sym=\Sym_{\basek} U_{\Lattice}$ in Definition \ref{def:sa} to obtain the structure algebra $\sa{\pmoment}.$
\item Each label $\alpha \in \Root^+$ of the moment graph $\pmoment$ has degree one as an element of $\sym.$
\end{enumerate}
\end{remark}

\begin{example}[Type~$A_2$] \label{egg:saA2modP2}
Let $\basek=\Z.$ Let $\Root \subset U$ be a $\basek$-realization of the root system of type $A_2$ such that $U=\spn_{\basek}\Root$ with a simple system $\Simple =\{\alpha_1, \alpha_2\}$ and a subset $\Simple_P=\{\alpha_1\}.$  Then $\sym=\Z[\alpha_1, \alpha_2]$ and
\begin{equation*}
\begin{split}
\sa{\pmoment} = \Big\{(z_e, z_{s_2}, z_{s_1s_2}) \in \sym \oplus \sym \oplus \sym \mid z_e - z_{s_2} \in \alpha_2 \sym, \\ z_{s_2}-z_{s_1s_2} \in \alpha_1 \sym, z_e-z_{s_1s_2} \in (\alpha_1 + \alpha_2) \sym \Big\}.
\end{split}
\end{equation*}
\end{example}

\section{Combinatorial Schubert classes} \label{sec:schubert}
In this section, we define the notion of \emph{combinatorial Schubert classes} of a structure algebra $\sa{\pmoment}$ and state their key properties mostly without proofs. The reader is referred to \cite[Appendix~D]{AppD} and \cite{Ka11} for the details.

\begin{definition}[cf. {\cite[Appendix~D.3]{AppD}, \cite[Corollary~4.11]{Ka11}}] \label{def:schubert}
The \emph{combinatorial Schubert class corresponding to} $w \in \Weyl$ is a function $\schubert{w}: \Weyl \ra \sym$ given by the following formula: let $x \in \Weyl$ and let $x=\reflection{1}\reflection{2}\cdots\reflection{r}$ be a (not necessarily reduced) expression and we denote the simple root corresponding to $\reflection{i}$ by $\alpha_i$ for each $1 \leq i \leq r.$ Then we define if $w \nleq x,$ $\schubert{w}(x)=0$; and if $w \leq x,$
\begin{equation} \label{eq:schubertcoordinate}
\schubert{w}(x) =  \sum_{\gi} \prod_{p =1}^{\ell} \reflection{1} \reflection{2} \cdots \reflection{i_{p}-1}(\alpha_{i_{p}})
\end{equation}
where $\ell = \length{w},$ $\gi = (i_1, \cdots, i_{\ell}) \in \Z^{\ell}$ such that $1 \leq i_1 < i_2 < \cdots < i_{\ell} \leq r$ and $w=\reflection{i_1} \reflection{i_2} \cdots \reflection{i_{\ell}}.$ By convention, we set $\schubert{e}(x)=1$ for all $x \in \Weyl.$
\end{definition}

\noindent
It is well-known that the formula \eqref{eq:schubertcoordinate} does not depend on the choice of the expression $x = \reflection{1}\reflection{2}\cdots\reflection{r}$ (cf. \cite[Section~4]{Billey97} for when $\Weyl$ is crystallographic).

From Definition \ref{def:schubert}, it is not clear if the Schubert class $\schubert{w}$ belongs to $\sa{\moment}$ for each $w \in \Weyl.$ However, this is indeed the case (\cite[Lemma~D.4]{AppD}), and for $w \in \Weyl^P,$ the restriction $\schubert{w}|_{\Weyl^P}$ belongs to $\sa{\pmoment}$ (Lemma~\ref{prop:schubertcoset} (b)). Moreover, the set of all Schubert classes $\{\schubert{w} \mid w \in \Weyl^P \}$ forms an $\sym$-module basis for $\sa{\pmoment}$ (Lemma~\ref{prop:basis}).

\begin{example}[Type $A_2$] \label{egg:schubertA2}
Below are the Schubert classes of the structure algebra $\sa{\moment}$ where $\moment$ is of type $A_2.$ 
From top to bottom, left to right, the diagrams depict the Schubert classes $\schubert{e};$ $\schubert{\reflection{1}}$ and $\schubert{\reflection{2}};$ $\schubert{\reflection{2}\reflection{1}}$ and $\schubert{\reflection{1}\reflection{2}};$ and finally $\schubert{\reflection{1}\reflection{2}\reflection{1}}.$

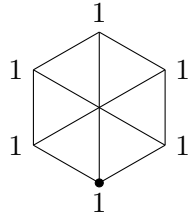
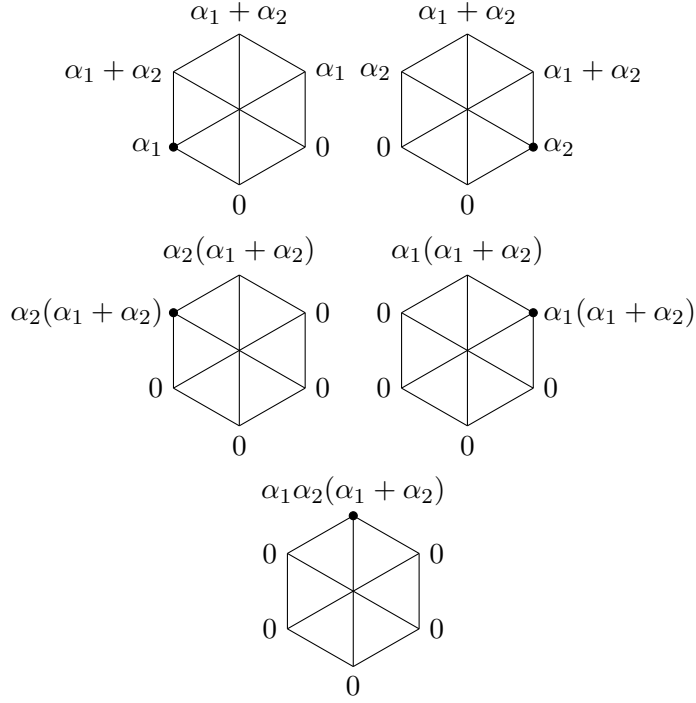
\centering
\begin{tikzpicture}[scale=0.5]
\draw  (1.73205, -1) -- (0, 0) node[anchor=south] {$1$};
\draw (-1.73205, -1) -- (0, 0) ;
\draw (0, -4) -- (0, 0) ;
\draw (1.73205, -3) -- (1.73205, -1) node[anchor=west] {$1$};
\draw (-1.73205, -3) -- (1.73205, -1);
\draw (-1.73205, -3) --  (-1.73205, -1) node[anchor=east] {$1$};
\draw (1.73205, -3) --  (-1.73205, -1) ;
\draw (0, -4) -- (-1.73205, -3) node[anchor=east] {$1$};
\draw (0, -4) --  (1.73205, -3) node[anchor=west] {$1$};
\draw (0, -4) circle (2pt) node[anchor=north] {$1$};
\filldraw[black] (0, -4) circle (3pt);
\end{tikzpicture}

\begin{tikzpicture}[scale=0.5]
\draw  (1.73205, -1) -- (0, 0) node[anchor=south] {$\alpha_1+\alpha_2$};
\draw (-1.73205, -1) -- (0, 0) ;
\draw (0, -4) -- (0, 0) ;
\draw (1.73205, -3) -- (1.73205, -1) node[anchor=west] {$\alpha_1$};
\draw (-1.73205, -3) -- (1.73205, -1);
\draw (-1.73205, -3) --  (-1.73205, -1) node[anchor=east] {$\alpha_1+\alpha_2$};
\draw (1.73205, -3) --  (-1.73205, -1) ;
\draw (0, -4) -- (-1.73205, -3) node[anchor=east] {$\alpha_1$};
\draw (0, -4) --  (1.73205, -3) node[anchor=west] {$0$};
\draw (0, -4) node[anchor=north] {$0$};
\filldraw[black] (-1.73205, -3) circle (3pt);

\begin{scope}[xshift=6cm, on grid]
\draw  (1.73205, -1) -- (0, 0) node[anchor=south] {$\alpha_1+\alpha_2$};
\draw (-1.73205, -1) -- (0, 0) ;
\draw (0, -4) -- (0, 0) ;
\draw (1.73205, -3) -- (1.73205, -1) node[anchor=west] {$\alpha_1+\alpha_2$};
\draw (-1.73205, -3) -- (1.73205, -1);
\draw (-1.73205, -3) --  (-1.73205, -1) node[anchor=east] {$\alpha_2$};
\draw (1.73205, -3) --  (-1.73205, -1) ;
\draw (0, -4) -- (-1.73205, -3) node[anchor=east] {$0$};
\draw (0, -4) --  (1.73205, -3) node[anchor=west] {$\alpha_2$};
\draw (0, -4) node[anchor=north] {$0$};
\filldraw[black]  (1.73205, -3) circle (3pt); 
\end{scope}
\end{tikzpicture}

\begin{tikzpicture}[scale=0.5]
\draw  (1.73205, -1) -- (0, 0) node[anchor=south] {$\alpha_2 (\alpha_1+\alpha_2)$};
\draw (-1.73205, -1) -- (0, 0) ;
\draw (0, -4) -- (0, 0) ;
\draw (1.73205, -3) -- (1.73205, -1) node[anchor=west] {$0$};
\draw (-1.73205, -3) -- (1.73205, -1);
\draw (-1.73205, -3) --  (-1.73205, -1) node[anchor=east] {$\alpha_2 (\alpha_1+\alpha_2)$};
\draw (1.73205, -3) --  (-1.73205, -1) ;
\draw (0, -4) -- (-1.73205, -3) node[anchor=east] {$0$};
\draw (0, -4) --  (1.73205, -3) node[anchor=west] {$0$};
\draw (0, -4) node[anchor=north] {$0$};
\filldraw[black]  (-1.73205, -1) circle (3pt);

\begin{scope}[xshift=6cm, on grid]
\draw  (1.73205, -1) -- (0, 0) node[anchor=south] {$\alpha_1 (\alpha_1+\alpha_2)$};
\draw (-1.73205, -1) -- (0, 0) ;
\draw (0, -4) -- (0, 0) ;
\draw (1.73205, -3) -- (1.73205, -1) node[anchor=west] {$\alpha_1 (\alpha_1+\alpha_2)$};
\draw (-1.73205, -3) -- (1.73205, -1);
\draw (-1.73205, -3) --  (-1.73205, -1) node[anchor=east] {$0$};
\draw (1.73205, -3) --  (-1.73205, -1) ;
\draw (0, -4) -- (-1.73205, -3) node[anchor=east] {$0$};
\draw (0, -4) --  (1.73205, -3) node[anchor=west] {$0$};
\draw (0, -4) node[anchor=north] {$0$};
\filldraw[black]  (1.73205, -1) circle (3pt);
\end{scope}
\end{tikzpicture}

\begin{tikzpicture}[scale=0.5]
\draw  (1.73205, -1) -- (0, 0) node[anchor=south] {$\alpha_1 \alpha_2 (\alpha_1+\alpha_2)$};
\draw (-1.73205, -1) -- (0, 0) ;
\draw (0, -4) -- (0, 0) ;
\draw (1.73205, -3) -- (1.73205, -1) node[anchor=west] {$0$};
\draw (-1.73205, -3) -- (1.73205, -1);
\draw (-1.73205, -3) --  (-1.73205, -1) node[anchor=east] {$0$};
\draw (1.73205, -3) --  (-1.73205, -1) ;
\draw (0, -4) -- (-1.73205, -3) node[anchor=east] {$0$};
\draw (0, -4) --  (1.73205, -3) node[anchor=west] {$0$};
\draw (0, -4) node[anchor=north] {$0$};
\filldraw[black] (0, 0) circle (3pt);
\end{tikzpicture}
\end{example}

\noindent
Here are the basic properties of a Schubert class $\schubert{w}$ stated without a proof.

\begin{lemma}[{\cite[Appendix~D.1]{AppD}}] \label{prop:schubert1} 
Let $w \in \Weyl.$
\begin{enumerate}[\normalfont (a)]
\item For each $x \in \Weyl,$ $\schubert{w}(x)$ is homogeneous of degree $\length{w}.$
\item $\schubert{w}(w) = \prod_{\alpha \in \Root^+, w\inv(\alpha) \in \Root^-} \alpha.$
\item If $w \leq x,$ then $\schubert{w}(x) \neq 0.$
\end{enumerate}
\end{lemma}

\begin{remark} \label{rmk:schubertdegree}
By (a), we may call the common homogeneous degree of $\schubert{w}(x) \,(x \in \Weyl)$ the \emph{degree of} $\schubert{w},$ denoted by $\deg \schubert{w}.$
By Lemmas~\ref{prop:schubert1} (b) and \ref{prop:coxeter} (a), we have
\begin{equation*}
\deg{\schubert{w}} = \length{w} = |\{\alpha \in \Root ^+ \mid \length{\reflection{\alpha}w} < \length{w}\}| = |\{w' \ra w \in \ke \mid w' \in \Weyl^P\}|.
\end{equation*}
\end{remark}
\noindent
Lemmas~\ref{prop:schubertind}--\ref{prop:schubertcoset} establish that for $w \in \Weyl^P,$ the restriction $\schubert{w}|_{\Weyl^P}$ belongs to $\sa{\pmoment}.$

\begin{lemma}[{\cite[Appendix~D.2]{AppD}}] \label{prop:schubertind}
Let $x \in \Weyl$ and let $\alpha$ be a simple root. 
\begin{align} \label{eq:schubertind}
\schubert{w}(x\reflection{\alpha}) =
\begin{cases}
\schubert{w}(x), & \text{if} \,\,\, \length{w\reflection{\alpha}} > \length{w} \\
\schubert{w}(x) + x(\alpha)\schubert{w\reflection{\alpha}}(x), & \text{if} \,\,\, \length{w\reflection{\alpha}} < \length{w}
\end{cases}
\end{align}
\end{lemma}

\begin{lemma} \label{prop:schubertcoset}
\begin{enumerate}[\normalfont (a)]
\item Let $w \in \Weyl^P,$ $x \in \Weyl$ and $y \in \Weyl_P.$ Then $\schubert{w}(xy)=\schubert{w}(x).$
\item Let $w, x \in \Weyl^P$ and $\alpha \in \Root^+.$ Then $\schubert{w}(x) - \schubert{w}(\lbar{\reflection{\alpha}x}) \in \alpha \sym.$
\end{enumerate}
\end{lemma}

\begin{proof}
(a) \, Let $y=\reflection{1}\reflection{2}\cdots\reflection{\ell}$ be a reduced expression and denote the simple reflection associated with $\reflection{i}$ by $\alpha_i$ for each $1 \leq i \leq \ell.$ Note that each $\alpha_i$ belongs to $\Simple_P.$ Since $w \in \Weyl^P,$ we have $\length{w\reflection{i}}>\length{w}$ for each $1 \leq i \leq \ell.$ By applying \eqref{eq:schubertind} repeatedly, we obtain
\begin{equation*}
\schubert{w}(x)=\schubert{w}(x\reflection{1})=\schubert{w}(x\reflection{1}\reflection{2}) = \cdots \schubert{w}(x\reflection{1}\reflection{2}\cdots\reflection{\ell}) = \schubert{w}(xy).
\end{equation*}
(b) \, We have $\reflection{\alpha}x = \lbar{\reflection{\alpha}x}y$ for some $y \in \Weyl_P.$ Applying (a), we obtain
\begin{equation*}
\schubert{w}(x) - \schubert{w}(\lbar{\reflection{\alpha}x}) = \schubert{w}(xy\inv) - \schubert{w}(\reflection{\alpha}xy\inv) \in \alpha \sym. \qedhere
\end{equation*}
\end{proof}

\noindent
When we speak of a Schubert class of $\sa{\pmoment}$ corresponding to $w \in \Weyl^P,$ we mean the restriction $\schubert{w}|_{\Weyl^P}.$ For the notational convenience, we denote $\schubert{w}|_{\Weyl^P}$ also by $\schubert{w}.$

\begin{lemma} [{\cite[Appendix~D.5]{AppD}}] \label{prop:basis}
The set $\{\schubert{w} \mid w \in \Weyl^P \}$ is an $\sym$-basis of $\sa{\pmoment}.$
\end{lemma}

The following characterization of a Schubert class will be the key in Section~\ref{sec:lifting}.

\begin{lemma}[cf. {\cite[Proposition~4.9]{Ka11}}] \label{ob}
Let $w \in \Weyl^P$ and $\xi := (x_{w'})_{w' \in \Weyl^P} \in \sa{\pmoment}.$ The following are equivalent:
\begin{enumerate}[\normalfont (a)]
\item{
$\xi$ is homogeneous of degree $\length{w}$ such that $x_w \neq 0$ and $x_{w'}=0$ for all $w' \ngeq w,$
}
\item{
 $\xi = c \cdot \schubert{w}$ for some nonzero constant $c \in \basek.$ 
}
\end{enumerate}
\end{lemma}

\begin{proof}
We give a proof that (a) implies (b) (the other direction is clear). 
Let $w=e.$ If $\xi$ is homogeneous of degree $\length{e}=0$ with $c:=x_e \neq 0,$ then $x_{w'}=c$ for all $w' \in \Weyl^P$ by the edge relations of $\sa{\pmoment}.$ Thus $\xi=c \cdot \schubert{e}.$
Let $w>e.$ Since $x_{w'}=0$ for all $w'<w,$ 
for $x_w$ ($\neq 0$) to satisfy all the relations $x_{w'}-x_w \in \kl(w' \ra w)\sym,$ it must be divisible by the product $\prod_{w'<w}{\kl(w' \ra w)}$ (since each label is a prime element of $\sym$). Notice that $\deg(\prod_{w'<w}{\kl(w' \ra w)}) = \sum_{w' < w} {\deg(\kl(w' \ra w))} = \length{w}$ (Remarks~\ref{rmk:labeldegree} (c) and \ref{rmk:schubertdegree}).
Since $x_w$ also has degree $\length{w},$ we obtain
$x_w = c \prod_{w'<w}{\kl(w' \ra w)}$
for some nonzero $c \in \basek.$ 
Take the difference $\xi - c\cdot \schubert{w} := (y_{w'})_{w' \in \Weyl^P}.$ Observe that $y_w = 0.$ We also have $y_{w'}=0$ for all $w' \ngeq w$ by the assumption on $\xi$ as well as Proposition \ref{prop:schubert1}~(c). We show that $y_{w'}=0$ for all $w'>w$ as well. Let $v \in \Weyl^P$ such that $v>w$ and $\length{v}=\length{w}+1.$ Then for all $w' <v,$ we obtain $y_{w'}=0$ thanks to $\length{w'} \leq \length{w},$ $y_w=0$ and $y_{w''}=0$ for all $w'' \ngeq w.$ Hence for $y_v$ to satisfy all the relations $y_{w'}-y_v \in \kl(w' \ra v)\sym,$ it must be divisible by the product $\prod_{w'<v}{\kl(w' \ra v)},$ which has degree $\length{v}$ by the same argument made for $x_w$ above. Since $\deg (\xi - c\cdot \schubert{w}) = \length{w} < \length{v},$ $y_v=0$ follows.
We repeat the argument to obtain $y_{w'}=0$ for all $w' > v$ and $\length{w'}=\length{v}+1.$ We repeat this process until we finally obtain $y_{\lbar{w_0}}=0$ for the longest element $\lbar{w_0}$ of $\Weyl^P.$
Thus we obtain $y_{w'}=0$ for all $w' > w$ and complete the proof.
\end{proof}

\section{Twisted quadratic foldings and induced folding map} \label{sec:folding}
In this section, we recall the notion of a \emph{twisted quadratic folding} of a finite crystallographic root system introduced by Lanini--Zainoulline \cite{LZ18}.

Let $\basek \subset \R$ be a (unital) subring and let $\basel$ be a free quadratic $\basek$-algebra. The term \emph{quadratic} indicates that the underlying free $\basek$-module is of rank $2.$ 
Then $\basel=\basek[x]/(p(x)),$ where $p(x)=x^2-c_1x+c_2$ is some monic polynomial over $\basek,$ and $(p(x))$ denotes the ideal of the polynomial ring $\basek[x]$ generated by $p(x).$ We assume that $\basel$ is \emph{separable} over $\basek.$ Equivalently, the discriminant $D={c_1}^2 -4c_2$ is invertible in $\basek$ (cf. \cite[Appendix~C]{Hahn}). The trivial quadratic $\basek$-algebra $\basek \times \basek$ can be obtained by taking $p(x)=x^2-1.$ We will refer to this case as \emph{split case}. We will refer to the case where $\basel$ is not isomorphic to $\basek \times \basek$ as \emph{non-split case}. When $\basel$ is non-split, we further assume $D > 0$ so that $\basel$ can be identified with a subring of $\R.$ 
It is known that $p(x)$ splits over $\basel.$ We denote the two distinct roots of $p(x)$ in $\basel$ by $\tau$ and $\sigma,$ so that $p(x)=(x-\tau)(x-\sigma)$ over $\basel.$ In the split case, we choose $\tau=[1, -1]$ and $\sigma=[-1, 1]$ in the basis $\{[1, 0], [0, 1]\}$ of $\basek \times \basek.$ In the non-split case, we will assume $\tau > \sigma.$ In what follows, we fix a $\basek$-basis of $\basel$ as $\{1, \tau\}.$

\subsection{Construction} \label{ssec:foldingcon}
The central step in the construction of a twisted quadratic folding is to define a \emph{folded $\basek$-representation} of a given root datum $\Root \hra \dual{\Lattice}$ (Definition~\ref{def:foldedrep}). We first recall two key notions involved in the definition, namely, the \emph{$\tau$-form} and the \emph{$\tau$-operator}.

Let $\ku$ be a free $\basel$-module of rank $n$ equipped with a basis and the dot product $(- \cdot -).$ Let $U$ denote the free $\basek$-module of rank $2n$ obtained from $\ku$ by restriction of scalars from $\basel$ to $\basek.$ For $u=\langle x_i, y_i \rangle_{i=1}^n \in U$ $(x_i, y_i \in \basek)$ with respect to the $\{1, \tau\}$-basis, we write $u_{\basel}:=\langle x_i+y_i\tau \rangle_{i=1}^n \in \ku.$ Let $\pr_{\tau}: \basel \ra \basek$ be the $\basek$-linear projection given by $\pr_{\tau}(x + y\tau) = x$ where $x, y \in \basek.$
Then the \emph{$\tau$-form} on $U$ is a symmetric $\basek$-bilinear form $\tauform{-}{-}$ given by
\begin{equation} \label{eq:tauform}
\tauform{u}{u'} = \pr_{\tau}(u_{\basel} \cdot u'_{\basel})
\end{equation}
for $u, u' \in U.$ The \emph{$\tau$-operator} on $U$ is the $\basek$-linear automorphism $\tauo: U \ra U$ defined by the property $(\tauo u)_{\basel} = \tau \cdot u_{\basel},$ 
where on the right-hand side, the element $u_{\basel} \in \ku$ is multiplied by the scalar $\tau.$ The $\tau$-operator $\tauo$ is adjoint with respect to the $\tau$-form (\cite[Lemma~2.8]{LZ18}).
A nonzero element $u \in U$ is called \emph{$\tau$-rational} if $(u_{\basel} \cdot u_{\basel}) \in \basek.$ 

\begin{lemma}[{\cite[Lemma~2.10]{LZ18}}] \label{taurational}
Let $u \in U$ be $\tau$-rational. Then the following hold:
\begin{enumerate}[\normalfont (a)]
\item $\tauform{u}{\kt u} = 0,$
\item $\tauform{\kt u}{\kt u} = -c_2 \tauform{u}{u} = -c_2(u_{\basel} \cdot u_{\basel}),$
\item If $(u_{\basel} \cdot u_{\basel}) \neq 0,$ then $u \neq \kt u.$
\end{enumerate}
\end{lemma}

\begin{definition}[Folded representation, {\cite[Definition~3.1]{LZ18}}] \label{def:foldedrep}
Let $\Root \hra \dual{\Lattice}$ be a root datum with a simple system $\Simple \subset \Root$ specified.
Let $U$ be a free $\basek$-module of some even rank $2n$ equipped with a basis and the dot product $\blform{-}{-}.$ 
A geometric $\basek$-representation $(\Root, U)_{\basek}$ of the root datum $\Root \hra \dual{\Lattice}$ (Definition \ref{def:georep}) is called a \emph{folded $\basek$-representation} if the following conditions are satisfied:
\begin{enumerate}[(F1)]
\item There exists a free quadratic separable $\basek$-algebra $\basel$ and a free $\basel$-module $\ku$ of rank $n$ such that $U$ is obtained from $\ku$ by restriction of scalars to $\basek.$ [Then fix the two roots $\tau, \sigma \in \basel,$ the $\tau$-operator $\tauo$, and the $\tau$-form $\tauform{-}{-}$ on $U$.]
\item The $\tau$-form $\tauform{-}{-}$ coincides with the dot product $\blform{-}{-}$ on $U$.
\item $\tauo (\Lattice) = \Lattice$ and $\tauo (\RLattice) = \RLattice$. 
\item The simple system $\Simple$ admits a partition $\Simple = \Simple^{\tauo} \sqcup \Simple_{rat} \sqcup \tauo (\Simple_{rat}),$
where $\Simple^{\tauo} := \{ \alpha \in \Delta \mid \kt \alpha = \alpha\}$ ($\tauo$-invariant simple roots) and 
$\Simple_{rat}$ consists of simple roots that are $\tau$-rational.
\end{enumerate}
\end{definition}
\noindent
Here are some remarks. First, (F2) is equivalent to $c_2=-1$ (\cite[Lemma~2.3]{LZ18}). In particular, the defining equation $\tau(\tau-c_1)=1$ of $\tau$ shows that $\tau$ is invertible in $\basel.$ Second, in the split case ($c_1=0$), all simple roots that are not fixed under $\tauo$ are $\tau$-rational. Hence we divide these simple roots into $\Simple_{rat} \sqcup \tauo(\Simple_{rat})$ by making a choice. Finally, since $\tauo$ is an automorphism of $U,$ $\Simple_{rat} \cap \tauo(\Simple_{rat}) = \emptyset$ implies $|\Simple_{rat}| = |\tauo(\Simple_{rat})|.$ In the non-split case, it is known that $\Simple^{\tauo}=\emptyset.$

We now consider the eigenspaces of $\tauo.$ In $U_{\basel} := U \otimes_{\basek} \basel,$ the extended $\tau$-operator $\tauo_{\basel}$ admits two eigenspaces $\Utau$ and $\Usig$ corresponding to eigenvalues $\tau$ and $\sigma$ respectively. Moreover, the decomposition $U_{\basel}=\Utau \oplus \Usig$ is orthogonal with respect to $\tauform{-}{-}$ on $U_{\basel}$ (\cite[Lemma~2.9]{LZ18}). We compose the canonical inclusion $U \hra U_{\basel}$ with the projection $U_{\basel} \ra \Utau$ to obtain the map $\tauproj: U \ra \Utau.$ For $u \in U,$ we have the formula:\\
\begin{equation*} 
\tauproj(u) = \tfrac{1}{\tau - \sigma}(\tauo -\sigma I)(u),
\end{equation*}
and $\tauproj$ is known to be bijective (\cite[after Lemma~2.9]{LZ18}). Given $u, u' \in U,$ we also have 
\begin{equation} \label{eq:tauformproj}
\tauform{\tauproj(u)}{\tauproj(u')} = \tfrac{\sigma^2 -c_2}{D} (u_{\basel} \cdot u'_{\basel}).
\end{equation}

We set $\Simple_{\tau}:= \tauproj(\Simple^{\tauo} \sqcup \Simple_{rat})$ and for each $\alpha \in \Simple \sqcup \Simple_{rat},$ we write $\lbar{\alpha} := \tauproj(\alpha).$ Given $\lbar{\alpha} \in \Simple_{\tau},$ the following $\basel$-linear map $\Reflection{\lbar{\alpha}} \in \Aut_{\basel}(\Utau)$ is known to behave as a reflection with respect to $\lbar{\alpha}:$ if $\alpha \in \Simple^{\tauo},$ define $\Reflection{\lbar{\alpha}}$ to be the reflection $\reflection{\alpha} \in \Aut_{\basek}(U)$ extended to $U_{\basel}$, then restricted to $\Utau$; if $\alpha \in \Simple_{rat},$ define $\Reflection{\lbar{\alpha}}$ to be the composite $\reflection{\alpha}\reflection{\tauo\alpha} \in \Aut_{\basek}(U)$ extended to $U_{\basel},$ then restricted to $\Utau.$ This procedure is well-defined and the map $\Reflection{\lbar{\alpha}}$ commutes with $\tauo_{\basel}|_{\Utau}.$ We define $\TWeyl:=\langle \Reflection{\lbar{\alpha}} \mid \alpha \in \Simple^{\tauo} \sqcup \Simple_{rat} \rangle$ and $\Root_{\tau}:= \{u(\lbar{\alpha}) \mid \lbar{\alpha} \in \Simple_{\tau}, u \in \TWeyl\}.$ Then $\TWeyl$ is a finite $\basel$-reflection group with root system $\Root_{\tau}.$ The set $\Simple_{\tau}$ is a simple system of $\Root_{\tau}.$ The passage from $(\Root, U)_{\basek}$ to $\Root_{\tau} \subset \Utau$ is called a \emph{twisted quadratic folding}. It is noted that $\Root_{\tau}$ is finite but not necessarily crystallographic. Finally, the following assignment of generators extends to an embedding of Coxeter groups:
\begin{align*} 
\embedding: \,\, \TWeyl \hra \Weyl : \,\, &\Reflection{\lbar{\alpha}} \mapsto \reflection{\alpha} \quad &\text{if} \,\, \alpha \in \Simple^{\tauo}, \\
& \Reflection{\lbar{\alpha}} \mapsto \reflection{\alpha}\reflection{\tauo\alpha} &\text{if} \,\, \alpha \in \Simple_{rat}. 
\end{align*}
All of this is established in \cite[Sections~3--4]{LZ18}. By direct computation, we may check
\begin{lemma} \label{prop:embeddingproperty}
$u\big(\tauproj(x)\big) = \tauproj\big(\embedding(u)(x)\big)$ for all $u \in \TWeyl,$ $x \in U.$
\end{lemma}

Given a (possibly empty) subset $\Simple_P \subset \Simple$ such that $\tauo\big(\spn_{\basek}\Simple_P\big) = \spn_{\basek}\Simple_P,$ we can conduct the procedure we have described so far for the subsystem $\Root_P \subset U$ and its reflection group $\Weyl_P=\langle \reflection{\alpha} \mid \alpha \in \Simple_P \rangle,$ and obtain the folded root system $(\Root_P)_{\tau} \subset \Utau$ and associated reflection group $(\Weyl_P)_{\tau}.$ By \cite[Lemma~5.2]{LZ18} and Proposition~\ref{prop:orbit}, we have

\begin{lemma} \label{prop:restriction}
The embedding $\embedding: \TWeyl \hra \Weyl$ restricts to $\TWeyl^P \hra \Weyl^P,$ where $\TWeyl^P$ denotes the set of minimal length representatives of the coset space $\TWeyl/(\Weyl_P)_{\tau}.$
\end{lemma}

\subsection{Induced folding map}

Let $\Root \rightsquigarrow \Root_{\tau}$ be a twisted quadratic folding and $\Root_P \rightsquigarrow (\Root_P)_{\tau}$ be its subfolding as discussed above. We keep the notation of Section~\ref{ssec:foldingcon}.
Let $\tpmoment$ be the moment graph associated with the pair $(\Root_{\tau}, (\Root_P)_{\tau}).$
Recall that by Definition \ref{def:georep}, $U$ contains $U_{\Lattice}:=\Lattice \otimes_{\Z} \basek$ as a free $\basek$-submodule. Let $\sym_{\tau} := \Sym_{\basel} \tauproj (U_{\Lattice}).$ We replace $\pmoment$ by $\tpmoment$ and $\sym$ by $\sym_{\tau}$ in Definition \ref{def:sa} to obtain the folded structure algebra $\sa{\tpmoment}.$ The map $\tauproj|_{U_{\Lattice}}: U_{\Lattice} \ra \tauproj(U_{\Lattice})$ naturally extends to the ring homomorphism $\tauproj: \sym \ra \sym_{\tau}$ (recall that $\sym=\Sym_{\basek}U_{\Lattice}$). 

\begin{proposition}[{\cite[Theorem~6.2]{LZ18}}] 
There is an induced ring homomorphism $\embedding^*: \sa{\pmoment} \ra \sa{\tpmoment},$ called the \emph{folding map}, given by 
\begin{equation*}
\embedding^*(\xi)(u) = \tauproj\Big(\xi\big(\embedding(u)\big)\Big)
\end{equation*}
for all $\xi \in \sa{\pmoment}$ and $u \in \TWeyl^P.$ Moreover, it commutes with the $\sym$-action on $\sa{\pmoment},$ that is, $\embedding^*(f  \cdot \xi) = \tauproj(f) \cdot \big(\embedding^*(\xi)\big)$ for $\xi \in \sa{\pmoment}$ and $f \in \sym.$
\end{proposition}

\begin{remark} \label{rmk:tauproj} 
Since the map $\tauproj: U \ra \Utau$ is bijective, in particular, its restriction  $\tauproj|_{U_{\Lattice}}: U_{\Lattice} \ra \tauproj(U_{\Lattice})$ is injective, and so is the induced map $\tauproj: \sym \ra \sym_{\tau}.$
\end{remark}

\subsection{Local analysis of a folding}
The following result about the local behaviour of the map $\tauproj|_{\Simple}: \Simple \ra \Simple_{\tau}$ is used to establish combinatorial properties of the folding map $\embedding^*$ in Section~\ref{sec:comfolding} (Lemma~\ref{prop:philocal}).

\begin{lemma} \label{prop:foldinglocal1}
Let $\Root \rightsquigarrow \Root_{\tau}$ be a twisted quadratic folding of $\Root$ simply-laced.
Let $\alpha, \beta \in \Simple^{\tauo} \sqcup \Simple_{rat}$ such that $m := m(\Reflection{\lbar{\alpha}}, \Reflection{\lbar{\beta}}) \geq 3$. The possible configuration of the full subgraph $(\tauproj|_{\Simple})\inv(\{\lbar{\alpha}, \lbar{\beta}\})$ of the Coxeter diagram of $\Weyl$ is one of the following up to symmetry of $\alpha$ and $\beta$:
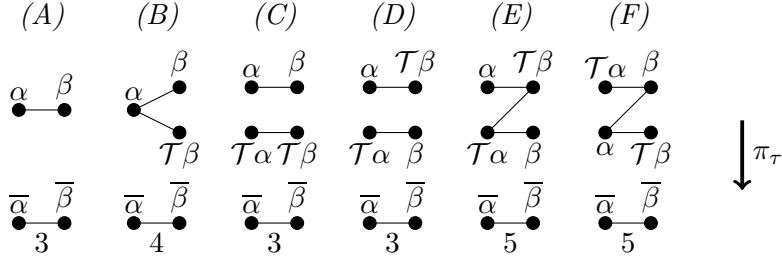
\begin{center}
\begin{tikzpicture}[scale=0.6]
\draw (2.5,1) node[anchor=south] {(A)};
\filldraw[black] (2,-0.5) circle (4pt) node[anchor=south] {$\alpha$};
\filldraw[black] (3,-0.5) circle (4pt) node[anchor=south] {$\beta$};

\draw (2,-0.5) -- (3,-0.5);

\filldraw[black] (2,-3) circle (4pt) node[anchor=south] {$\lbar{\alpha}$};
\filldraw[black] (3,-3) circle (4pt) node[anchor=south] {$\lbar{\beta}$};

\draw (2,-3) -- (3,-3);
\draw (2.5,-3) node[anchor=north] {$3$};
\end{tikzpicture}
\,\,
\begin{tikzpicture}[scale=0.6]
\draw (1.5,1) node[anchor=south] {(B)};
\filldraw (1,-0.5) circle (4pt) node[anchor=south] {$\alpha$};
\filldraw[black] (2,0) circle (4pt) node[anchor=south] {$\beta$};
\filldraw[black] (2,-1) circle (4pt) node[anchor=north] {$\tauo\beta$};

\draw (1,-0.5) -- (2,0);
\draw (1,-0.5) -- (2,-1);

\filldraw[black] (1,-3) circle (4pt) node[anchor=south] {$\lbar{\alpha}$};
\filldraw[black] (2,-3) circle (4pt) node[anchor=south] {$\lbar{\beta}$};

\draw (1,-3) -- (2,-3);

\draw (1.5,-3) node[anchor=north] {$4$};
\end{tikzpicture}
\begin{tikzpicture}[scale=0.6]
\draw (2.5,1) node[anchor=south] {(C)};
\filldraw[black] (2,0) circle (4pt) node[anchor=south] {$\alpha$};
\filldraw[black] (3,0) circle (4pt) node[anchor=south] {$\beta$};
\filldraw[black] (3,-1) circle (4pt) node[anchor=north] {$\tauo\beta$};
\filldraw[black] (2,-1) circle (4pt) node[anchor=north] {$\tauo \alpha$};

\draw (2,0) -- (3,0);
\draw (3,-1) -- (2, -1);

\filldraw[black] (2,-3) circle (4pt) node[anchor=south] {$\lbar{\alpha}$};
\filldraw[black] (3,-3) circle (4pt) node[anchor=south] {$\lbar{\beta}$};

\draw (2,-3) -- (3,-3);
\draw (2.5,-3) node[anchor=north] {$3$};
\end{tikzpicture}
\begin{tikzpicture}[scale=0.6]
\draw (2.5,1) node[anchor=south] {(D)};
\filldraw[black] (2,0) circle (4pt) node[anchor=south] {$\alpha$};
\filldraw[black] (3,0) circle (4pt) node[anchor=south] {$\tauo \beta$};
\filldraw[black] (3,-1) circle (4pt) node[anchor=north] {$\beta$};
\filldraw[black] (2,-1) circle (4pt) node[anchor=north] {$\tauo \alpha$};

\draw (2,0) -- (3,0);
\draw (3,-1) -- (2, -1);

\filldraw[black] (2,-3) circle (4pt) node[anchor=south] {$\lbar{\alpha}$};
\filldraw[black] (3,-3) circle (4pt) node[anchor=south] {$\lbar{\beta}$};

\draw (2,-3) -- (3,-3);
\draw (2.5,-3) node[anchor=north] {$3$};
\end{tikzpicture}
\begin{tikzpicture}[scale=0.6]
\draw (2.5,1) node[anchor=south] {(E)};
\filldraw[black] (2,0) circle (4pt) node[anchor=south] {$\alpha$};
\filldraw[black] (3,0) circle (4pt) node[anchor=south] {$\tauo \beta$};
\filldraw[black] (3,-1) circle (4pt) node[anchor=north] {$\beta$};
\filldraw[black] (2,-1) circle (4pt) node[anchor=north] {$\tauo \alpha$};

\draw (2,0) -- (3,0);
\draw (3,0) -- (2,-1);
\draw (3,-1) -- (2, -1);

\filldraw[black] (2,-3) circle (4pt) node[anchor=south] {$\lbar{\alpha}$};
\filldraw[black] (3,-3) circle (4pt) node[anchor=south] {$\lbar{\beta}$};

\draw (2,-3) -- (3,-3);
\draw (2.5,-3) node[anchor=north] {$5$};
\end{tikzpicture}
\begin{tikzpicture}[scale=0.6]
\draw (2.5,1) node[anchor=south] {(F)};
\filldraw[black] (2,0) circle (4pt) node[anchor=south] {$\tauo\alpha$};
\filldraw[black] (3,0) circle (4pt) node[anchor=south] {$\beta$};
\filldraw[black] (3,-1) circle (4pt) node[anchor=north] {$\tauo\beta$};
\filldraw[black] (2,-1) circle (4pt) node[anchor=north] {$\alpha$};

\draw (2,0) -- (3,0);
\draw (3,0) -- (2,-1);
\draw (3,-1) -- (2, -1);

\filldraw[black] (2,-3) circle (4pt) node[anchor=south] {$\lbar{\alpha}$};
\filldraw[black] (3,-3) circle (4pt) node[anchor=south] {$\lbar{\beta}$};

\draw (2,-3) -- (3,-3);
\draw (2.5,-3) node[anchor=north] {$5$};

\node (A) at (5, -0.5) {};
\node (B) at (5, -2.5) {};
\draw [->, line width=0.5mm] (A) -- (B);
\draw (5,-1.5) node[anchor=west] {$\tauproj$};
\end{tikzpicture}
\end{center}
\end{lemma}

\begin{remark}
Consider the $\basek$-linear projection $p: \basel \ra \basek$ given by $\tau \mapsto 1$ and its induced map $p: U_{\basel} \ra U.$ In the following argument, we use the substitution $\tau=1$ in the split case. This correspond to replacing the $\basel$-linear eigenspace decomposition $U_{\basel}=\Utau \oplus \Usig$ of $\tauo_{\basel}$ by the $\basek$-linear eigenspace decomposition $U = p(\Utau) \oplus p(\Usig)$ of $\tauo.$ See \cite[Remark~2.13]{LZ18}.
\end{remark}

\begin{proof}
We may assume, without loss of generality, $\tauform{\alpha}{\alpha}=1$ for all $\alpha \in \Simple.$
We write $\alpha_{\basel} \cdot \beta_{\basel} = a +b\tau$ for some $a, b \in \basek.$ Let us recall the following formulas from \eqref{eq:tauform} and \eqref{eq:tauformproj}:
\begin{equation*}
\tauform{\alpha}{\beta} = \pr_{\tau}(\alpha_{\basel} \cdot \beta_{\basel}), \quad
\tauform{\lbar{\alpha}}{\lbar{\beta}} = \tfrac{\sigma^2 +1}{D}(\alpha_{\basel} \cdot \beta_{\basel}).
\end{equation*}
Since $m \geq 3,$ we have $\tauform{\lbar{\alpha}}{\lbar{\beta}} \neq 0,$ and hence $\alpha_{\basel} \cdot \beta_{\basel} \neq 0.$ Hence either $a \neq 0$ or $b \neq 0$ (or both). We know $\tauform{\alpha}{\tauo \alpha}=0,$ $\tauform{\beta}{\tauo\beta}=0$ (Lemma~\ref{taurational} (a)), and $\tauform{\alpha}{\tauo\beta}=\tauform{\tauo\alpha}{\beta}$ (the adjoint property of $\tauo$). We compute the remaining combinations below:
\begin{align*}
\tauform{\alpha}{\beta} &= a, \\
\tauform{\alpha}{\tauo\beta} &= \tauform{\tauo\alpha}{\beta}=\pr_{\tau}(\tau(a+b\tau))=\pr_{\tau}(b+(a+c_1b)\tau)=b, \\
\tauform{\tauo\alpha}{\tauo\beta}&=\pr_{\tau}(\tau^2(a+b\tau)) = \pr_{\tau}(a+c_1b + (c_1^2b+c_1a+b)\tau) = a+c_1b.
\end{align*}

\noindent
Since the original root system $\Root$ is finite crystallographic, its Coxeter diagram contains no circuit. Hence at least one of the above three values must be zero. Since $\Root$ is taken to be simply laced, $a \neq 0$ (resp. $b \neq 0$) implies $a =\tauform{\alpha}{\beta}= -\tfrac{1}{2}$ (resp. $b=\tauform{\alpha}{\tauo\beta}=-\tfrac{1}{2}$).
We also have
\begin{align*}
\alpha_{\basel} \cdot \alpha_{\basel} =
\begin{cases}
1 + \tau & \text{if} \,\, \alpha \in \Simple^{\tauo},\\
1 & \text{if} \,\, \alpha \in \Simple_{rat}. 
\end{cases}
\end{align*}
For $\alpha \in \Simple_{rat},$ this is a consequence of Lemma \ref{taurational} (b). For $\alpha \in \Simple^{\tauo},$  write $\alpha_{\basel} \cdot \alpha_{\basel} = c+d\tau$ for some $c, d \in \basek.$ Then combine $c=\tauform{\alpha}{\alpha}=\tauform{\alpha}{\tauo\alpha}=d$ and $\tauform{\alpha}{\alpha}=1.$

We denote the angle between $\lbar{\alpha}$ and $\lbar{\beta}$ by $\theta.$\\[.1cm]
\noindent
\textbf{Case~1:} $c_1 = 0$ (split case).
We have $\tauform{\alpha}{\beta}=\tauform{\tauo\alpha}{\tauo\beta}=a.$ Suppose $\alpha = \tauo\alpha.$ Then we also have $b = \tauform{\tauo\alpha}{\beta}=\tauform{\alpha}{\beta}=a,$ and hence $a \neq 0$ and $b \neq 0.$ We have
\begin{align*}
\tauform{\lbar{\alpha}}{\lbar{\beta}} &= \tfrac{1}{2}(-\tfrac{1}{2}-\tfrac{1}{2}\tau),\,\,
\tauform{\lbar{\alpha}}{\lbar{\alpha}} = \tfrac{1}{2}(1 + \tau),\\
\tauform{\lbar{\beta}}{\lbar{\beta}} &= 
\begin{cases}
\tfrac{1}{2}(1+\tau) & \text{if} \,\, \beta \in \Simple^{\tauo},\\
\tfrac{1}{2} & \text{if} \,\, \beta \in \Simple_{rat}.
\end{cases}
\end{align*}
Hence if $\beta \in \Simple^{\tauo},$ we obtain $-\cos \theta= -\tfrac{1}{2},$ yielding the diagram (A).
If $\beta \in \Simple_{rat},$ we obtain $-\cos \theta=-\tfrac{1}{\sqrt{2}},$ yielding the diagram (B).

Now suppose $\alpha, \beta \notin \Simple^{\tauo}.$ If $a=0,$ then $b \neq 0.$ We have
\begin{equation*}
\tauform{\lbar{\alpha}}{\lbar{\beta}} = \tfrac{1}{2}(-\tfrac{1}{2}\tau),\,\, 
\tauform{\lbar{\alpha}}{\lbar{\alpha}} = \tauform{\lbar{\beta}}{\lbar{\beta}}= \tfrac{1}{2}.
\end{equation*}
Hence $-\cos \theta= -\tfrac{1}{2},$ yielding the diagram (D).
If $b=0,$ then $a \neq 0.$ We have
\begin{equation*}
\tauform{\lbar{\alpha}}{\lbar{\beta}} =  \tfrac{1}{2}(-\tfrac{1}{2}),\,\,
\tauform{\lbar{\alpha}}{\lbar{\alpha}} = \tauform{\lbar{\beta}}{\lbar{\beta}}= \tfrac{1}{2}.
\end{equation*}
Hence $-\cos \theta = -\tfrac{1}{2},$ yielding the diagram (C). 
\\[.1cm]
\textbf{Case~2:} $c_1 \neq 0$ (non-split case). 
If $a =0,$ then $b \neq 0.$ The equality $\tauform{\tauo\alpha}{\tauo\beta}=-\tfrac{1}{2}c_1=-\tfrac{1}{2}$ enforces $c_1=1,$ and $\tau=\tfrac{1+\sqrt{5}}{2}$ (the golden ratio). We have
\begin{equation*}
\tauform{\lbar{\alpha}}{\lbar{\beta}} = \tfrac{\sigma^2+1}{5}(-\tfrac{1}{2}\tau),\,\,
\tauform{\lbar{\alpha}}{\lbar{\alpha}} = \tauform{\lbar{\beta}}{\lbar{\beta}} =  \tfrac{\sigma^2+1}{5}.
\end{equation*}
Hence $-\cos \theta= \tfrac{-1-\sqrt{5}}{4},$ yielding the diagram (E).  
If $b=0,$ then $a \neq 0.$ We have
\begin{equation*}
\tauform{\lbar{\alpha}}{\lbar{\beta}} =  \tfrac{\sigma^2+1}{5}(-\tfrac{1}{2}),\,\,
\tauform{\lbar{\alpha}}{\lbar{\alpha}} = \tauform{\lbar{\beta}}{\lbar{\beta}} =  \tfrac{\sigma^2+1}{5}.
\end{equation*}
Hence $-\cos \theta= -\tfrac{1}{2},$ yielding the diagram (C). 
If $a + c_1b =0,$ then both $a \neq 0$ and $b \neq 0.$ The equality $-\tfrac{1}{2}-\tfrac{1}{2}c_1=0$ enforces $c_1=-1,$ and $\tau=\tfrac{-1+\sqrt{5}}{2}.$ We have
\begin{equation*}
\tauform{\lbar{\alpha}}{\lbar{\beta}} = \tfrac{\sigma^2+1}{5}(-\tfrac{1}{2}-\tfrac{1}{2}\tau),\,\,
\tauform{\lbar{\alpha}}{\lbar{\alpha}} = \tauform{\lbar{\beta}}{\lbar{\beta}} = \tfrac{\sigma^2+1}{5}.
\end{equation*}
Hence $-\cos \theta = \tfrac{-1-\sqrt{5}}{4},$ yielding the diagram (F).
\end{proof}

\subsection{Folding diagrams} \label{app:folding}

\setlength{\intextsep}{5pt plus 1.0pt minus 2.0pt}

This subsection collects the Coxeter diagrams of all known examples of twisted quadratic foldings (cf. \cite[Examples~4.3--4.5]{LZ18}). The number at each vertex is the index of each simple reflection that the vertex represents. In the Coxeter diagram associated with the original root system (drawn first in each folding), the simple reflections represented by black dots correspond to the elements of $\Simple_{rat},$ those represented by white dots correspond to the elements of $\tauo(\Simple_{rat}),$ and those represented by double white dots correspond to the elements of $\Simple^{\tauo}$ respectively.

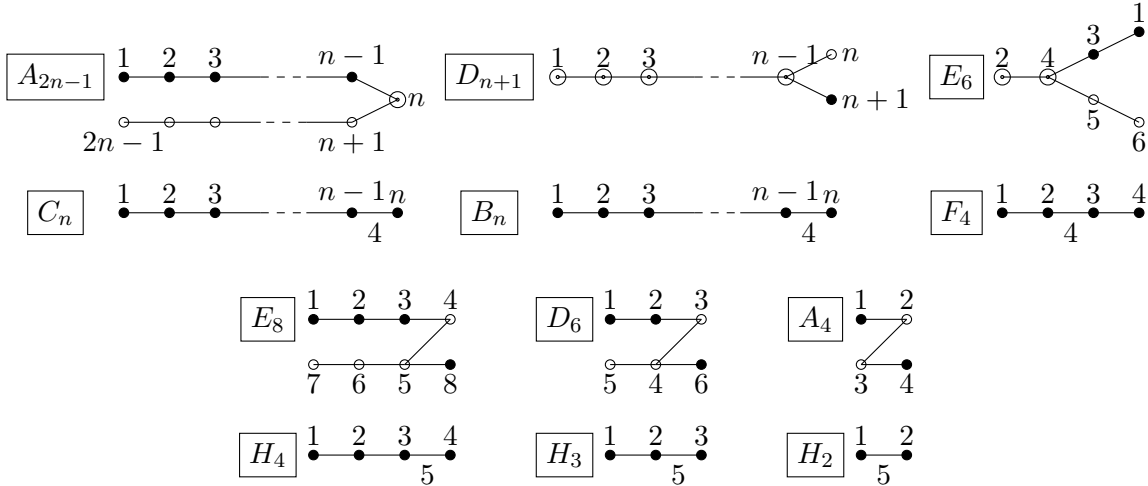
\begin{center}
\begin{tikzpicture}[scale=0.6]
\node[draw] at (-1.5,0) {$A_{2n-1}$};
\filldraw[black] (0,0) circle (3pt) node[anchor=south] {$1$};
\filldraw[black] (1,0) circle (3pt) node[anchor=south] {$2$};
\filldraw[black] (2,0) circle (3pt) node[anchor=south] {$3$};
\filldraw[black] (5,0) circle (3pt) node[anchor=south] {$n-1$};
\draw [double, double distance=2pt] (6,-0.5) circle (3pt) node[anchor=west] {$n$};

\draw[black] (5,-1) circle (3pt) node[anchor=north] {$n+1$};
\draw[black] (2,-1) circle (3pt) node[anchor=north] {};
\draw[black] (1,-1) circle (3pt) node[anchor=north] {};
\draw[black] (0,-1) circle (3pt) node[anchor=north] {$2n-1$};
\draw (0,0) -- (1,0);
\draw (1,0) -- (2,0);
\draw (2,0) -- (3,0);
\draw[dashed] (3,0) -- (4,0);
\draw (4,0) -- (5,0);
\draw (5,0) -- (6, -0.5);
\draw (6,-0.5) -- (5, -1);
\draw (5,-1) -- (4, -1);
\draw[dashed] (4,-1) -- (3,-1);
\draw (3,-1) -- (2, -1);
\draw (2,-1) -- (1, -1);
\draw (1,-1) -- (0, -1);

\node[draw] at (-1.5,-3) {$C_n$};
\filldraw[black] (0,-3) circle (3pt) node[anchor=south] {$1$};
\filldraw[black] (1,-3) circle (3pt) node[anchor=south] {$2$};
\filldraw[black] (2,-3) circle (3pt) node[anchor=south] {$3$};
\filldraw[black] (5,-3) circle (3pt) node[anchor=south] {$n-1$};
\filldraw[black] (6,-3) circle (3pt) node[anchor=south] {$n$};
\draw (0,-3) -- (1,-3);
\draw (1,-3) -- (2,-3);
\draw (2,-3) -- (3,-3);
\draw[dashed] (3,-3) -- (4,-3);
\draw (4,-3) -- (5,-3);
\draw (5,-3) -- (6,-3);
\draw (5.5,-3) node[anchor=north] {$4$};
\end{tikzpicture}
\begin{tikzpicture}[scale=0.6]
\node[draw] at (-1.5,0) {$D_{n+1}$};
\draw [double, double distance=2pt] (0,0) circle (3pt) node[anchor=south] {$1$};
\draw [double, double distance=2pt] (1,0) circle (3pt) node[anchor=south] {$2$};
\draw [double, double distance=2pt] (2,0) circle (3pt) node[anchor=south] {$3$};
\draw [double, double distance=2pt] (5,0) circle (3pt) node[anchor=south] {$n-1$};
\filldraw[black] (6,-0.5) circle (3pt) node[anchor=west] {$n+1$};
\draw[black] (6,0.5) circle (3pt) node[anchor=west] {$n$};

\draw (0,0) -- (1,0);
\draw (1,0) -- (2,0);
\draw (2,0) -- (3,0);
\draw[dashed] (3,0) -- (4,0);
\draw (4,0) -- (5,0);
\draw (5,0) -- (6, -0.5);
\draw (5,0) -- (6,0.5);

\node[draw] at (-1.5,-3) {$B_n$};
\filldraw[black] (0,-3) circle (3pt) node[anchor=south] {$1$};
\filldraw[black] (1,-3) circle (3pt) node[anchor=south] {$2$};
\filldraw[black] (2,-3) circle (3pt) node[anchor=south] {$3$};
\filldraw[black] (5,-3) circle (3pt) node[anchor=south] {$n-1$};
\filldraw[black] (6,-3) circle (3pt) node[anchor=south] {$n$};
\draw (0,-3) -- (1,-3);
\draw (1,-3) -- (2,-3);
\draw (2,-3) -- (3,-3);
\draw[dashed] (3,-3) -- (4,-3);
\draw (4,-3) -- (5,-3);
\draw (5,-3) -- (6,-3);
\draw (5.5,-3) node[anchor=north] {$4$};
\end{tikzpicture}
\begin{tikzpicture}[scale=0.6]
\node[draw] at (-1,0)  {$E_6$};
\draw [double, double distance=2pt] (0,0) circle (3pt) node[anchor=south] {$2$};
\draw [double, double distance=2pt] (1,0) circle (3pt) node[anchor=south] {$4$};

\filldraw[black] (2,0.5) circle (3pt) node[anchor=south] {$3$};
\filldraw[black] (3,1) circle (3pt) node[anchor=south] {$1$};
\draw[black] (2,-0.5) circle (3pt) node[anchor=north] {$5$};
\draw[black] (3,-1) circle (3pt) node[anchor=north] {$6$};

\draw (0,0) -- (1,0);
\draw (1,0) -- (2,0.5);
\draw (2,0.5) -- (3,1);
\draw (1,0) -- (2,-0.5);
\draw (2,-0.5) -- (3, -1);

\node[draw] at (-1,-3) {$F_4$};
\filldraw[black] (0,-3) circle (3pt) node[anchor=south] {$1$};
\filldraw[black] (1,-3) circle (3pt) node[anchor=south] {$2$};
\filldraw[black] (2,-3) circle (3pt) node[anchor=south] {$3$};
\filldraw[black] (3,-3) circle (3pt) node[anchor=south] {$4$};
\draw (0,-3) -- (1,-3);
\draw (1,-3) -- (2,-3);
\draw (2,-3) -- (3,-3);

\draw (1.5,-3) node[anchor=north] {$4$};
\end{tikzpicture}
\end{center}
\begin{center}
\begin{tikzpicture}[scale=0.6]
\node[draw] at (-1,0) {$E_8$};
\filldraw[black] (0,0) circle (3pt) node[anchor=south] {$1$};
\filldraw[black] (1,0) circle (3pt) node[anchor=south] {$2$};
\filldraw[black] (2,0) circle (3pt) node[anchor=south] {$3$};
\draw[black] (3,0) circle (3pt) node[anchor=south] {$4$};

\filldraw[black] (3,-1) circle (3pt) node[anchor=north] {$8$};
\draw[black] (2,-1) circle (3pt) node[anchor=north] {$5$};
\draw[black] (1,-1) circle (3pt) node[anchor=north] {$6$};
\draw[black] (0,-1) circle (3pt) node[anchor=north] {$7$};

\draw (0,0) -- (1,0);
\draw (1,0) -- (2,0);
\draw (2,0) -- (3,0);
\draw (3,0) -- (2,-1);

\draw (3,-1) -- (2, -1);
\draw (2,-1) -- (1, -1);
\draw (1,-1) -- (0, -1);

\node[draw] at (-1,-3) {$H_4$};
\filldraw[black] (0,-3) circle (3pt) node[anchor=south] {$1$};
\filldraw[black] (1,-3) circle (3pt) node[anchor=south] {$2$};
\filldraw[black] (2,-3) circle (3pt) node[anchor=south] {$3$};
\filldraw[black] (3,-3) circle (3pt) node[anchor=south] {$4$};
\draw (0,-3) -- (1,-3);
\draw (1,-3) -- (2,-3);
\draw (2,-3) -- (3,-3);

\draw (2.5,-3) node[anchor=north] {$5$};
\end{tikzpicture}
\quad \quad
\begin{tikzpicture}[scale=0.6]
\node[draw] at (0,0) {$D_6$};
\filldraw[black] (1,0) circle (3pt) node[anchor=south] {$1$};
\filldraw[black] (2,0) circle (3pt) node[anchor=south] {$2$};
\draw[black] (3,0) circle (3pt) node[anchor=south] {$3$};

\filldraw[black] (3,-1) circle (3pt) node[anchor=north] {$6$};
\draw[black] (2,-1) circle (3pt) node[anchor=north] {$4$};
\draw[black] (1,-1) circle (3pt) node[anchor=north] {$5$};

\draw (1,0) -- (2,0);
\draw (2,0) -- (3,0);
\draw (3,0) -- (2,-1);

\draw (3,-1) -- (2, -1);
\draw (2,-1) -- (1, -1);

\node[draw] at (0,-3) {$H_3$};
\filldraw[black] (1,-3) circle (3pt) node[anchor=south] {$1$};
\filldraw[black] (2,-3) circle (3pt) node[anchor=south] {$2$};
\filldraw[black] (3,-3) circle (3pt) node[anchor=south] {$3$};
\draw (1,-3) -- (2,-3);
\draw (2,-3) -- (3,-3);

\draw (2.5,-3) node[anchor=north] {$5$};
\end{tikzpicture}
\quad \quad
\begin{tikzpicture}[scale=0.6]
\node[draw] at (1,0) {$A_4$};
\filldraw[black] (2,0) circle (3pt) node[anchor=south] {$1$};
\draw[black] (3,0) circle (3pt) node[anchor=south] {$2$};

\filldraw[black] (3,-1) circle (3pt) node[anchor=north] {$4$};
\draw[black] (2,-1) circle (3pt) node[anchor=north] {$3$};

\draw (2,0) -- (3,0);
\draw (3,0) -- (2,-1);

\draw (3,-1) -- (2, -1);

\node[draw] at (1,-3) {$H_2$};
\filldraw[black] (2,-3) circle (3pt) node[anchor=south] {$1$};
\filldraw[black] (3,-3) circle (3pt) node[anchor=south] {$2$};

\draw (2,-3) -- (3,-3);
\draw (2.5,-3) node[anchor=north] {$5$};
\end{tikzpicture}
\end{center}



\section{Combinatorial folding map} \label{sec:comfolding}
Let $\Root \rightsquigarrow \Root_{\tau}$ be a twisted quadratic folding with $\Root$ simply laced. 
Let $(\Weyl, S)$ denote the Coxeter system associated with the $\basek$-reflection group $\Weyl$ of $\Root$ where $S=\{ \reflection{\alpha} \mid \alpha \in \Simple \},$ and let $(\TWeyl,S_{\tau})$ denote the Coxeter system associated with the $\basel$-reflection group $\TWeyl$ where $S_{\tau}=\{\Reflection{\beta} \mid \beta \in \Simple_{\tau}\}.$ We will use $\tlength{u}$ to denote the length of $u \in \TWeyl.$

We study liftings of Schubert classes (Definition~\ref{def:schubertlifting}) by treating the folding $\Root \rightsquigarrow \Root_{\tau}$ on the level of the Coxeter generating sets $S$ and $S_{\tau}.$ We define a set-theoretic map $\adm: S \ra S_{\tau}$ by 
\begin{equation} \label{eq:adm}
\adm: S \ra S_{\tau}: \reflection{\alpha},  \reflection{\tauo \alpha} \mapsto \Reflection{\lbar{\alpha}}
\end{equation}
for each $\alpha \in \Simple^{\tauo} \sqcup \Simple_{rat}.$ 
Note that $\adm$ is surjective due to the surjectivity of $\tauproj|_{\Simple}: \Simple \ra \Simple_{\tau}.$

\begin{definition} \label{def:opposite}
Let  $R \in S_{\tau}$ be such that $\adm\inv(R)$ has exactly two elements. If $s$ is one of them, the other one is denoted by $s^*$ and called the \emph{opposite simple reflection} of $s.$
\end{definition}

The following is the rephrasing of Lemma~\ref{prop:foldinglocal1} in terms of $\adm.$

\begin{lemma} \label{prop:philocal}
Let $R, R' \in S_{\tau}$ such that $m(R, R') \geq 3.$ The possible configuration of the full subgraph $\adm\inv(\{R, R'\})$ of the Coxeter diagram of $(\Weyl, S)$ is one of the following, up to symmetry of $R$ and $R'.$ 
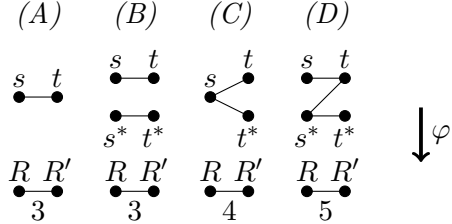
\begin{center}
\begin{tikzpicture}[scale=0.5]
\draw (2.5,1) node[anchor=south] {(A)};
\filldraw[black] (2,-0.5) circle (4pt) node[anchor=south] {$s$};
\filldraw[black] (3,-0.5) circle (4pt) node[anchor=south] {$t$};

\draw (2,-0.5) -- (3,-0.5);

\filldraw[black] (2,-3) circle (4pt) node[anchor=south] {$R$};
\filldraw[black] (3,-3) circle (4pt) node[anchor=south] {$R'$};

\draw (2,-3) -- (3,-3);
\draw (2.5,-3) node[anchor=north] {$3$};
\end{tikzpicture}
\begin{tikzpicture}[scale=0.5]
\draw (2.5,1) node[anchor=south] {(B)};
\filldraw[black] (2,0) circle (4pt) node[anchor=south] {$s$};
\filldraw[black] (3,0) circle (4pt) node[anchor=south] {$t$};
\filldraw[black] (3,-1) circle (4pt) node[anchor=north] {$t^*$};
\filldraw[black] (2,-1) circle (4pt) node[anchor=north] {$s^*$};

\draw (2,0) -- (3,0);
\draw (3,-1) -- (2, -1);

\filldraw[black] (2,-3) circle (4pt) node[anchor=south] {$R$};
\filldraw[black] (3,-3) circle (4pt) node[anchor=south] {$R'$};

\draw (2,-3) -- (3,-3);
\draw (2.5,-3) node[anchor=north] {$3$};
\end{tikzpicture}
\begin{tikzpicture}[scale=0.5]
\draw (1.5,1) node[anchor=south] {(C)};
\filldraw (1,-0.5) circle (4pt) node[anchor=south] {$s$};
\filldraw[black] (2,0) circle (4pt) node[anchor=south] {$t$};
\filldraw[black] (2,-1) circle (4pt) node[anchor=north] {$t^*$};

\draw (1,-0.5) -- (2,0);
\draw (1,-0.5) -- (2,-1);

\filldraw[black] (1,-3) circle (4pt) node[anchor=south] {$R$};
\filldraw[black] (2,-3) circle (4pt) node[anchor=south] {$R'$};

\draw (1,-3) -- (2,-3);

\draw (1.5,-3) node[anchor=north] {$4$};
\end{tikzpicture}
\begin{tikzpicture}[scale=0.5]
\draw (2.5,1) node[anchor=south] {(D)};
\filldraw[black] (2,0) circle (4pt) node[anchor=south] {$s$};
\filldraw[black] (3,0) circle (4pt) node[anchor=south] {$t$};
\filldraw[black] (3,-1) circle (4pt) node[anchor=north] {$t^*$};
\filldraw[black] (2,-1) circle (4pt) node[anchor=north] {$s^*$};

\draw (2,0) -- (3,0);
\draw (3,0) -- (2,-1);
\draw (3,-1) -- (2, -1);

\filldraw[black] (2,-3) circle (4pt) node[anchor=south] {$R$};
\filldraw[black] (3,-3) circle (4pt) node[anchor=south] {$R'$};

\draw (2,-3) -- (3,-3);
\draw (2.5,-3) node[anchor=north] {$5$};

\node (A) at (5, -0.5) {};
\node (B) at (5, -2.5) {};
\draw [->, line width=0.5mm] (A) -- (B);
\draw (5,-1.5) node[anchor=west] {$\adm$};
\end{tikzpicture}
\end{center}
\end{lemma}

\noindent
For a pair $(R, R')$ of simple reflections in $\TWeyl$ such that $m(R, R')=2,$ we have the following.

\begin{lemma} \label{prop:m2}
Let $R, R' \in S_{\tau}$ such that $m(R,R')=2.$ Let $s \in~\adm\inv(R)$ and $t \in \adm\inv(R').$ Then $m(s, t)=2.$
\end{lemma}

\begin{proof}
Let us write $R=\Reflection{\lbar{\alpha}}$ and $R'=\Reflection{\lbar{\beta}}$ for some $\alpha, \beta \in \Simple^{\tauo} \sqcup \Simple_{rat}.$ 
The condition $m(R, R') = 2$ is equivalent to $\tauform{\lbar{\alpha}}{\lbar{\beta}}=0.$ 
By formula \eqref{eq:tauformproj}, $(\alpha_{\basel} \cdot \beta_{\basel}) = 0.$  Hence $\tauform{\alpha}{\beta} = pr_{\tau}(\alpha_{\basel} \cdot \beta_{\basel}) =0,$ $\tauform{\alpha}{\tauo\beta} =\tauform{\tauo\alpha}{\beta} =pr_{\tau}(\tau\alpha_{\basel} \cdot \beta_{\basel}) = 0,$ and $\tauform{\tauo\alpha}{\tauo\beta} = pr_{\tau}(\tau\alpha_{\basel} \cdot \tau\beta_{\basel}) =0.$
Thus for all possible choices of $s \in \adm\inv(\Reflection{\lbar{\alpha}})$ and $t \in \adm\inv(\Reflection{\lbar{\beta}}),$ we have $m(s, t)=2.$  
\end{proof}

\noindent
The map $\adm$ has the following property.

\begin{lemma} \label{prop:phi} 
Let $\reflection{i_1} \cdots \reflection{i_{\ell}}$ be a reduced expression in $\Weyl,$ and let $\Reflection{j_1} \cdots \Reflection{j_k}$ be a reduced expression in $\TWeyl.$ Suppose $\adm(\reflection{i_p}) \neq \adm(\reflection{i_{p+1}})$ for all $1 \leq p \leq \ell-1.$ Then $\reflection{i_1} \cdots \reflection{i_{\ell}}$ is a subexpression of $\embedding(\Reflection{j_1}) \cdots \embedding(\Reflection{j_k})$ if and only if $\adm(\reflection{i_1}) \cdots \adm(\reflection{i_{\ell}})$ is a subexpression of $\Reflection{j_1} \cdots \Reflection{j_k}.$
\end{lemma}

\begin{proof}
Let $\Reflection{j} \in S_{\tau}.$ Then $\reflection{i} \in S$ is a subexpression of $\embedding(\Reflection{j})$ if and only if $\adm(\reflection{i}) = \Reflection{j}.$ This follows from the surjectivity of $\adm$ and the property $\embedding(\adm(\reflection{i})) = \reflection{i}$ (if $\alpha_i \in \Simple^{\tauo}$) or $\reflection{i} \reflection{i}^*$ (if $\alpha_i \in \Simple_{rat} \sqcup \tauo(\Simple_{rat})$).
Let $\reflection{i_1} \cdots \reflection{i_{\ell}}$ be a subexpression of $\embedding(\Reflection{j_1}) \cdots \embedding(\Reflection{j_k}).$ Then for each $1 \leq p \leq \ell,$ $\reflection{i_p}$ is a subexpression of $\embedding(\Reflection{j_{q(p)}})$ for some $1 \leq q(p) \leq k$ such that $q(p) \leq q(p+1)$ and we have $\adm(\reflection{i_p}) = \Reflection{j_{q(p)}}.$ Now the condition $\adm(\reflection{i_p}) \neq \adm(\reflection{i_{p+1}})$ for all $1 \leq p \leq \ell-1$ guarantees that $q(p) < q(p+1)$ for each $p.$ Therefore, $\adm(\reflection{i_1}) \cdots \adm(\reflection{i_{\ell}})$ is a subexpression of $\Reflection{j_1} \cdots \Reflection{j_k}.$
Conversely, let $\adm(\reflection{i_1}) \cdots \adm(\reflection{i_{\ell}})$ be a subexpression of $\Reflection{j_1} \cdots \Reflection{j_k}.$ Then for each $1 \leq p \leq \ell,$ $\adm(\reflection{i_p}) = \Reflection{j_{q(p)}}$ for some $1 \leq q(1) < \cdots < q(\ell) \leq k.$ Each $\reflection{i_p}$ is a subexpression of $\embedding(\Reflection{j_{q(p)}}).$ Therefore, $\reflection{i_1} \cdots \reflection{i_{\ell}}$ is a subexpression of $\embedding(\Reflection{j_1}) \cdots \embedding(\Reflection{j_k}).$
\end{proof}

In order to yield further properties of the map $\adm: S \ra S_{\tau}$ and the group embedding $\embedding: \TWeyl \hra \Weyl,$ we introduce the notion of a \emph{folding branch} and a \emph{collapsing part}.

\begin{definition} \label{def:foldingbranch}
Let $\cd$ and $\cd_{\tau}$ denote the Coxeter diagram of $\Weyl$ and $\TWeyl$ respectively. A \emph{folding branch} of $\cd$ with respect to $\adm: S \ra S_{\tau}$ is a full subgraph $\cd'$ of $\cd$ obtained by choosing one vertex from $\adm \inv(R)$ for each $R \in \cd_{\tau}$ such that $m(s, s')=3$ for all $s, s' \in \cd'$ with $m(\adm(s), \adm(s')) \geq 3.$ 
A \emph{collapsing part} of $\cd$ with respect to $\adm$ is the full subgraph $\adm\inv(\{R, R'\})$ of $\cd$ where $R, R' \in S_{\tau}$ and $m(R, R') \geq 4.$
\end{definition}

A folding branch exists by Lemma~\ref{prop:philocal} since for every pair $R, R' \in \cd_{\tau}$ with $m(R, R')\geq 3,$ there exists a pair $s, s' \in \cd$ such that $\adm(s)=R,$ $\adm(s')=R'$ and $m(s, s')=3.$

\begin{example}[Type~$E_8$ to $H_4$] \label{egg:foldingbranch} 
We describe folding branches and a collapsing part for the folding from type $E_8$ to $H_4$ shown in Section~\ref{app:folding}. 
We have $S = \{\reflection{1}, \cdots, \reflection{8}\}$ and $S_{\tau}=\{\Reflection{1}, \cdots, \Reflection{4}\}.$ The map $\adm: S \ra S_{\tau}$ sends $\reflection{1}, \reflection{7}$ to $\Reflection{1},$ $\reflection{2}, \reflection{6}$ to $\Reflection{2},$ $\reflection{3}, \reflection{5}$ to $\Reflection{3},$ and $\reflection{4}, \reflection{8}$ to $\Reflection{4}.$ The Coxeter diagram of the original group admits three folding branches: $\cd_1 = \{\reflection{1}, \reflection{2}, \reflection{3}, \reflection{4}\},$ $\cd_2=\{\reflection{7}, \reflection{6}, \reflection{5}, \reflection{4}\},$ and $\cd_3=\{\reflection{7}, \reflection{6}, \reflection{5}, \reflection{8}\}.$ It has a collapsing part $\cd'=\{\reflection{3}, \reflection{4}, \reflection{5}, \reflection{8}\}.$
\end{example}

\noindent
Definition~\ref{def:foldingbranch} allows us to prove the following key property of the embedding $\embedding: \TWeyl \hra \Weyl.$

\begin{proposition} \label{prop:embeddingbruhat}
If $\Reflection{j_1}\cdots\Reflection{j_{\ell}}$ is a reduced expression in $\TWeyl,$ then $\embedding(\Reflection{j_1})\cdots \embedding(\Reflection{j_{\ell}})$ is a reduced expression in $\Weyl.$
\end{proposition}

\begin{proof}
We break down the map $\embedding$ into two steps.
First, for each $1 \leq p \leq \ell,$ let us choose $\reflection{i_p} \in \adm\inv(\Reflection{j_p})$ so that $\reflection{i_1}, \cdots, \reflection{i_{\ell}}$ all belong to a single folding branch. This is possible, since $\adm$ maps each folding branch of the Coxeter diagram $\cd$ of $\Weyl$ onto the Coxeter diagram $\cd_{\tau}$ of $\TWeyl.$ 
By Lemmas~\ref{prop:philocal} and \ref{prop:m2}, for those $1 \leq p \leq \ell-1$ such that $m(\Reflection{j_p}, \Reflection{j_{p+1}}) = 2$ or $3,$ we have $m(\reflection{i_p}, \reflection{i_{p+1}}) = m(\Reflection{j_p}, \Reflection{j_{p+1}}).$ Hence the only situation where the word $\reflection{i_1}\cdots\reflection{i_{\ell}}$ is not reduced is when $\Reflection{j_1}\cdots\Reflection{j_{\ell}}$ contains a consecutive subword of the form $[RR']_{m(R,R')}$ where $m(R,R') \geq 4,$ that is, equal to either $4$ or $5$ by Lemma~\ref{prop:philocal}.

Next, for each $1 \leq p \leq \ell,$ if $|\adm\inv (\Reflection{i_p})|=2,$ then insert $\reflection{i_p}^*$ on the right of $\reflection{i_p}.$ 
When $|\adm\inv (\Reflection{i_p})|=1,$ we adopt the convention $\reflection{i_p}^*=\reflection{i_p}.$
By Lemma~\ref{prop:philocal}, $\reflection{i_1}^*, \cdots, \reflection{i_{\ell}}^*$ also belong to a single folding branch (different from the first one if $|\adm\inv (\Reflection{i_p})|=2$ for at least one $p$). In particular, the observation in the previous paragraph applies to the word $\reflection{i_1}^* \cdots \reflection{i_{\ell}}^*$ as well.  We also notice that if $\reflection{i_p}^* \neq \reflection{i_p}$ and $\reflection{i_p}^*$ does not belong to a collapsing part of $\cd,$ then it commutes with all $\reflection{i_1}, \cdots, \reflection{i_{\ell}}.$ This is because by Lemma~\ref{prop:philocal}, each $s \in \cd$ not belonging to a collapsing part is disconnected from any folding branch not containing $s.$

Based on these observations, in order to determine if the $\Weyl$-word $\embedding(\Reflection{j_1})\cdots \embedding(\Reflection{j_{\ell}})$ is reduced or not, it suffices to look at the subwords consisting of elements of a collapsing part of $\cd.$ More precisely, for $R, R' \in S_{\tau}$ with $m(R, R')\in \{4, 5\},$ we will study the word of the form $[\embedding(R)\embedding(R')]_m$ where $4 \leq m \leq m(R, R'),$ and show that it is reduced.

First, consider the case $m(R,R')=4.$ The full subgraph $\adm\inv(\{R,R'\})$ is depicted in Lemma~\ref{prop:philocal} (C). We have 
\begin{equation*}
\embedding(R)\embedding(R')\embedding(R)\embedding(R') = ss^*tss^*t  \sim s^*sts\boxed{s^*}t \sim ss^*ts^*\boxed{s}t,
\end{equation*}
where $\sim$ denotes the braid relation on $\FM{S}.$ 
Notice that in the second word to last, the second $s^*$ (boxed) is playing the role of a separator between $sts$ and $t,$ preventing the occurrence of $stst.$ In the last word, the second $s$ (boxed) is separating $s^*ts^*$ and $t,$ preventing the occurrence of $s^*ts^*t.$ One also checks (by direct manipulation of the word) that $tsts$ and $ts^*ts^*$ cannot occur. Thus the word $ss^*tss^*t$ is reduced.
Next, consider the case $m(R,R')=5.$ The full subgraph $\adm\inv(\{R,R'\})$ is depicted in Lemma~\ref{prop:philocal} (D). We have
\begin{equation*}
\embedding(R)\embedding(R')\embedding(R)\embedding(R') = ss^*tt^*ss^*tt^*  \sim s^*t^*sts\boxed{s^*}tt^*.
\end{equation*}
Observe that in the word on the right, the second $s^*$ (boxed) is separating $sts$ and $t,$ preventing the occurrence of $stst.$ By looking at all other braid equivalent words, one checks (by a similar separation phenomenon) that $tsts,$ $s^*ts^*t,$ $ts^*ts^*,$ $s^*t^*s^*t^*,$ or $t^*s^*t^*s^*$ cannot occur either. Thus the word $ss^*tt^*ss^*tt^*$ is reduced. In the same way, we see that $\embedding(R')\embedding(R)\embedding(R')\embedding(R)$ is reduced. Finally,
\begin{equation*}
\embedding(R)\embedding(R')\embedding(R)\embedding(R')\embedding(R)=ss^*tt^*ss^*tt^*ss^* \sim s^*t^*sts\boxed{s^*}tst^*s^*.
\end{equation*}
Observe that in the word on the right, once again $s^*$ (boxed) is separating $sts$ and $ts$ preventing the occurrence of either $stst$ or $ststs.$ By looking at all other braid equivalent words, one checks (by a similar separation phenomenon) that $tsts,$ $s^*ts^*t,$ $ts^*ts^*,$ $s^*t^*s^*t^*,$ $t^*s^*t^*s^*,$ $tstst,$ $s^*ts^*ts^*,$ $ts^*ts^*t,$ $s^*t^*s^*t^*s^*,$ or $t^*s^*t^*s^*t^*$ cannot occur either. Thus the word $ss^*tt^*ss^*tt^*ss^*$ is reduced.
\end{proof}

\begin{corollary}
If $u' \leq u$ in $\TWeyl,$ then $\embedding(u') \leq \embedding(u)$ in $\Weyl.$
\end{corollary}

Recall the notation $\red{\Weyl}$ from Section~\ref{ssec:coxeter}. Let $\FM{S_{\tau}}$ be the free monoid generated on the Coxeter generating set $S_{\tau}.$ 
The map $\adm: S \ra S_{\tau}$ induces a new map $\widehat{\adm}: \red{\Weyl} \ra \FM{S_{\tau}}$ as follows. Let $\mathbf{i}:=\reflection{i_1} \cdots \reflection{i_{\ell}} \in \red{\Weyl}.$ 
Then apply $\adm$ to each $\reflection{i_p},$ and whenever $\adm(\reflection{i_p})=\adm(\reflection{i_{p+1}})$ occurs, replace the word $\adm(\reflection{i_p})\adm(\reflection{i_{p+1}})$ by $\adm(\reflection{i_p}).$
We introduce one notation to facilitate our discussion.

\begin{definition}
For $w \in \Weyl,$ we define $\admred{w}$ to be the subset of $\red{w}$ consisting of reduced words $\mathbf{i}$ such that the number of \emph{adjacency of opposite simple reflections} appearing in $\mathbf{i}$ is maximal among all the elements of $\red{w}.$ We define
\begin{equation*}
\admred{\Weyl} := \bigcup_{w \in \Weyl} \admred{w}.
\end{equation*}
\end{definition}

\noindent
For example, in the folding from type $A_3$ to $C_2,$ consider $w=\reflection{2}\reflection{1}\reflection{2}\reflection{3} \in \Weyl.$ We have
\begin{equation*}
\red{w}=\{\reflection{2}\reflection{1}\reflection{2}\reflection{3}, \reflection{1}\reflection{2}[\reflection{1}\reflection{3}],  \reflection{1}\reflection{2}[\reflection{3}\reflection{1}]\}, \quad 
\admred{w}=\{\reflection{1}\reflection{2}[\reflection{1}\reflection{3}], \reflection{1}\reflection{2}[\reflection{3}\reflection{1}]\},
\end{equation*}
where the adjacent opposite simple reflections are square-bracketed. The maximal number of adjacency of opposite simple reflections is one.

The following statement about $\widehat{\adm}$ is the foundation for the proof of Proposition~\ref{prop:c3rephrase2}. 

\begin{proposition} \label{prop:phihat}
Let $w \in \Weyl$ and let $\mathbf{i}=\reflection{i_1}\cdots\reflection{i_{\ell}} \in \admred{w}.$ Then $\widehat{\adm}(\mathbf{i}) \in \red{\TWeyl}.$
\end{proposition} 

\begin{proof}
First, we prove the following claim:
\begin{claim} \label{cl:adjacency}
Let $w$ and $\mathbf{i}$ be as in the statement. Let $\mathbf{j} := \widehat{\adm}(\mathbf{i}) = \Reflection{j_1}\cdots\Reflection{j_k}$ where $1 \leq k \leq \ell$ and each $\Reflection{j_q}$ denotes a simple reflection in $\TWeyl.$ Then for $\mathbf{j'} := \Reflection{j_1}\cdots\Reflection{j_{k-1}},$ there exists $w' \in \Weyl$ and $\mathbf{i'} \in \admred{w'}$ and $\widehat{\adm}(\mathbf{i'}) = \mathbf{j'}.$
\end{claim}

There are $1=a_1 < \cdots < a_k \leq \ell$ such that $\adm(\reflection{i_{a_p}}) = \Reflection{j_p}$ and $\adm(\reflection{i_{a_p}}) \neq \adm(\reflection{i_{a_p+1}})$ for each $1 \leq p \leq k.$ Take $w'=\reflection{i_1}\reflection{i_2}\cdots \reflection{i_{a_{k-1}}}.$  The word $\mathbf{i'}:=\reflection{i_1}\reflection{i_2}\cdots \reflection{i_{a_{k-1}}}$ belongs to $\red{w'}$ since it is a consecutive subword of the reduced word $\mathbf{i}.$ Moreover, since $\mathbf{i'}$ is obtained from $\mathbf{i} \in \admred{w}$ by removing either $\reflection{i_{a_k}}$ or the word $\reflection{i_{a_k}}^*\reflection{i_{a_k}}$ from the right end, $\mathbf{i'} \in \admred{w'}$ must be the case (otherwise, $\mathbf{i} \notin \admred{w}$). This proves the claim.

Since by construction we have $\Reflection{j_q} \neq \Reflection{j_{q+1}}$ for all $1 \leq q \leq k-1,$ in order for the sequence $\Reflection{j_1} \cdots \Reflection{j_k}$ not to be reduced, the only possibility is that it contains a pattern $[RR']_m$ where $R, R' \in S_{\tau}$ and $m > m(R, R').$ We argue that this situation cannot occur.

Let us consider all possible patterns of the $\adm$-preimage of each pair of simple reflections $R, R' \in S_{\tau}$ such that $m(R, R') \geq 3.$ There are four distinct patterns as in Lemma~\ref{prop:philocal}.
In (A), there is nothing to prove. 
In (B),  let us list all words $\mathbf{i} \in \admred{\Weyl}$ such that $\widehat{\adm}(\mathbf{i})=RR'R.$ Up to symmetry between $s$ and $s^*$ (resp. $t$ and $t^*$), they are:
\begin{equation*}
 [ss^*][tt^*][ss^*], \quad 
 [ss^*][tt^*]s, \quad
 s[tt^*][ss^*], \quad
 s[tt^*]s, \quad
 s[tt^*]s^*, \quad
 sts.
\end{equation*}
where a pair of opposite simple reflections are enclosed by brackets.
Direct computation shows that adding $t$ or $t^*$ to the right of each of the above reduced word will result in either a non-reduced word or a reduced word  not in $\admred{\Weyl}.$ We give two examples to describe these points: 
\begin{equation*}
 [ss^*][tt^*][ss^*]t \sim \underline{stst}s^*t^*s^*, \quad \quad
 stst^*  \sim s[tt^*]s.
\end{equation*}
The first shows a case where the word is not reduced (see the underlined part), and the second shows a case where the word is reduced but not in $\admred{\Weyl}.$

In (C),  let us list all words $\mathbf{i} \in \admred{\Weyl}$ such that $\widehat{\adm}(\mathbf{i})=RR'RR'.$ 
\begin{equation*}
[ss^*]t[ss^*]t, \quad st[ss^*]t, \quad  s^*t[ss^*]t. 
\end{equation*}
Direct computation shows that adding $s$ or $s^*$ to the right end of each of the above word results in either a non-reduced word or a reduced word not in $\admred{\Weyl}$: 
\begin{align*}
  [ss^*]t[ss^*]ts &\sim s\underline{s^*ts^*t}st \quad &st[ss^*]ts^* &\sim \underline{stst}s^*t   \\
  [ss^*]t[ss^*]ts^* &\sim s^*\underline{stst}s^*t  &s^*t[ss^*]ts &\sim \underline{s^*ts^*t}st  \\
 st[ss^*]ts &\sim [ss^*]t[ss^*]t  &s^*t[ss^*]ts^*&\sim [ss^*]t[ss^*]t
\end{align*}
The underlined parts indicate that the word is not reduced. 
The last line of each column shows examples where the word is reduced but not in $\admred{\Weyl}.$

In (D), let us list all reduced words $\mathbf{i} \in \admred{\Weyl}$ such that $\widehat{\adm}(\mathbf{i})=RR'RR'R.$ 
\begin{center}
\begin{tabular}{lllll}
 $[ss^*][tt^*][ss^*][tt^*][ss^*]$ & &  $[ss^*][tt^*]s^*[tt^*]s$ & & $s[tt^*][ss^*]ts$ \\
 $s[tt^*][ss^*][tt^*][ss^*]$ & & $s[tt^*][ss^*][tt^*]s$ & & $st[ss^*][tt^*]s^*$ \\
 $s^*[tt^*][ss^*][tt^*][ss^*]$ & & $s[tt^*][ss^*][tt^*]s^*$ & & $s^*[tt^*][ss^*]ts$ \\
 $[ss^*][tt^*][ss^*][tt^*]s$ & & $s^*[tt^*][ss^*][tt^*]s$ & & $s^*[tt^*]s^*[tt^*]s$ \\
 $[ss^*][tt^*][ss^*][tt^*]s^*$ & & $s^*[tt^*][ss^*][tt^*]s^*$ & & $s[tt^*]s^*[tt^*]s^*$ \\
 $st[ss^*][tt^*][ss^*]$ & & $sts^*[tt^*][ss^*]$ & & $[ss^*]t^*s^*t[ss^*]$ \\
 $[ss^*][tt^*][ss^*]ts$ & & $s^*t^*[ss^*]t[ss^*]$ & & $[ss^*]ts^*t^*[ss^*]$ \\
 $s[tt^*]s^*[tt^*][ss^*]$ & & $st[ss^*][tt^*]s$ & &
\end{tabular}
\end{center}

\noindent
Direct computation shows that adding $t$ or $t^*$ to the right end of each of the above words results in either a non-reduced word, or a reduced word not in $\admred{\Weyl}.$ Given below are two examples to illustrate these points. The first is an example of a non-reduced word (see the underlined part). The second is an example of a reduced word not in $\admred{\Weyl}$ (compare with the last line).
\begin{align*}
 [ss^*][tt^*][ss^*][tt^*][ss^*]t &\sim ss^*tt^*s^*stst^*s^* \quad &s[tt^*][ss^*][tt^*][ss^*]t &\sim t^*tsts^*tt^*ss^*t \\
 & \sim ss^*t^*ts^*tst^*ts^*t  & &\sim tt^*ss^*ts^*t^*s^*st \\
 & \sim ss^*t^*s^*ts^*t^*s^*sts^*  & &\sim [tt^*][ss^*][tt^*][ss^*][tt^*]\\
 & \sim st^*s^*\underline{t^*tt^*}s^*t^*sts^* 
\end{align*}
Finally, by Claim \ref{cl:adjacency},  we may discard the possibility of the appearance of the word $[RR']_m$ in $\Reflection{j_1} \cdots \Reflection{j_k}$ where $R,R' \in S_{\tau}$ for \emph{all} $m > m(R,R').$ This completes the proof.
\end{proof}

\section{Liftings of Schubert classes} \label{sec:lifting}

Let $u \in \Weyl_{\tau}^P$ with $\tlength{u} = {\ell}$ so that $\deg \tschubert{u} = {\ell}.$
Motivated by Lemma \ref{ob}, in order to find a lifting of $\tschubert{u},$ we look for an element $w \in \Weyl^P$ with the following properties:
\begin{enumerate}[(C1)]
\item $\length{w} ={\ell},$
\item $w \leq \embedding (u),$
\item $w \nleq \embedding (u')$ for all $u' \in \TWeyl^P$ such that $u \nleq u'.$
\end{enumerate}

\begin{proposition} \label{prop:liftingcondition} 
Let $u \in \TWeyl^P$ with $\tlength{u} = {\ell}$ and $w \in \Weyl^P.$ Then $\embedding^*(\schubert{w}) = c \cdot \tschubert{u}$ for some nonzero $c \in \basel$ if and only if $w$ satisfies the above conditions \normalfont{(C1)--(C3)}.
\end{proposition}

\begin{proof}
\noindent
$(\Ra)$ \,
Suppose $\embedding^*(\schubert{w}) = c \cdot \tschubert{u}$ for some $c \in \basel^{\times}.$
Then we have $\length{w} = \deg \schubert{w} = \deg \embedding^*(\schubert{w}) = \deg \tschubert{u} = \tlength{u} =\ell.$ Thus (C1) is verified. 
Since $0 \neq c\cdot \tschubert{u}(u) = \embedding^*(\schubert{w})(u) = \tauproj\big(\schubert{w}(\embedding(u))\big),$ we have $\schubert{w}(\embedding(u)) \neq 0.$ By Definition \ref{def:schubert}, $w \leq \embedding(u)$ follows, verifying (C2).
Finally, let $u' \in \TWeyl^P$ with $u \nleq u'.$ We have $0 = c \cdot \tschubert{u}(u') = \embedding^*(\schubert{w})(u') = \tauproj\big(\schubert{w}(\embedding(u'))\big).$
Since $\tauproj$ is injective (Remark \ref{rmk:tauproj}), we obtain $\schubert{w}(\embedding(u')) = 0.$ Now Lemma \ref{prop:schubert1} (c) gives us $w \nleq \embedding(u'),$ verifying (C3).

\noindent
$(\La)$ \,
Let $w$ satisfy (C1)--(C3).
We show that $\embedding^*(\schubert{w})$ is homogeneous of degree $\ell$ with properties $\embedding^*(\schubert{w})(u) \neq 0$ and $\embedding^*(\schubert{w})(u')=0$ for all $u' \ngeq u,$ then apply Lemma \ref{ob}. 
First, we have $\deg \embedding^*(\schubert{w}) = \deg \schubert{w} = \length{w}=\ell$ by (C1).
Next, we have $\embedding^*(\schubert{w})(u)= \tauproj\big(\schubert{w}(\embedding(u))\big).$ By (C2), we have $\embedding(u) \geq w,$ hence $\schubert{w}(\embedding(u)) \neq 0$ by Lemma \ref{prop:schubert1} (c). Since $\tauproj$ is injective, we obtain $\embedding^*(\schubert{w})(u) \neq 0.$
Finally, Let $u' \ngeq u.$ By (C3), we have $\embedding (u') \ngeq w$ which implies $\schubert{w}(\embedding (u'))= 0$ by Definition \ref{def:schubert}. Thus $\embedding^*(\schubert{w})(u') = \tauproj\big(\schubert{w}(\embedding(u'))\big) = 0.$ 
\end{proof}

\noindent
Proposition~\ref{prop:liftingcondition} states the condition for a Schubert class $\tschubert{u}$ of $\sa{\tpmoment}$ to admit a lifting $\schubert{w}$ in $\sa{\pmoment}$ in terms of the relationship between the group elements $u \in \TWeyl^P$ and $w \in \Weyl^P.$ This provides the first answer to Question~\ref{q:lifting}. However, verifying (C3) is not easy even for the smallest examples of foldings since it requires the full description of the moment graphs involved. We aim to restate the conditions (C1)--(C3) using the combinatorics of Coxeter groups.

Let us recall the map $\adm: S \ra S_{\tau}$ defined in \eqref{eq:adm}.
Using $\adm,$ we now define a subset $\foldingset{\Weyl} \subset \Weyl$ which is later shown to control the liftability of Schubert classes of $\sa{\tpmoment}.$

\begin{definition} \label{def:foldingset}
The \emph{folding subset} of $\Weyl,$ denoted by $\foldingset{\Weyl},$ is a subset of $\Weyl$ consisting of the identity element $e \in \Weyl$ and elements $w \in \Weyl$ satisfying the following.
\begin{enumerate}[\normalfont (FS1)]
\item If $\reflection{i_1} \cdots \reflection{i_{\ell}} \in \red{w},$ then $\adm(\reflection{i_1}) \cdots \adm(\reflection{i_{\ell}}) \in \red{\TWeyl}.$
\item If $\reflection{i_1} \cdots \reflection{i_{\ell}}, \reflection{i'_1} \cdots \reflection{i'_{\ell}} \in \red{w},$ then $\adm(\reflection{i_1}) \cdots \adm(\reflection{i_{\ell}}), \adm(\reflection{i'_1}) \cdots \adm(\reflection{i'_{\ell}}) \in \red{u}$ for some common $u \in \TWeyl.$
\end{enumerate}
\end{definition}

\begin{remark} \label{rmk:foldingset} 
\begin{enumerate}[\normalfont (a)]
\item Let $w \in \foldingset{\Weyl}$ and $\reflection{i_1}\cdots \reflection{i_{\ell}} \in \red{w}.$ For any $1 \leq p \leq q \leq \ell,$ the consecutive subexpression $w'=\reflection{i_p}\cdots \reflection{i_q}$ satisfies (FS1) and (FS2). Hence $w' \in \foldingset{\Weyl}.$  
\item Let us recall the map $\widehat{\adm}: \red{\Weyl} \ra \FM{S_{\tau}}$ defined before Proposition \ref{prop:phihat}. By (FS1), the map $\widehat{\adm}: \red{\Weyl} \ra \FM{S_{\tau}}$ restricted to $\cup_{w \in \foldingset{\Weyl}} \red{w}$ is the same as applying $\adm: S \ra S_{\tau}$ to each simple reflection in $\reflection{i_1}\cdots \reflection{i_{\ell}} \in \red{w}.$
\item Let $w \in \Weyl$ with a unique reduced expression $\reflection{i_1}\cdots\reflection{i_{\ell}}.$ Then $w \in \foldingset{\Weyl}$ by the following argument. By (b) and Proposition \ref{prop:phihat}, $\adm(\reflection{i_1})\cdots\adm(\reflection{i_{\ell}}) \in \red{\TWeyl},$ verifying (FS1). Since $\red{w}$ is a singleton, (FS2) trivially holds.
\end{enumerate}
\end{remark}

\noindent
By definition of the folding subset $\foldingset{\Weyl},$ the map $\adm: S \ra S_{\tau}$ extends to 
\begin{align} \label{eq:phibar}
\lbar{\adm}: \foldingset{\Weyl} \longra \TWeyl: \, &w=e \mapsto e \nonumber \\
&w=\reflection{i_1} \cdots \reflection{i_{\ell}} \mapsto \adm(\reflection{i_1}) \cdots \adm(\reflection{i_{\ell}}) \nonumber
\end{align}
where $w=\reflection{i_1} \cdots \reflection{i_{\ell}}$ is any reduced expression. 
The condition (FS2) guarantees that $\lbar{\adm}$ is well-defined, and (FS1) guarantees that it preserves the length, that is,  $\tlength{\lbar{\adm}(w)} = \length{w}$ for $w \in \foldingset{\Weyl}.$
Theorem \ref{prop:liftingcondition1} will restate the lifting conditions (C1)--(C3) solely in terms of $\lbar{\adm}.$

\begin{proposition} \label{prop:c3rephrase1} 
Let $u \in \TWeyl$ with $\tlength{u}={\ell}.$ Suppose $w \in \Weyl$ satisfies \normalfont{(C1), (C2)} and $w \in \foldingset{\Weyl}.$ Then $w$ satisfies \normalfont{(C3)}. 
\end{proposition}

\begin{proof}
\noindent
Let $w \in \foldingset{\Weyl}$ and let $u = \Reflection{j_1} \cdots \Reflection{j_{\ell}}$ be a reduced expression.
By applying the group homomorphism $\embedding,$ we obtain
\begin{equation}\label{Reduced}
\embedding(u) = \embedding(\Reflection{j_1}) \cdots \embedding(\Reflection{j_{\ell}}), 
\end{equation}
which is a reduced expression in $\Weyl$ by Proposition \ref{prop:embeddingbruhat}. 
Since $w \leq \embedding(u),$ by Lemma~\ref{prop:coxeter} (b), $w$ admits some reduced expression $w = \reflection{i_1} \cdots \reflection{i_{\ell}}$ which is a subexpression of \eqref{Reduced}. Since $w \in \foldingset{\Weyl},$ we have $\adm(\reflection{i_p}) \neq \adm(\reflection{i_{p+1}})$ for all $1 \leq p \leq \ell-1.$ By Lemma~\ref{prop:phi}, the sequence $\adm(\reflection{i_1}) \cdots \adm(\reflection{i_{\ell}})$ is a subexpression of $u = \Reflection{j_1} \cdots \Reflection{j_{\ell}}.$ Since the numbers of simple reflections in the two expressions are the same, we obtain $\adm(\reflection{i_1}) \cdots \adm(\reflection{i_{\ell}}) = \Reflection{j_1} \cdots \Reflection{j_{\ell}} = u.$

Let $u' \in \TWeyl$ such that $w \leq \embedding(u').$ We aim to show $u \leq u'.$
Let $u' = \Reflection{j'_1} \cdots \Reflection{j'_k}$ be a reduced expression. Then 
\begin{equation} \label{Reduced2}
\embedding(u') = \embedding(\Reflection{j'_1}) \cdots \embedding(\Reflection{j'_k})
\end{equation}
is a reduced expression in $\Weyl.$ Since $w \leq \embedding(u'),$ by Lemma~\ref{prop:coxeter} (b), $w$ admits some reduced expression $w = \reflection{i'_1} \cdots \reflection{i'_{\ell}}$ which is a subexpression of \eqref{Reduced2}.
Since $w \in \foldingset{\Weyl},$ we have $\adm(\reflection{i'_p}) \neq \adm(\reflection{i'_{p+1}})$ for all $1 \leq p \leq \ell-1.$ By Lemma~\ref{prop:phi}, it follows that $\adm(\reflection{i'_1}) \cdots \adm(\reflection{i'_{\ell}})$ is a subexpression of $u' = \Reflection{j'_1} \cdots \Reflection{j'_k}.$ Since $w \in \foldingset{\Weyl},$ $\adm(\reflection{i'_1}) \cdots \adm(\reflection{i'_{\ell}}) =\adm(\reflection{i_1}) \cdots \adm(\reflection{i_{\ell}})= u.$ By Lemma~\ref{prop:coxeter} (b), $u \leq u'$ follows.
\end{proof}

\begin{proposition} \label{prop:c3rephrase2} 
Let $u \in \TWeyl$ with $\tlength{u}={\ell}.$ Suppose $w \in \Weyl$ satisfies the conditions \normalfont{(C1)--(C3)}. Then $w \in \foldingset{\Weyl}.$
\end{proposition}

\begin{proof}
Our aim is to establish the defining properties (FS1), (FS2) of $\foldingset{\Weyl}$ for $w.$ First, we prove the following claim, which is apparently weaker than (FS1):

\begin{claim}
Every reduced expression $w = \reflection{i_1} \cdots \reflection{i_{\ell}}$ satisfies $\adm(\reflection{i_p}) \neq \adm(\reflection{i_{p+1}})$ for all $1 \leq p \leq \ell-1.$
\end{claim}

\noindent
By the definition of the set $\admred{w}$ (see before Proposition \ref{prop:phihat}), it suffices to prove the statement for the elements of $\admred{w}.$
Let $\reflection{i'_1} \cdots \reflection{i'_{\ell}} \in \admred{w}$ and recall the map $\widehat{\adm}: \red{\Weyl} \ra \FM{S_{\tau}}.$ 
By Proposition \ref{prop:phihat}, $\widehat{\adm}(\reflection{i'_1}\cdots \reflection{i'_{\ell}})$ is $\TWeyl$-reduced.
Write the $\TWeyl$-reduced word so obtained as $\Reflection{j'_1}\cdots \Reflection{j'_k}$ for some $k \leq \ell,$ and let $u':=\Reflection{j'_1} \cdots \Reflection{j'_k} \in \TWeyl.$  Suppose $\adm(\reflection{i'_p}) = \adm(\reflection{i'_{p+1}})$ for some $1 \leq p \leq \ell-1.$ Then $k < \ell,$ that is, $\tlength{u'} < \tlength{u}.$ Since $w = \reflection{i'_1} \cdots \reflection{i'_{\ell}}$ is a subexpression of $\embedding(u') = \embedding(\Reflection{j'_1}) \cdots \embedding(\Reflection{j'_k})$ by construction, by Lemma~\ref{prop:coxeter} (b), we obtain $w \leq \embedding(u')$ where $u \nleq u',$ a contradiction to (C3). Hence $\adm(\reflection{i'_p}) \neq \adm(\reflection{i'_{p+1}})$ for all $1 \leq p \leq \ell-1.$ 

We now proceed to prove (FS1).
Knowing that every element of $\red{w}$ has no adjacency of pairs of opposite simple reflections has two consequences. First, Proposition \ref{prop:phihat} applies to each $\reflection{i_1} \cdots \reflection{i_{\ell}} \in \red{w}.$ Second, the map $\widehat{\adm}$ coincides with the operation of applying $\adm$ to each $\reflection{i_p}.$
Thus, for each  $\reflection{i_1} \cdots \reflection{i_{\ell}} \in \red{w},$ the expression $\widehat{\adm}(\reflection{i_1} \cdots \reflection{i_{\ell}}) = \adm(\reflection{i_1}) \cdots \adm(\reflection{i_{\ell}})$ is $\TWeyl$-reduced by Proposition \ref{prop:phihat}.
Finally, we prove (FS2). Let $u = \Reflection{j_1} \cdots \Reflection{j_{\ell}}$ be a reduced expression. Since $w \leq \embedding(u),$ by Lemma~\ref{prop:coxeter} (b), $w$ has a reduced expression $\reflection{i_1}\cdots \reflection{i_{\ell}}$ which is a subexpression of $\embedding(\Reflection{j_1}) \cdots \embedding(\Reflection{j_{\ell}}).$ Since $\adm(\reflection{i_p}) \neq \adm(\reflection{i_{p+1}})$ holds for all $1 \leq p \leq \ell-1,$ by Lemma~\ref{prop:phi}, $\adm(\reflection{i_1}) \cdots \adm(\reflection{i_{\ell}})$ is a subexpression of $\Reflection{j_1} \cdots \Reflection{j_{\ell}}.$ Since the numbers of simple reflections involved in the two expressions are the same, we obtain $\adm(\reflection{i_1}) \cdots \adm(\reflection{i_{\ell}}) = \Reflection{j_1} \cdots \Reflection{j_{\ell}}.$ 
Suppose that there is $\reflection{i'_1} \cdots \reflection{i'_{\ell}} \in \red{w}$ with $\adm(\reflection{i'_1}) \cdots \adm(\reflection{i'_{\ell}}) \in \red{u'}$ for some $u' \neq u$ in $\TWeyl.$
By construction, $\reflection{i'_1} \cdots \reflection{i'_{\ell}}$ is a subexpression of $\embedding\big(\adm(\reflection{i'_1})\big) \cdots \embedding\big(\adm(\reflection{i'_{\ell}})\big).$ Since $u'=\adm(\reflection{i'_1}) \cdots \adm(\reflection{i'_{\ell}})$ is $\TWeyl$-reduced by (FS1), $\embedding(u')=\embedding(\adm(\reflection{i'_1})) \cdots \embedding(\adm(\reflection{i'_{\ell}}))$ is a $\Weyl$-reduced expression (Proposition~\ref{prop:embeddingbruhat}). By Lemma~\ref{prop:coxeter} (b), we obtain $w = \reflection{i'_1} \cdots \reflection{i'_{\ell}} \leq \embedding(u')$ where $\tlength{u'} = \tlength{u}$ and $u' \neq u$ (hence $u \nleq u'$). This is a contradiction to (C3).
\end{proof}

\begin{theorem} \label{prop:liftingcondition1}
Let $u \in \TWeyl$ with $\tlength{u}={\ell}$ and $w \in \Weyl.$ Then $w$ satisfies \normalfont{(C1)--(C3)} if and only if $w \in \foldingset{\Weyl}$ and $u=\lbar{\adm}(w).$
\end{theorem}

\begin{proof}
Suppose $w$ satisfies (C1)--(C3). Then by Proposition \ref{prop:c3rephrase2}, $w \in \foldingset{\Weyl}.$ Let $u = \Reflection{j_1} \cdots \Reflection{j_{\ell}}$ be a reduced expression. Since $\length{w}=\ell$ and $w \leq \embedding(u),$ by Lemma~\ref{prop:coxeter} (b), $w$ has a reduced expression $w = \reflection{i_1} \cdots \reflection{i_{\ell}}$ which is a subexpression of $\embedding(\Reflection{j_1}) \cdots \embedding(\Reflection{j_{\ell}}).$ Since we have $\adm(\reflection{i_p}) \neq \adm(\reflection{i_{p+1}})$ for all $1 \leq p \leq \ell-1$ thanks to $w \in \foldingset{\Weyl},$ by Lemma~\ref{prop:phi}, it follows that $\adm(\reflection{i_1}) \cdots \adm(\reflection{i_{\ell}})$ is a subexpression of $\Reflection{j_1} \cdots \Reflection{j_{\ell}}.$ Since the numbers of simple reflections in the two expressions are the same, we obtain $\adm(\reflection{i_1}) \cdots \adm(\reflection{i_{\ell}}) = \Reflection{j_1} \cdots \Reflection{j_{\ell}},$ that is, $\lbar{\adm}(w) = u.$

Conversely, suppose $w \in \foldingset{\Weyl}$ and $u=\lbar{\adm}(w).$ Since $\lbar{\adm}$ is length-preserving, $\length{w}=\ell$ must hold, establishing (C1). Next, $u=\lbar{\adm}(w)$ implies that for some $\Weyl$-reduced expression $w = \reflection{i_1} \cdots \reflection{i_{\ell}},$ we have $u=\adm(\reflection{i_1}) \cdots \adm(\reflection{i_{\ell}})$ which is $\TWeyl$-reduced. Then $\reflection{i_1} \cdots \reflection{i_{\ell}}$ is a subexpression of $\embedding(\adm(\reflection{i_1})) \cdots \embedding(\adm(\reflection{i_{\ell}}))$ by construction. Hence by Lemma~\ref{prop:coxeter} (b), $w \leq \embedding(u),$ establishing (C2). Finally, by Proposition \ref{prop:c3rephrase1}, $w \in \foldingset{\Weyl}$ implies (C3).
\end{proof}

\noindent
By Proposition~\ref{prop:liftingcondition} and Theorem~\ref{prop:liftingcondition1}, we obtain the following lifting criterion.

\begin{corollary} \label{prop:liftingcondition2}
Let $u \in \TWeyl^P.$ Then the Schubert class $\tschubert{u}$ admits a lifting to $\sa{\pmoment}$ if and only if $u \in \lbar{\adm}(\foldingset{\Weyl} \cap \Weyl^P).$
\end{corollary}

\begin{corollary} \label{prop:reduced}
Let $u \in \TWeyl^P$ with $\tlength{u}={\ell}.$ If there exists $w \in \Weyl^P$ with a unique reduced expression $w = \reflection{i_1} \cdots \reflection{i_{\ell}}$ and if $\adm(\reflection{i_1}) \cdots \adm(\reflection{i_{\ell}}) = u,$ then $\tschubert{u}$ lifts to $\schubert{w}.$
\end{corollary}

\begin{proof}
By Remark \ref{rmk:foldingset} (c), $w$ belongs to $\foldingset{\Weyl},$ and we have $\lbar{\adm}(w)=u.$ Hence by Corollary \ref{prop:liftingcondition2}, $\tschubert{u}$ admits a lifting, namely $\schubert{w}.$
\end{proof}

When a Schubert class $\tschubert{u} \in \sa{\tpmoment}$ lifts to $\schubert{w} \in \sa{\pmoment},$ we explicitly know the nonzero constant $c \in \basel$ appearing in the Definition \ref{def:schubertlifting}.

\begin{proposition} \label{prop:liftingcoefficient}
Let $\schubert{w} \in \sa{\pmoment}$ be a lifting of $\tschubert{u} \in \sa{\tpmoment}.$ Then
$\embedding^*(\schubert{w}) = \tau^m \cdot \tschubert{u}$
for some unique $m \in \Z_{\geq 0}.$
\end{proposition}

\begin{proof}
We have $\embedding^*(\schubert{w}) = c \cdot \tschubert{u}$ for some nonzero $c \in \basel.$ In particular, by looking at the coordinate at $\embedding(u),$ this implies $\tauproj\big(\schubert{w}(\embedding(u))\big) = c\cdot \tschubert{u}(u).$ 
Let $w=\reflection{1}\cdots \reflection{\ell}$ be a $\Weyl$-reduced expression, and write $\adm(\reflection{i})=:\Reflection{i}.$ Then $u=\lbar{\adm}(w)=\Reflection{1}\cdots \Reflection{\ell}$ is a $\TWeyl$-reduced expression (since $w \in \foldingset{\Weyl}$ by Theorem \ref{prop:liftingcondition1}). Apply $\embedding$ to this expression and write the $\Weyl$-reduced expression so obtained as $\embedding(u)=\reflection{1}'\cdots \reflection{L}'$ for appropriate $L.$ By construction, this expression contains $w=\reflection{1}\cdots \reflection{\ell}$ as a subexpression. Denote this subexpression by $\reflection{i_1}' \cdots \reflection{i_{\ell}}'$ where $1 \leq i_1 < \cdots < i_{\ell} \leq L.$ We have $\reflection{i_p}'=\reflection{p}$ and $\adm(\reflection{i_p}') = \adm(\reflection{p}) = \Reflection{p}$ for each $1 \leq p \leq \ell.$ Let $\alpha_{p}$ denote the simple root corresponding to $\reflection{p},$ and let $\beta_p$ denote the simple root corresponding to $\Reflection{p}.$ By Definition \ref{def:schubert}, we compute 
\begin{equation} \label{eq:nonzeroentry}
\schubert{w}\big(\embedding(u)\big) = \prod_{p=1}^{\ell} \mathbf{r}_{i_p}(\alpha_{p}),
\end{equation}
where $\mathbf{r}_{i_p}(\alpha_{p}) = \reflection{1}'\reflection{2}' \cdots \reflection{i_p -1}'(\alpha_{p}).$ By Lemma~\ref{prop:embeddingproperty},
\[
\tauproj\big(\mathbf{r}_{i_p}(\alpha_{p})\big) = \tauproj\big(\embedding(\Reflection{1}\cdots\Reflection{p-1})(\alpha_{p})\big) =
\begin{cases}
\Reflection{1}\cdots\Reflection{p-1}(\beta_p) \,\, &\text{if} \,\, \alpha_{p} \in \Simple^{\tauo} \sqcup \Simple_{rat},\\
\Reflection{1}\cdots\Reflection{p-1}(\tau\beta_p) &\text{if} \,\, \alpha_p \in \tauo(\Simple_{rat}).
\end{cases}
\]
Hence we obtain
\begin{equation*}
\tauproj\big(\schubert{w}(\embedding(u))\big) = \tau^m \prod_{p=1}^{\ell} \Reflection{1}\cdots\Reflection{p-1}(\beta_p) = \tau^m \cdot \tschubert{u}(u)
\end{equation*}
where $m$ is the number of elements $\alpha_p \in \Simple_{rat}$ appearing in \eqref{eq:nonzeroentry}. Thus $c = \tau^m.$
\end{proof}

\begin{definition}
Let $u \in \TWeyl^P.$ We say that $u$ is \emph{liftable} or \emph{admits a lifting} if there exists an element $w \in \Weyl^P$ such that $w \in \foldingset{\Weyl}$ and $\lbar{\adm}(w)=u.$ Any such element $w \in \Weyl^P$ is called a \emph{lifting} of $u$ and we say that $u$ \emph{lifts to} $w.$
\end{definition}

\noindent
Using this terminology, Corollary \ref{prop:liftingcondition2} says that a  Schubert class $\tschubert{u} \in \sa{\tpmoment}$ lifts to $\schubert{w} \in \sa{\pmoment}$ if and only if $u$ lifts to $w.$

\begin{example}[Type $A_4$ to $H_2$] \label{egg:typeA4lifting}
Let $\Root \rightsquigarrow \Root_{\tau}$ be the folding of the root system of type $A_4$ to $H_2.$ See the folding diagram in Section~\ref{app:folding}. 
The folded Coxeter group $\TWeyl$ is the dihedral group $D_{10}.$ For each $u \in \TWeyl$ with $0 \leq \tlength{u} \leq 4,$ Table~\ref{fig:liftingsA4} shows a complete list of liftings $w \in \foldingset{\Weyl}.$ 
In this case, every lifting $w \in \foldingset{\Weyl}$ has a unique reduced expression. 
It will be shown in Lemma~\ref{prop:basicnonliftable2} (VI)(a) that the longest element $u = \Reflection{1}\Reflection{2}\Reflection{1}\Reflection{2}\Reflection{1} = \Reflection{2}\Reflection{1}\Reflection{2}\Reflection{1}\Reflection{2}$ is not liftable.

\begin{table} 
\begin{center}
\begin{tabular}{ |c|c|c|}
\hline
$\tlength{u}$ & $u \in \TWeyl$ & Liftings $w \in \foldingset{\Weyl}$ \\
\hline
$4$ & $\Reflection{1}\Reflection{2}\Reflection{1}\Reflection{2}$ & $\reflection{1}\reflection{2}\reflection{3}\reflection{4}$ \\
\hline
$4$ & $\Reflection{2}\Reflection{1}\Reflection{2}\Reflection{1}$ & $\reflection{4}\reflection{3}\reflection{2}\reflection{1}$ \\
\hline
$3$ & $\Reflection{1}\Reflection{2}\Reflection{1}$ & $\reflection{1}\reflection{2}\reflection{3},$ \, $\reflection{3}\reflection{2}\reflection{1}$ \\
\hline
$3$ & $\Reflection{2}\Reflection{1}\Reflection{2}$ & $\reflection{2}\reflection{3}\reflection{4},$ \, $\reflection{4}\reflection{3}\reflection{2}$ \\
\hline
$2$ & $\Reflection{1}\Reflection{2}$ & $\reflection{1}\reflection{2},$ \, $\reflection{3}\reflection{2},$ \, $\reflection{3}\reflection{4}$ \\
\hline
$2$ & $\Reflection{2}\Reflection{1}$ & $\reflection{2}\reflection{1},$ \, $\reflection{2}\reflection{3},$ \, $\reflection{4}\reflection{3}$  \\
\hline
$1$ & $\Reflection{1}$ & $\reflection{1},$ \, $\reflection{3}$ \\
\hline
$1$ & $\Reflection{2}$ & $\reflection{2},$ \, $\reflection{4}$\\
\hline
$0$ & $e$ & $e$ \\
\hline
\end{tabular}
\caption{Liftings of $\TWeyl$-elements: type $A_4$ to type $H_2$}
\label{fig:liftingsA4}
\end{center}
\end{table}
\end{example}

\begin{example}[Type $D_6 / A_4$ to $H_3 / H_2$] \label{egg:typeD6lifting}
Let $\Root \rightsquigarrow \Root_{\tau}$ be the folding of the root system of type $D_6$ to $H_3.$ See the folding diagram in Section~\ref{app:folding}.
We take the parabolic subgroup $\Weyl_P := \langle \reflection{2}, \reflection{3}, \reflection{4}, \reflection{6}\rangle,$ which is of type $A_4.$ 
We have $(\Weyl_P)_{\tau} = \langle \Reflection{2}, \Reflection{3}\rangle,$ and it forms a parabolic subgroup of $\TWeyl$ of type $H_2.$  
For each $u \in \TWeyl^P,$ Table \ref{fig:liftingsD6modA4} shows a complete list of liftings $w \in \foldingset{\Weyl} \cap \Weyl^P.$
In Table \ref{fig:liftingsD6modA4}, observe first that each $w = \reflection{i_1} \cdots \reflection{i_{\ell}}$ (reduced expression) satisfies $\adm(\reflection{i_p})= \Reflection{j_p}$ for each $1 \leq p \leq \ell,$ where $u=\Reflection{j_1} \cdots \Reflection{j_{\ell}}$ is the prescribed reduced expression of $u.$ For those $w \in \Weyl^P$ with $0 \leq \length{w} \leq 5,$ $w$ has a unique reduced expression. Hence by Corollary \ref{prop:reduced}, $\tschubert{u}$ lifts to $\schubert{w}.$ For those $w \in \Weyl^P$ with $6 \leq \length{w} \leq 10,$ observe that each $w$ has precisely two distinct reduced expressions by commuting $\reflection{5}$ and $\reflection{6},$ and that the two reduced expressions obtained  by applying $\adm$ to each $\reflection{i_p} (1 \leq p \leq \length{w})$ give rise to the same element $u \in \TWeyl.$ Hence $w \in \foldingset{\Weyl} \cap \Weyl^P,$ and by Corollary \ref{prop:liftingcondition2}, $\tschubert{u}$ lifts to $\schubert{w}$ for $w \in \Weyl$ given in Table~\ref{fig:liftingsD6modA4}.  

\begin{table}
\begin{center}
\begin{tabular}{ |c|c|c|}
\hline
$\tlength{u}$ & $u \in \TWeyl$ & Liftings $w \in \foldingset{\Weyl} \cap \Weyl^P$ \\
\hline
$10$ & $\Reflection{1}\Reflection{2}\Reflection{3}\Reflection{2}\Reflection{1}\Reflection{3}\Reflection{2}\Reflection{3}\Reflection{2}\Reflection{1}$ & $\reflection{1}\reflection{2}\reflection{3}\reflection{4}\reflection{5}\reflection{6}\reflection{4}\reflection{3}\reflection{2}\reflection{1}$ \\
\hline
$9$ & $\Reflection{2}\Reflection{3}\Reflection{2}\Reflection{1}\Reflection{3}\Reflection{2}\Reflection{3}\Reflection{2}\Reflection{1}$ & $\reflection{2}\reflection{3}\reflection{4}\reflection{5}\reflection{6}\reflection{4}\reflection{3}\reflection{2}\reflection{1}$ \\
\hline
$8$ & $\Reflection{3}\Reflection{2}\Reflection{1}\Reflection{3}\Reflection{2}\Reflection{3}\Reflection{2}\Reflection{1}$ & $\reflection{3}\reflection{4}\reflection{5}\reflection{6}\reflection{4}\reflection{3}\reflection{2}\reflection{1}$ \\
\hline
$7$ & $\Reflection{2}\Reflection{1}\Reflection{3}\Reflection{2}\Reflection{3}\Reflection{2}\Reflection{1}$ & $\reflection{4}\reflection{5}\reflection{6}\reflection{4}\reflection{3}\reflection{2}\reflection{1}$ \\
\hline
$6$ & $\Reflection{1}\Reflection{3}\Reflection{2}\Reflection{3}\Reflection{2}\Reflection{1}$ & $\reflection{5}\reflection{6}\reflection{4}\reflection{3}\reflection{2}\reflection{1}$ \\
\hline
$5$ & $\Reflection{3}\Reflection{2}\Reflection{3}\Reflection{2}\Reflection{1}$ & $\reflection{6}\reflection{4}\reflection{3}\reflection{2}\reflection{1}$  \\
\hline
$5$ & $\Reflection{1}\Reflection{2}\Reflection{3}\Reflection{2}\Reflection{1}$ & $\reflection{5}\reflection{4}\reflection{3}\reflection{2}\reflection{1},$  \, $\reflection{1}\reflection{2}\reflection{3}\reflection{4}\reflection{5}$\\
\hline
$4$ & $\Reflection{2}\Reflection{3}\Reflection{2}\Reflection{1}$ & $\reflection{4}\reflection{3}\reflection{2}\reflection{1},$ \, $\reflection{2}\reflection{3}\reflection{4}\reflection{5}$\\
\hline
$3$ & $\Reflection{3}\Reflection{2}\Reflection{1}$ & $\reflection{3}\reflection{2}\reflection{1},$ \, $\reflection{3}\reflection{4}\reflection{5},$ \, $\reflection{6}\reflection{4}\reflection{5}$\\
\hline
$2$ & $\Reflection{2}\Reflection{1}$ & $\reflection{2}\reflection{1},$ \, $\reflection{4}\reflection{5}$\\
\hline
$1$ & $\Reflection{1}$ & $\reflection{1},$ \, $\reflection{5}$\\
\hline
$0$ & $e$ & $e$ \\
\hline
\end{tabular}
\caption{Liftings of $\TWeyl$-elements: type $D_6/A_4$ to type $H_3/H_2$}
\label{fig:liftingsD6modA4}
\end{center}
\end{table}

\end{example}

\section{Nonliftable elements} \label{sec:nolifting}
In this section, we discuss a way to detect nonliftable elements of $\TWeyl.$

\begin{lemma} \label{prop:sublifting1}
Let $u \in \TWeyl$ with $\tlength{u}={\ell}$ and let $w \in \foldingset{\Weyl}$ be a lifting of $u.$ Let $\Reflection{j_1}\cdots\Reflection{j_{\ell}} \in \red{u}$ and suppose that there exists $\reflection{i_1}\cdots\reflection{i_{\ell}} \in \red{w}$ such that $\adm(\reflection{i_p})=\Reflection{j_p}$ for each $1 \leq p \leq \ell.$ Then for any $1 \leq p \leq q \leq \ell,$ the consecutive subexpression $u':=\Reflection{j_p}\Reflection{j_{p+1}}\cdots\Reflection{j_q}$ lifts to $w' := \reflection{i_p}\reflection{i_{p+1}}\cdots \reflection{i_q}.$
\end{lemma}

\begin{proof}
Since $w'= \reflection{i_p} \cdots \reflection{i_q}$ is a consecutive subexpression of $w=\reflection{i_1} \cdots \reflection{i_{\ell}} \in \foldingset{\Weyl},$ $w'$ also belongs to $\foldingset{\Weyl}$ (Remark~\ref{rmk:foldingset} (a)). By construction, $\lbar{\adm}(w') = u'.$ Hence $w'$ is a lifting of $u'.$
\end{proof}

\begin{lemma} \label{prop:basicnonliftable}
Let $R, R' \in S_{\tau}.$ 
\begin{enumerate}[\normalfont (a)]
\item If $m(R, R')=4,$ then $u=[R, R']_4 = RR'RR'$ is not liftable.
\item If $m(R, R')=5,$ then $u=[R, R']_5=RR'RR'R$ is not liftable.
\end{enumerate}
\end{lemma}

\begin{remark}
If $m(R, R') = 3,$ then $u=RR'R$ is always liftable (by the existence of a folding branch, see after Definition~\ref{def:foldingbranch}). 
\end{remark}

\begin{proof}
(a) \, The set $\red{u}$ consists of precisely two elements: $RR'RR'$ and $R'RR'R.$ By Lemma~\ref{prop:philocal}, in the Coxeter diagram, the configuration of the full subgraph $\adm\inv(\{R, R'\})$ is (C) (up to symmetry of $R$ and $R'$). Suppose, for the sake of contradiction, $w \in \foldingset{\Weyl}$ is a lifting of $u.$ \\[.1cm]
\textbf{Case~1:}\, $RR'RR' \in \widehat{\adm}\big(\red{w}\big).$ There exists $\mathbf{i}:=\reflection{i_1}\cdots\reflection{i_4} \in \red{w}$ such that
\begin{equation*}
\widehat{\adm}(\mathbf{i}) = \adm(\reflection{i_1}) \cdots \adm(\reflection{i_4}) = RR'RR'.
\end{equation*}
By Lemma \ref{prop:sublifting1}, $w'=\reflection{i_1}\reflection{i_2}\reflection{i_3}$ is a lifting of $u'=RR'R.$ Since $w' \in \foldingset{\Weyl},$ the only possibilities are $\reflection{i_1}\reflection{i_2}\reflection{i_3}=sts^*$ or $s^*ts.$ On the other hand, $\adm(\reflection{i_4})=R'$ forces $\reflection{i_4}=t.$ However, both $w=sts^*t$ and $w=s^*tst$ violate (FS2) of  Definition \ref{def:foldingset}. \\[.1cm]
\textbf{Case~2:}\, $R'RR'R \in \widehat{\adm}\big(\red{w}\big).$ There exists $\mathbf{i}:=\reflection{i_1}\cdots\reflection{i_4} \in \red{w}$ such that
\begin{equation*}
\widehat{\adm}(\mathbf{i}) = \adm(\reflection{i_1}) \cdots \adm(\reflection{i_4}) = R'RR'R.
\end{equation*}
By Lemma \ref{prop:sublifting1}, $w'=\reflection{i_2}\reflection{i_3}\reflection{i_4}$ is a lifting of $u'=RR'R.$ Since $w' \in \foldingset{\Weyl},$ the only possibilities are $\reflection{i_2}\reflection{i_3}\reflection{i_4}=sts^*$ or $s^*ts.$ On the other hand, $\adm(\reflection{i_1})=R'$ forces $\reflection{i_1}=t.$ However, $w=tsts^*$ and $w=ts^*ts$ both  violate (FS2) of  Definition \ref{def:foldingset}.

\noindent
(b) \, The set $\red{u}$ consists of precisely two elements: $RR'RR'R$ and $R'RR'RR'.$ By Lemma~\ref{prop:philocal}, in the Coxeter diagram, the configuration of the full subgraph $\adm\inv(\{R, R'\})$ is (D) (up to symmetry of $R$ and $R'$). Suppose, for the sake of contradiction, $w \in \foldingset{\Weyl}$ is a lifting of $u.$ \\[.1cm]
\textbf{Case~1:}\, $RR'RR'R \in \widehat{\adm}\big(\red{w}\big).$ There exists $\mathbf{i}:=\reflection{i_1}\cdots\reflection{i_5} \in \red{w}$ such that
\begin{equation*}
\widehat{\adm}(\mathbf{i}) = \adm(\reflection{i_1}) \cdots \adm(\reflection{i_5}) = RR'RR'R.
\end{equation*}
By Lemma \ref{prop:sublifting1}, $w'=\reflection{i_1}\reflection{i_2}\reflection{i_3}\reflection{i_4}$ is a lifting of $u'=RR'RR'.$ Since $w' \in \foldingset{\Weyl},$ the only possibility is $\reflection{i_1}\reflection{i_2}\reflection{i_3}\reflection{i_4}=sts^*t^*.$ On the other hand, $\adm(\reflection{i_5})=R$ forces $\reflection{i_5}=s$ or $s^*.$ However, $w=sts^*t^*s$ and $w=sts^*t^*s^*$ both  violate (FS2) of  Definition \ref{def:foldingset}. \\[.1cm]
\textbf{Case~2:}\, $R'RR'RR' \in \widehat{\adm}\big(\red{w}\big).$ We may argue in exactly the same way as in Case~1.
\end{proof}

\begin{proposition} \label{prop:sublifting}
Let $u \in \TWeyl$ with $\tlength{u}={\ell}$ and let $w \in \foldingset{\Weyl}$ be a lifting of $u.$ Then for every $ \Reflection{j_1} \cdots \Reflection{j_{\ell}} \in \red{u},$ there exists $\reflection{i_1} \cdots \reflection{i_{\ell}} \in \red{w}$ such that $\adm(\reflection{i_p})=\Reflection{j_p}$ for each $1 \leq p \leq \ell.$
\end{proposition}

\begin{proof}
We first remark that the last part of the claim is equivalent to $\widehat{\adm}(\reflection{i_1}\cdots\reflection{i_{\ell}})=\Reflection{j_1}\cdots\Reflection{j_{\ell}}$ (Remark~\ref{rmk:foldingset} (b)).
Let $\mathbf{i} \in \red{w}.$ Assume $\widehat{\adm}\big(\red{w}\big) \subsetneq \red{u}$ and let $\mathbf{j} \in \red{u} \backslash \widehat{\adm}\big(\red{w}\big).$ Since both $\widehat{\adm}(\mathbf{i})$ and $\mathbf{j}$ belong to $\red{u},$ there exists a finite sequence of $\TWeyl$-braid moves that transforms $\widehat{\adm}(\mathbf{i})$ into $\mathbf{j},$ and the corresponding sequence $\widehat{\adm}(\mathbf{i}) = \mathbf{j}_0 \ra \mathbf{j}_1 \ra \cdots \ra \mathbf{j}_r=\mathbf{j}$ of elements of $\red{u}.$ Since $\widehat{\adm}(\mathbf{i}) \in \widehat{\adm}\big(\red{w}\big)$ and $\mathbf{j} \notin \widehat{\adm}\big(\red{w}\big),$ there exists $0 \leq k \leq r-1$ such that $\mathbf{j}_k \in \widehat{\adm}\big(\red{w}\big)$ and $\mathbf{j}_{k+1} \notin \widehat{\adm}\big(\red{w}\big).$ Let us write $\mathbf{j}_k=\widehat{\adm}(\mathbf{i'})$ for some $\mathbf{i'}=\reflection{i'_1}\cdots\reflection{i'_{\ell}} \in \red{w},$ and set $\Reflection{j'_q}:=\adm(\reflection{i'_q})$ for each $1 \leq q \leq \ell.$
We will argue that in the $\TWeyl$-braid move $[\Reflection{j'_p}\Reflection{j'_{p+1}}]_m=[\Reflection{j'_{p+1}}\Reflection{j'_p}]_m$ that transforms $\mathbf{j}_k$ into $\mathbf{j}_{k+1},$ $m$ cannot be equal to $2$ or $3.$ After that, we show that $m \geq 4$ is also not possible, yielding a contradiction to our hypothesis.

The case of $m=2$ is discarded as follows. By Lemma~\ref{prop:m2}, $m=2$ would imply that any $s \in \adm\inv(\Reflection{j'_p})$ and $t \in \adm\inv(\Reflection{j'_{p+1}})$ commute. Hence, there is a corresponding $\Weyl$-braid move $\reflection{i'_p}\reflection{i'_{p+1}}=\reflection{i'_{p+1}}\reflection{i'_p},$ which transforms $\mathbf{i'}$ to another element $\mathbf{i''} \in \red{w}$ such that $\widehat{\adm}(\mathbf{i''})=\mathbf{j}_{k+1}.$ This contradicts $\mathbf{j}_{k+1} \notin \widehat{\adm}\big(\red{w}\big).$
The case of $m=3$ is discarded as follows. Set $R= \Reflection{j'_p}$ and $R'=\Reflection{j'_{p+1}}.$ By Lemma~\ref{prop:philocal}, there are two possible configurations of the full subgraph $\adm\inv(\{R, R'\}),$ namely (A) and (B). In (A), there is a corresponding $\Weyl$-braid move $[\reflection{i'_p}\reflection{i'_{p+1}}]_3=[\reflection{i'_{p+1}}\reflection{i'_p}]_3,$ which transforms $\mathbf{i'}$ to  another element $\mathbf{i''} \in \red{w}$ such that $\widehat{\adm}(\mathbf{i''})=\mathbf{j}_{k+1},$ contradicting $\mathbf{j}_{k+1} \notin \widehat{\adm}\big(\red{w}\big).$
In (B), the possible $\widehat{\adm}$-preimages of the word $[RR']_3$ consisting of $3$ simple reflections are: $sts, sts^*, st^*s^*, s^*t^*s^*, s^*t^*s, s^*ts,$ among which only $sts$ and $s^*t^*s^*$ are compatible with $w \in \foldingset{\Weyl}.$
Hence, there is a corresponding $\Weyl$-braid move $[\reflection{i'_p}\reflection{i'_{p+1}}]_3=[\reflection{i'_{p+1}}\reflection{i'_p}]_3,$ which transforms $\mathbf{i'}$ to  another element $\mathbf{i''} \in \red{w}$ such that $\widehat{\adm}(\mathbf{i''})=\mathbf{j}_{k+1}.$ This contradicts $\mathbf{j}_{k+1} \notin \widehat{\adm}\big(\red{w}\big).$ 
As for the case $m \geq 4,$ we argue as follows. Set $R= \Reflection{j'_p}$ and $R'=\Reflection{j'_{p+1}}.$ If $m=4,$ by Lemma \ref{prop:basicnonliftable} (a),  $u'=[RR']_4$ is not liftable. Since $\mathbf{j}_k \in \widehat{\adm}(\red{w}),$ this is a contradiction to Lemma \ref{prop:sublifting1}. If $m=5,$ by Lemma \ref{prop:basicnonliftable} (b), $u''=[RR']_5$ is not liftable. Since $\mathbf{j}_k \in \widehat{\adm}(\red{w}),$ this is a contradiction to Lemma \ref{prop:sublifting1}.
By Lemma~\ref{prop:philocal}, the cases $m \geq 6$ do not exist. This completes the proof.
\end{proof}
\noindent
By a direct application of Lemma \ref{prop:sublifting1} and Proposition \ref{prop:sublifting}, we obtain

\begin{corollary} \label{prop:noliftingcondition}
Let $u \in \TWeyl$ with $\tlength{u}={\ell}.$ If $u$ admits a $\TWeyl$-reduced expression $u = \Reflection{j_1} \cdots \Reflection{j_{\ell}}$ which has a consecutive subexpression $u'= \Reflection{j_p}\Reflection{j_{p+1}}\cdots \Reflection{j_q} \, (1 \leq p \leq q \leq \ell)$ such that $u'$ is not liftable, then $u$ is not liftable.
\end{corollary}

\begin{proposition} \label{prop:basicnonliftable2}
The following elements of $\TWeyl$ are not liftable:
\begin{enumerate}[\normalfont (I)]
\item Type $A_{2n-1}$ to $C_n:$
\begin{inparaenum}[\normalfont (a)]
\item $\Reflection{n}\Reflection{n-1}\Reflection{n},$
\item $\Reflection{n-1}\Reflection{n-2}\Reflection{n-1}\Reflection{n}\Reflection{n-1},$
\item $\Reflection{n-1}\Reflection{n}\Reflection{n-1}\Reflection{n}.$
\end{inparaenum}
\item Type $D_{n+1}$ to $B_n:$
\begin{inparaenum}[\normalfont (a)]
\item $\Reflection{n-1}\Reflection{n}\Reflection{n-1},$
\item $\Reflection{n-1}\Reflection{n}\Reflection{n-1}\Reflection{n}.$
\end{inparaenum}
\item Type $E_6$ to $F_4:$
\begin{inparaenum}[\normalfont (a)]
\item $\Reflection{2}\Reflection{3}\Reflection{2},$
\item $\Reflection{3}\Reflection{4}\Reflection{3}\Reflection{2}\Reflection{3},$
\item $\Reflection{2}\Reflection{3}\Reflection{2}\Reflection{3}.$
\end{inparaenum}
\item Type $E_8$ to $H_4:$
\begin{inparaenum}[\normalfont (a)]
\item $\Reflection{3}\Reflection{2}\Reflection{3}\Reflection{4}\Reflection{3},$
\item $\Reflection{3}\Reflection{4}\Reflection{3}\Reflection{4}\Reflection{3}.$
\end{inparaenum}
\item Type $D_6$ to $H_3:$
\begin{inparaenum}[\normalfont (a)]
\item $\Reflection{2}\Reflection{1}\Reflection{2}\Reflection{3}\Reflection{2},$
\item $\Reflection{2}\Reflection{3}\Reflection{2}\Reflection{3}\Reflection{2}.$
\end{inparaenum}
\item Type $A_4$ to $H_2:$
\begin{inparaenum}[\normalfont (a)]
\item $\Reflection{1}\Reflection{2}\Reflection{1}\Reflection{2}\Reflection{1}.$
\end{inparaenum}
\end{enumerate}
\end{proposition}

\begin{proof}
(I)(a), (II)(a), (III)(a): Proof for these three cases follows the same argument as one can see from the folding diagrams in Section~\ref{app:folding}. We present here a proof for (I)(a).
Set $u = \Reflection{n}\Reflection{n-1}\Reflection{n}.$ If $u$ is liftable, then there exists some $w=\reflection{i_1}\reflection{i_2}\reflection{i_3} \in \foldingset{\Weyl}$ such that $\adm(\reflection{i_1})=\Reflection{n}, \, \adm(\reflection{i_2}) = \Reflection{n-1}$ and $\adm(\reflection{i_3})=\Reflection{n}$ (Proposition~\ref{prop:sublifting}). There are precisely two candidates: $w_1=\reflection{n}\reflection{n-1}\reflection{n}$ and $w_2=\reflection{n}\reflection{n+1}\reflection{n}.$ Observe that $\reflection{n}\reflection{n-1}\reflection{n}=\reflection{n-1}\reflection{n}\reflection{n-1}$ in $\Weyl,$ but 
\begin{equation*}
\adm(\reflection{n})\adm(\reflection{n-1})\adm(\reflection{n}) =\Reflection{n}\Reflection{n-1}\Reflection{n} \neq \Reflection{n-1}\Reflection{n}\Reflection{n-1} = \adm(\reflection{n-1})\adm(\reflection{n})\adm(\reflection{n-1}).
\end{equation*}
Hence $w_1 \notin \foldingset{\Weyl}.$ Similarly, $w_2 \notin \foldingset{\Weyl}.$

\noindent
(I)(b) and (III)(b): Proof for these two cases follows the same argument. We present here a proof for (I)(b). Set $u = \Reflection{n-1}\Reflection{n-2}\Reflection{n-1}\Reflection{n}\Reflection{n-1}.$ If $u$ is liftable, then by Proposition \ref{prop:sublifting}, there exists some $w=\reflection{i_1}\cdots \reflection{i_5} \in \foldingset{\Weyl}$ such that $\adm(\reflection{i_p})=\Reflection{j_p}$ for each $1 \leq p \leq 5.$ By Lemma \ref{prop:sublifting1}, $\reflection{i_1}\reflection{i_2}$ is a lifting of $u_1=\Reflection{n-1}\Reflection{n-2}$ and $\reflection{i_3}\reflection{i_4}\reflection{i_5}$ is a lifting of $u_2 = \Reflection{n-1}\Reflection{n}\Reflection{n-1}.$ The element $u_1$ has two liftings $w_1 = \reflection{n-1}\reflection{n-2}$ and $w_1' = \reflection{n+1}\reflection{n+2},$ and $u_2$ also has two liftings $w_2 = \reflection{n-1}\reflection{n}\reflection{n+1}$ and $w_2' = \reflection{n+1}\reflection{n}\reflection{n-1}.$ Thus there are four candidates for $w$:
\begin{align*}
v_1 = w_1w_2 &= \reflection{n-1}\reflection{n-2} \reflection{n-1}\reflection{n}\reflection{n+1} = \reflection{n-2}\reflection{n-1}\reflection{n}\reflection{n-2}\reflection{n+1},\\
v_2 = w_1w_2' &= \reflection{n-1}\reflection{n-2}\reflection{n+1}\reflection{n}\reflection{n-1}, \\
v_3 = w_1'w_2 &= \reflection{n+1}\reflection{n+2}\reflection{n-1}\reflection{n}\reflection{n+1}, \\
v_4 = w_1'w_2' &=  \reflection{n+1}\reflection{n+2}\reflection{n+1}\reflection{n}\reflection{n-1} = \reflection{n+2}\reflection{n+1}\reflection{n}\reflection{n+2}\reflection{n-1}.
\end{align*}
Notice that a reduced expression of $v_1$ and $v_2$ contains a consecutive subexpression $\reflection{n-2}\reflection{n+1}$ and observe that $\reflection{n-2}\reflection{n+1}=\reflection{n+1}\reflection{n-2}$ in $\Weyl$ but $\Reflection{n-2}\Reflection{n-1} \neq \Reflection{n-1}\Reflection{n-2}$ in $\TWeyl.$ Hence $v_1, v_2 \notin \foldingset{\Weyl}.$ A similar argument shows $v_3, v_4 \notin \foldingset{\Weyl}$ (look at the subexpression $\reflection{n+2}\reflection{n-1}$). Hence $u$ is not liftable.

\noindent
(I)(c), (II)(b), (III)(c) follow from Lemma \ref{prop:basicnonliftable} (a).

\noindent
(IV)(a) and (V)(a): Proof for these two cases follows the same argument as one can see from the folding diagrams in Section~\ref{app:folding}. We present here a proof for (IV)(a).  Set $u = \Reflection{3}\Reflection{2}\Reflection{3}\Reflection{4}\Reflection{3}.$ If $u$ is liftable, then by Proposition \ref{prop:sublifting}, there exists some $w=\reflection{i_1}\cdots \reflection{i_5} \in \foldingset{\Weyl}$ such that $\adm(\reflection{i_p})=\Reflection{j_p}$ for each $1 \leq p \leq 5.$ By Lemma \ref{prop:sublifting1}, $\reflection{i_1}\reflection{i_2}$ is a lifting of $u_1=\Reflection{3}\Reflection{2}$ and $\reflection{i_3}\reflection{i_4}\reflection{i_5}$ is a lifting of $u_2 = \Reflection{3}\Reflection{4}\Reflection{3}.$ The element $u_1$ has two liftings $w_1 = \reflection{3}\reflection{2}$ and $w_1' = \reflection{5}\reflection{6},$ and $u_2$ also has two liftings $w_2 = \reflection{3}\reflection{4}\reflection{5}$ and $w_2' = \reflection{5}\reflection{4}\reflection{3}.$ Thus there are four candidates for $w$:
\begin{align*}
v_1 = w_1w_2 &= \reflection{3}\reflection{2} \reflection{3}\reflection{4}\reflection{5} = \reflection{2}\reflection{3}\reflection{4}\reflection{2}\reflection{5},\\
v_2 = w_1w_2' &= \reflection{3}\reflection{2}\reflection{5}\reflection{4}\reflection{3}, \\
v_3 = w_1'w_2 &= \reflection{5}\reflection{6}\reflection{3}\reflection{4}\reflection{5}, \\
v_4 = w_1'w_2' &=  \reflection{5}\reflection{6}\reflection{5}\reflection{4}\reflection{3} = \reflection{6}\reflection{5}\reflection{4}\reflection{6}\reflection{3}.
\end{align*}
Notice that reduced expressions of $v_1$ and $v_2$ contain a consecutive subexpression $\reflection{2}\reflection{5}$ and observe that $\reflection{2}\reflection{5}=\reflection{5}\reflection{2}$ in $\Weyl$ but $\Reflection{2}\Reflection{3} \neq \Reflection{3}\Reflection{2}$ in $\TWeyl.$ Hence $v_1, v_2 \notin \foldingset{\Weyl}.$ A similar argument shows $v_3, v_4 \notin \foldingset{\Weyl}$ (look at the subexpression $\reflection{6}\reflection{3}$). Hence $u$ is not liftable.

\noindent
(IV)(b), (V)(b) and (VI)(a) follow from Lemma \ref{prop:basicnonliftable} (b).
\end{proof}

Using Corollary \ref{prop:noliftingcondition} and Proposition \ref{prop:basicnonliftable2}, we can detect less trivial nonliftable elements of $\TWeyl$ as in the following example.

\begin{example}[Type $E_8 / D_6$ to $H_4 / H_3$] \label{egg:typeE8lifting}
Let $\Root \rightsquigarrow \Root_{\tau}$ be a folding of the root system of type $E_8$ to $H_4.$ See the folding diagram in Section~\ref{app:folding}. 
We take the parabolic subgroup $\Weyl_P := \langle \reflection{2}, \reflection{3}, \reflection{4}, \reflection{5}, \reflection{6}, \reflection{8} \rangle,$ which is of type $D_6.$ 
We have $(\Weyl_P)_{\tau}= \langle \Reflection{2}, \Reflection{3}, \Reflection{4}\rangle,$ which is a parabolic subgroup of $\TWeyl$ of type $H_3.$ Below is an example of a nonliftable element of $\TWeyl^P$ (of the smallest length).
\begin{equation*}
u = \Reflection{3}\Reflection{2}\Reflection{3}\Reflection{4}\Reflection{3}\Reflection{1}\Reflection{2}\Reflection{4}\Reflection{3}\Reflection{4}\Reflection{3}\Reflection{2}\Reflection{1} \in \TWeyl^P.
\end{equation*}
This reduced expression contains a nonliftable consecutive subexpression $u'=\Reflection{3}\Reflection{2}\Reflection{3}\Reflection{4}\Reflection{3}$ (Lemma \ref{prop:basicnonliftable2} (IV)(a)). Hence by Corollary \ref{prop:noliftingcondition}, it follows that $u$ is not liftable.
\end{example}

\section{Lifting property of the folding map} \label{sec:liftingprop}
We say that the folding map $\embedding^*: \sa{\pmoment} \ra \sa{\tpmoment}$ has the \emph{lifting property} if the Schubert class $\tschubert{u}$ admits a lifting for every $u \in \TWeyl^P.$ In this section, for each of the twisted quadratic foldings listed in Section~\ref{app:folding}, we apply the results established in Sections~\ref{sec:lifting} and \ref{sec:nolifting} to obtain a complete list of parabolic subsets $P \subset S$ (recall that $S$ denotes a Coxeter generating set of the original reflection group $\Weyl$) such that $\embedding^*$ has the lifting property.
We use $P_{\tau} \subset S_{\tau}$ to denote a parabolic subset for the folded Coxeter group $\TWeyl.$
It is recalled that by construction of a folding of a subsystem $\Root_P,$ the simple system $\Simple_P$ corresponding to $P$ must satisfy $\tauo(\spn_{\basek}\Simple_P) = \spn_{\basek}\Simple_P$ (see before Proposition~\ref{prop:restriction}). We use the following notational convention: if $S =\{\cgen{1}, \cdots, \cgen{n}\}$ and $I$ is a (possibly empty) subset of $\{1, \cdots, n\},$ then we write $P_{I}$ to denote the subset $\{\cgen{i} \mid i \notin I\} \subset S$. 

\begin{proposition} \label{prop:liftingprop}
The folding map $\embedding^*: \sa{\pmoment} \ra \sa{\tpmoment}$ has the lifting property if and only if $P \subset S$ is one of the following.
\begin{enumerate}[\normalfont (a)]
\item Type $A_{2n-1}$ to $C_n:$ $\Para{1, 2n-1} \subset P$ (equivalently, $P=\Para{1, 2n-1}$ or $S$).
\item Type $D_{n+1}$ to $B_n:$ $\Para{n, n+1} \subset P$ (equivalently, $P=\Para{n, n+1}$ or $S$).
\item Type $E_6$ to $F_4:$ $P=S.$
\item Type $E_8$ to $H_4:$ $P = S.$
\item Type $D_6$ to $H_3:$ $\Para{1,5} \subset P$ (equivalently, $P=\Para{1,5}$ or $S$).
\item Type $A_4$ to $H_2:$ $P \neq \emptyset$ (equivalently, $P=\Para{1,3},$ $\Para{2,4}$ or $S$).
\end{enumerate}
\end{proposition}

\begin{proof}
\noindent
(a) $(\Ra)$ \,
Suppose $\Para{1, 2n-1} \not\subset P,$ or equivalently, $\TPara{1} \not\subset P_{\tau}.$ We will find an element $u \in \TWeyl^P$ that is not liftable. There exists some $2 \leq i \leq n$ such that $\Reflection{i} \notin P_{\tau}.$
If $\Reflection{n} \notin P_{\tau},$ then $\Reflection{n}\Reflection{n-1}\Reflection{n} \in \TWeyl^P,$ which is not liftable (Lemma \ref{prop:basicnonliftable2} (I)(a)). If $\Reflection{n-1} \notin P_{\tau},$ then $\Reflection{n-1}\Reflection{n-2}\Reflection{n-1}\Reflection{n}\Reflection{n-1} \in \TWeyl^P,$ which is not liftable (Lemma \ref{prop:basicnonliftable2} (I)(b)).
For $2 \leq i \leq n-2,$ take 
$u=\Reflection{n-1}\Reflection{n-2}\Reflection{n-1}\Reflection{n}\Reflection{n-1} (\Reflection{n-3}\Reflection{n-2})(\Reflection{n-4}\Reflection{n-3}) \cdots (\Reflection{i-1}\Reflection{i}) \in \TWeyl^P.$
Since $u$ contains a nonliftable consecutive subexpression $u' =\Reflection{n-1}\Reflection{n-2}\Reflection{n-1}\Reflection{n}\Reflection{n-1},$ by Corollary \ref{prop:noliftingcondition}, $u$ is not liftable. Thus, $\embedding^*$ does not have the lifting property.

\noindent
$(\La)$ \, For $P=S,$ the statement is clear.
Let $P=\Para{1,2n-1},$ or equivalently, $P_{\tau}=\TPara{1}.$ An arbitrary element $u \in \TWeyl^P$ with $\tlength{u} \geq 1$ has a reduced expression either
\[
u=
\begin{cases}
\Reflection{j} \Reflection{j-1} \cdots \Reflection{1} \,\, &\text{for some} \,\, 1 \leq j \leq n, \,\,\text{or},\\
\Reflection{k}\Reflection{k+1} \cdots \Reflection{n-2}\Reflection{n-1}\Reflection{n}\Reflection{n-1}\Reflection{n-2} \cdots \Reflection{2}\Reflection{1} &\text{for some} \,\, 1 \leq k \leq n-1.
\end{cases}
\]
In the former case, $u$ has the following two liftings in $\Weyl^P$:
\begin{align*}
w_1 &= \reflection{j}\reflection{j-1} \cdots \reflection{2}\reflection{1},\\
w_2 &= \reflection{2n-j}\reflection{2n-j+1} \cdots \reflection{2n-2}\reflection{2n-1}.
\end{align*}
Observe that these are unique reduced expressions for $w_1$ and $w_2$ respectively, hence belong to $\foldingset{\Weyl}$ (Remark~\ref{rmk:foldingset} (c)). In the latter case, $u$ has the following two liftings in $\Weyl^P$:
\begin{align*}
w_3 &= \reflection{k}\reflection{k+1}\cdots \reflection{n-2}\reflection{n-1}\reflection{n}\reflection{n+1}\reflection{n+2} \cdots \reflection{2n}\reflection{2n-1}, \\
w_4 &= \reflection{2n-k}\reflection{2n-k-1}\cdots \reflection{n+2}\reflection{n+1}\reflection{n}\reflection{n-1}\reflection{n-2} \cdots \reflection{2}\reflection{1}. 
\end{align*}
Again, these are unique reduced expressions for $w_3$ and $w_4$ respectively, hence belong to $\foldingset{\Weyl}.$ Thus $\embedding^*$ has the lifting property. 

\noindent
(b) $(\Ra)$ \, Suppose $\Para{n,n+1} \not\subset P,$ or equivalently, $\TPara{n} \not\subset P_{\tau}.$ Then there exists $1 \leq i \leq n-1$ such that $\Reflection{i} \notin P_{\tau}.$ Take
$u = \Reflection{n-1}\Reflection{n}\Reflection{n-1}\Reflection{n-2} \cdots \Reflection{i} \in \TWeyl^P.$
Since $u$ contains a nonliftable consecutive subexpression $u' =\Reflection{n-1}\Reflection{n}\Reflection{n-1}$ (Lemma \ref{prop:basicnonliftable2} (II)(a)), by Corollary \ref{prop:noliftingcondition}, $u$ is not liftable. Hence $\embedding^*$ does not have the lifting property.

\noindent
$(\La)$ \, For $P=S,$ the statement is clear. Let $P = \Para{n, n+1},$ or equivalently, $P_{\tau}=\TPara{n}.$ First, we establish that each element of $\TWeyl^P$ has a reduced expression of the following form:
\begin{center}
\begin{tikzpicture}
\node[draw] at (0, 0) {$\mathbf{\longra} \Reflection{1}$};
\node[draw] at (1.5,0) {$\mathbf{\longra} \Reflection{2}$};
\draw (2.2,0) node[anchor=west] {$\cdots$};
\draw (0,-0.4) node[anchor=north] {$\mathbf{B_1}$};
\draw (1.5,-0.4) node[anchor=north] {$\mathbf{B_2}$};

\node[draw] at (3.5, 0) {$\mathbf{\longra} \Reflection{j}$};
\draw (3.5,-0.4) node[anchor=north] {$\mathbf{B_j}$};
\draw (4.2,0) node[anchor=west] {$\cdots$};

\node[draw] at (5.7, 0) {$\mathbf{\longra} \Reflection{n-2}$};
\node[draw] at (7.7, 0) {$\mathbf{\longra} \Reflection{n-1}$};
\node[draw] at (9.2, 0) {$\Reflection{n}$};

\draw (5.7,-0.4) node[anchor=north] {$\mathbf{B_{n-2}}$};
\draw (7.7,-0.4) node[anchor=north] {$\mathbf{B_{n-1}}$};
\draw (9.2,-0.4) node[anchor=north] {$\mathbf{B_n}$};
\end{tikzpicture}
\end{center}
where each block $\mathbf{B_j}$ is either empty ($\emptyset$) or a word of the form $\Reflection{k}\Reflection{k-1}\cdots\Reflection{j}$ for some $j \leq k \leq n$ (the arrow within each block $\mathbf{B_j}$ indicates that the subscripts of the simple reflections are decreasing), satisfying the following conditions:
if $\mathbf{B_j}\neq \emptyset,$ then $\mathbf{B_{j'}} \neq \emptyset$ for all $j' > j$;
if the block $\mathbf{B_j}$ starts with $\Reflection{n},$ then every block $\mathbf{B_{j'}}$ with $j'>j$ also starts with $\Reflection{n}$; 
if the block $\mathbf{B_j}$ does not start with $\Reflection{n},$ then $0 \leq |\mathbf{B_1}| \leq |\mathbf{B_2}| \leq \cdots \leq |\mathbf{B_j}|,$ where each $|\mathbf{B_i}|$ denotes the size of the block $\mathbf{B_i}.$ 
By construction, the above word is reduced, and every reduced word that is braid equivalent to it ends with $\Reflection{n},$ hence the group element represented by the above reduced word belongs to $\TWeyl^P.$ It remains to show that the reduced expressions of the above form exhaust all the elements of $\TWeyl^P.$ We will use a counting argument.

 We denote $b_i:=|\mathbf{B_i}|$ for each $1 \leq i \leq n.$ Let $j$ be the smallest number such that block $\mathbf{B_j}$ begins with $\Reflection{n}.$ Then we have
\begin{equation}\label{eq:sizecondition}
0 \leq b_1 \leq b_2 \leq \cdots \leq b_{j-1} \leq b_j=n-(j-1).
\end{equation}
We first take care of two extreme cases, namely, $j=1, n+1.$ When $j=1,$ every block begins with $\Reflection{n},$ yielding
$(\Reflection{n}\Reflection{n-1}\cdots\Reflection{1})(\Reflection{n}\Reflection{n-1}\cdots\Reflection{2}) \cdots (\Reflection{n}\Reflection{n-1})\Reflection{n}.$
When $j=n+1,$ no block begins with $\Reflection{n},$ yielding the empty word corresponding to $e \in \TWeyl.$
Now suppose $2 \leq j \leq n.$ We count the $(j-1)$-tuple $(b_1, \cdots, b_{j-1})$ satisfying \eqref{eq:sizecondition}.

For $x \in \Q,$ we denote the largest integer no greater than $x$ by $\lceil x \rceil.$ We need to distribute the numbers $\{0, 1, 2, \cdots, n-(j-1)\}$ among $b_1, b_2, \cdots, b_{j-1}$ in the weakly ascending order. Let $0 \leq k \leq \min\{j-2, n-(j-1)\}.$ We place $k$ separation walls between $j-1$ objects, and in the resulting $k+1$ separated sections, we distribute $k+1$ numbers in strictly ascending order from left to right. The number of choices for such distribution is
$\begin{psmallmatrix}
j-2 \\
k
\end{psmallmatrix}
\cdot
\begin{psmallmatrix}
n-(j-1)+1 \\
k+1
\end{psmallmatrix}.$
By varying $1 \leq j \leq n+1,$ the total number of choices is found to be
\begin{equation*} \label{eq:numberofchoices}
F(n) := 2 + \sum_{j=2}^{\lceil \tfrac{n+3}{2} \rceil} \sum_{k=0}^{j-2} 
\begin{psmallmatrix}
j-2 \\
k
\end{psmallmatrix}
\cdot
\begin{psmallmatrix}
n-(j-1)+1 \\
k+1
\end{psmallmatrix}
+ \sum_{j=\lceil \tfrac{n+3}{2} \rceil +1}^{n} \sum_{k=0}^{n-(j-1)}
\begin{psmallmatrix}
j-2 \\
k
\end{psmallmatrix}
\cdot
\begin{psmallmatrix}
n-(j-1)+1 \\
k+1
\end{psmallmatrix}.
\end{equation*}
On the other hand, we have $|\TWeyl^P|=\frac{|\TWeyl|}{|(\Weyl_P)_{\tau}|} = \frac{|B_n|}{|A_{n-1}|} =\frac{2^n\cdot n\ !}{n\ !}= 2^n.$
\begin{claim} \label{2^n}
We have $F(n)=2^n$ for all $n\geq 2.$ (Proof postponed at the end.)
\end{claim}
Let $\mu=\mathbf{B_1}\mathbf{B_2}\cdots\mathbf{B_n}$ be a $\TWeyl$-reduced word of the form described above, and let $u \in \TWeyl^P$ be the element represented by $\mu.$
We build a $\Weyl$-reduced word $\lambda$ of the following form:
\begin{center}
\begin{tikzpicture}
\node[draw] at (0, 0) {$\mathbf{\longra} \reflection{1}$};
\node[draw] at (1.5,0) {$\mathbf{\longra} \reflection{2}$};
\draw (2.2,0) node[anchor=west] {$\cdots$};
\draw (0,-0.4) node[anchor=north] {$\mathbf{A_1}$};
\draw (1.5,-0.4) node[anchor=north] {$\mathbf{A_2}$};

\node[draw] at (3.5, 0) {$\mathbf{\longra} \reflection{j}$};
\draw (3.5,-0.4) node[anchor=north] {$\mathbf{A_j}$};
\draw (4.2,0) node[anchor=west] {$\cdots$};

\node[draw] at (5.7, 0) {$\mathbf{\longra} \reflection{n-2}$};
\node[draw] at (7.5, 0) {$\mathbf{\longra} \reflection{n-1}$};
\node[draw] at (9.5, 0) {$\reflection{n}$ or $\reflection{n+1}$};

\draw (5.7,-0.4) node[anchor=north] {$\mathbf{A_{n-2}}$};
\draw (7.5,-0.4) node[anchor=north] {$\mathbf{A_{n-1}}$};
\draw (9.5,-0.4) node[anchor=north] {$\mathbf{A_n}$};
\end{tikzpicture}
\end{center}
where each block $\mathbf{A_j}$ is obtained from block $\mathbf{B_j}$ as follows: if $\mathbf{B_j}=\emptyset,$ then $\mathbf{A_j}=\emptyset$; if $\mathbf{B_j}=\Reflection{k}\Reflection{k-1}\cdots\Reflection{j},$ then $\mathbf{A_j}=s\reflection{k-1}\cdots \reflection{j}$ where $s = \reflection{k}$ if $k \neq n,$ and $s=\reflection{n}$ or $\reflection{n+1}$ if $k=n$ such that
if $\mathbf{A_j}$ begins with $\reflection{n}$ (resp. $\reflection{n+1}$), then $\mathbf{A_{j+1}}$ begins with $\reflection{n+1}$ (resp. $\reflection{n}$).
\begin{claim}
The element $w \in \Weyl^P$ represented by $\lambda$ above belongs to $\foldingset{\Weyl}.$
\end{claim}
We verify the conditions (FS1) and (FS2) of Definition \ref{def:foldingset}. By definition of $w$ via $\lambda,$ no element of $\red{w}$ contains the consecutive subword $\reflection{n}\reflection{n+1}$ or $\reflection{n+1}\reflection{n}.$ Hence Proposition~\ref{prop:phihat} applies to $\lambda,$ and we obtain that $\widehat{\adm}(\lambda)$ is $\TWeyl$-reduced. Moreover, since $\reflection{n}$ and $\reflection{n+1}$ are not adjacent in $\lambda,$ $\widehat{\adm}(\lambda)$ coincides with applying $\adm$ to each simple reflection in $\lambda.$ Thus (FS1) is verified.
Next, we observe from the Coxeter diagrams in Section~\ref{app:folding} that among the $\Weyl$-braid moves $[st]_{m(s,t)}=[ts]_{m(s,t)},$ the only ones for which the corresponding $\TWeyl$-moves $[\adm(s)\adm(t)]_{m(s,t)}=[\adm(t)\adm(s)]_{m(s,t)}$ are false are the following:
\begin{equation*}
\reflection{n}\reflection{n-1}\reflection{n} = \reflection{n-1}\reflection{n}\reflection{n-1}, \quad 
\reflection{n+1}\reflection{n-1}\reflection{n+1}=\reflection{n-1}\reflection{n+1}\reflection{n-1},
\end{equation*}
namely, the long braid moves within the collapsing part. By definition of $w$ via $\lambda,$ we can directly check that no element of $\red{w}$ contains the consecutive subword $\reflection{n}\reflection{n-1}\reflection{n},$ $\reflection{n-1}\reflection{n}\reflection{n-1},$ $\reflection{n+1}\reflection{n-1}\reflection{n+1}$ or $\reflection{n-1}\reflection{n+1}\reflection{n-1}.$ Hence, for every $\Weyl$-braid move $[st]_{m(s,t)}=[ts]_{m(s,t)}$ within $\red{w},$ there is a corresponding $\TWeyl$-braid move $[\adm(s)\adm(t)]_{m(s,t)}=[\adm(t)\adm(s)]_{m(s,t)}.$ Thus $\widehat{\adm}\big(\red{w}\big)$ $\subset \red{u}.$ This verifies (FS2). 
Hence $w \in \foldingset{\Weyl}$ and we have $\lbar{\adm}(w)=u$ as desired.\\[.1cm]
\textbf{Proof of Claim \ref{2^n}:} Set
\[
f(n, j) :=
\begin{cases}
\sum_{k=0}^{j-1}\begin{psmallmatrix}j-1\\k\end{psmallmatrix}\begin{psmallmatrix}n-j+1\\k+1\end{psmallmatrix} \,\, \text{if} \,\, 1\leq j \leq \lceil\tfrac{n+1}{2}\rceil, \\
\sum_{k=0}^{n-j}\begin{psmallmatrix}j-1\\k\end{psmallmatrix}\begin{psmallmatrix}n-j+1\\k+1\end{psmallmatrix} \,\, \text{if} \,\, \lceil\tfrac{n+1}{2}\rceil +1 \leq j \leq n-1.
\end{cases}
\]
We aim to show $f(n,j) = \begin{psmallmatrix}n \\ j\end{psmallmatrix}$ for all $1 \leq j \leq n-1.$ Once we have this,
$F(n) =  \sum_{k=0}^{n} \begin{psmallmatrix}n \\ j\end{psmallmatrix}= 2^n$
will follow. We prove by induction on $n\geq 2.$ If $n=2,$ then $j=1$ and we have $f(2,1)=2=\begin{psmallmatrix}2\\1\end{psmallmatrix}.$ Now assume $f(n,j)=\begin{psmallmatrix}n\\j\end{psmallmatrix}$ for $n \geq 2.$ Let $1 \leq j \leq \lceil \frac{n+1}{2}\rceil$ (the same argument will apply to the case $\lceil\tfrac{n+1}{2}\rceil +1 \leq j \leq n-1$). It suffices to show $f(n+1, j) - f(n,j)=\begin{psmallmatrix}n\\j-1\end{psmallmatrix}.$ We have
\begin{equation*}
f(n+1, j) - f(n,j) = \sum_{k=0}^{j-1}\begin{psmallmatrix}j-1\\k \end{psmallmatrix}\begin{psmallmatrix}n-j+1\\k\end{psmallmatrix} 
=\sum_{k=0}^{j-1} \begin{psmallmatrix}j-1\\k\end{psmallmatrix}\begin{psmallmatrix}n-j\\k\end{psmallmatrix} + \sum_{k=1}^{j-1}\begin{psmallmatrix}j-1\\k\end{psmallmatrix}\begin{psmallmatrix}n-j\\k-1\end{psmallmatrix}.
\end{equation*}
On the other hand,
\begin{equation*}
\begin{psmallmatrix}n\\j-1\end{psmallmatrix} = \tfrac{j}{n-j+1}\begin{psmallmatrix}n\\j\end{psmallmatrix} = \tfrac{j}{n-j+1}f(n,j) = \sum_{k=0}^{j-1}\begin{psmallmatrix}j\\k+1\end{psmallmatrix}\begin{psmallmatrix}n-j\\k\end{psmallmatrix} 
= \sum_{k=0}^{j-1} \begin{psmallmatrix}j-1\\k\end{psmallmatrix}\begin{psmallmatrix}n-j\\k\end{psmallmatrix} + \sum_{k=0}^{j-2} \begin{psmallmatrix}j-1\\k+1\end{psmallmatrix}\begin{psmallmatrix}n-j\\k\end{psmallmatrix}.
\end{equation*}
Thus $f(n+1,j)-f(n,j)=\begin{psmallmatrix}n\\j-1\end{psmallmatrix}$ holds and we finish the proof.

\noindent
(c) Suppose $P \neq S,$ or equivalently, $P_{\tau}\neq S_{\tau}.$
If $\Reflection{2} \notin P_{\tau},$ then $\Reflection{2}\Reflection{3}\Reflection{2} \in \TWeyl^P,$ which is not liftable (Lemma \ref{prop:basicnonliftable2} (III)(a)). If $\Reflection{1} \notin P_{\tau},$ then $\Reflection{2}\Reflection{3}\Reflection{2}\Reflection{1} \in \TWeyl^P,$ which is not liftable (Corollary \ref{prop:noliftingcondition}). If $\Reflection{3} \notin P_{\tau},$ then $\Reflection{3}\Reflection{4}\Reflection{3}\Reflection{2}\Reflection{3} \in \TWeyl^P,$ 
which is not liftable (Proposition \ref{prop:basicnonliftable2} (III)(b)). 
Finally, if $\Reflection{4} \notin P_{\tau},$ then $\Reflection{2}\Reflection{3}\Reflection{2}\Reflection{1}\Reflection{4}\Reflection{3}\Reflection{2}\Reflection{3}\Reflection{4} \in \TWeyl^P,$
which is not liftable (Corollary \ref{prop:noliftingcondition}).
Thus $\embedding^*$ does not have the lifting property. The other direction is clear.

\noindent
(d) Suppose $P \neq S,$ or equivalently, $P_{\tau}\neq S_{\tau}.$
If $\Reflection{3} \notin P_{\tau},$ then $\Reflection{3}\Reflection{2}\Reflection{3}\Reflection{4}\Reflection{3} \in \TWeyl^P,$ which is not liftable  (Proposition~\ref{prop:basicnonliftable2} (IV)(a)). 
If $\Reflection{1} \notin P_{\tau},$ then $\Reflection{3}\Reflection{2}\Reflection{3}\Reflection{4}\Reflection{3}\Reflection{1}\Reflection{2}\Reflection{4}\Reflection{3}\Reflection{4}\Reflection{3}\Reflection{2}\Reflection{1} \in \TWeyl^P,$
which is not liftable (Example~\ref{egg:typeE8lifting}). Similarly, if $\Reflection{2} \notin P_{\tau},$ then $\Reflection{3}\Reflection{2}\Reflection{3}\Reflection{4}\Reflection{3}\Reflection{1}\Reflection{2}\Reflection{4}\Reflection{3}\Reflection{4}\Reflection{3}\Reflection{2}$ $\in \TWeyl^P$ is not liftable.
Finally, if $\Reflection{4} \notin P_{\tau},$ then $\Reflection{3}\Reflection{2}\Reflection{3}\Reflection{4}\Reflection{3}\Reflection{4} \in \TWeyl^P$ is not liftable. Thus $\embedding^*$ does not have the lifting property. The other direction is clear.

\noindent
(e) $(\Ra)$ \, Suppose $\Para{1,5} \not\subset P ,$ or equivalently, $\TPara{1} \not\subset P_{\tau}.$ If $\Reflection{2} \notin P_{\tau},$ then $\Reflection{2}\Reflection{1}\Reflection{2}\Reflection{3}\Reflection{2} \in \TWeyl^P,$ which is not liftable (Proposition~\ref{prop:basicnonliftable2} (V)(a)). If $\Reflection{3} \notin P_{\tau},$ then $\Reflection{2}\Reflection{1}\Reflection{2}\Reflection{3}\Reflection{2}\Reflection{3} \in \TWeyl^P,$ which is not liftable (Corollary \ref{prop:noliftingcondition}). Thus $\embedding^*$ does not have the lifting property.

\noindent
$(\La)$ \, This direction is proved in Example \ref{egg:typeD6lifting} for $P=\Para{1, 5}$ and it is trivial for $P=S.$

\noindent
(f) By Example \ref{egg:typeA4lifting} and Proposition~\ref{prop:basicnonliftable2} (VI)(a), $u =\Reflection{1}\Reflection{2}\Reflection{1}\Reflection{2}\Reflection{1}$ is the only nonliftable element of $\TWeyl$. We have $u \notin \TWeyl^P$ if and only if $P_{\tau} \neq \emptyset.$ Hence it follows that $\embedding^*$ has the lifting property if and only if $P \neq \emptyset,$ i.e., $P=\Para{2, 4},$ $\Para{1, 3}$ or $S.$
\end{proof}

\subsection*{Acknowledgements}
I thank Kirill Zainoulline and Martina Lanini for introducing me to this area of study and suggesting me to look at this problem as well as for helpful discussions. The work was partly supported by Mitacs Globalink Research Award 2019. I also thank the Department of Mathematics at Universit{\`a} degli studi di Roma Tor Vergata for accommodating my research stay.

\bibliographystyle{unsrt}
\bibliography{Bib}

\begin{thebibliography}{10}

\bibitem{LZ18}
M.~Lanini and K.~Zainoulline.
\newblock Twisted quadratic foldings of root systems.
\newblock {\em St. Petersburg Math. J.}, 33:to appear, 2021.

\bibitem{St67}
R.~Steinberg.
\newblock {\em Lectures on {C}hevalley groups}.
\newblock Number~66 in Univ. Lect. Ser. Amer. Math. Soc., Providence, RI, 2016.

\bibitem{Springer}
T.~A. Springer.
\newblock {\em Linear algebraic groups}.
\newblock Birkh{\"a}user, Boston, 2nd edition, 2009.

\bibitem{Lusztig83}
G.~Lusztig.
\newblock Some examples of square integrable representations of semisimple
  $p$-adic groups.
\newblock {\em Trans. of Amer. Math. Soc.}, 277(2):623--653, 1983.

\bibitem{MP93}
R.~V. Moody and J.~Patera.
\newblock Quasicrystals and icosians.
\newblock {\em J. Phys. A: Math. Gen.}, 26(12):2829--2853, 1993.

\bibitem{F08coxeter}
P.~Fiebig.
\newblock The combinatorics of {C}oxeter categories.
\newblock {\em Trans. Amer. Math. Soc.}, 360(8):4211--4233, 2008.

\bibitem{GKM97}
M.~Goresky, R.~Kottwitz, and R.~MacPherson.
\newblock Equivariant cohomology, {K}oszul duality, and the localization
  theorem.
\newblock {\em Invent. {M}ath.}, 131(1):25--83, 1997.

\bibitem{CZZ15}
B.~Calm{\`e}s, K.~Zainoulline, and C.~Zhong.
\newblock Equivariant oriented cohomology of flag varieties.
\newblock {\em Documenta Math.}, Alexander S. Merkurjev's Sixtieth
  Birthday:113--144, 2015.

\bibitem{Borel54}
A.~Borel.
\newblock {K}{\"a}hlerian coset spaces of semisimple {L}ie groups.
\newblock {\em Proc. Nat. Acad, Sci. - PNAS}, 40(12):1147--1151, 1954.

\bibitem{Borel53}
A.~Borel.
\newblock Sur la cohomologie des espaces fibr{\'e}s principaux et des espaces
  homog{\`e}nes de groupes de {L}ie compacts.
\newblock {\em Ann. {M}ath.}, 57(1):115--207, 1953.

\bibitem{BGG73}
I.~N. Bernstein, I.~M. Gel'fand, and S.~I. Gel'fand.
\newblock {S}chubert cells and cohomology of the spaces {$G/P$}.
\newblock {\em Russ. Math. Surv.}, 28(3):1--26, 1973.

\bibitem{Demazure73}
M.~Demazure.
\newblock Invariants sym{\'e}triques entiers des groupes de {W}eyl et torsion.
\newblock {\em Invent. Math.}, 21(4):287--301, 1973.

\bibitem{Hiller82}
H.~Hiller.
\newblock {\em Geometry of {C}oxeter groups}.
\newblock Number~54 in Res. Notes in Math. Pitman Adv. Publ. Program, Boston,
  1982.

\bibitem{Ka11}
S.~Kaji.
\newblock Equivariant {S}chubert calculus of {C}oxeter groups.
\newblock {\em Proc. Steklov Inst. Math.}, 275(1):239--250, 2011.

\bibitem{SGA}
M.~Demazure and A.~Grothendieck.
\newblock {\em Sch{\'e}mas en groupes III}.
\newblock Number 153 in Lect. Notes in Math. Springer-Verlag, Berlin, 1970.

\bibitem{Humphreys}
J.~E. Humphreys.
\newblock {\em Reflection groups and {C}oxeter groups}.
\newblock Number~29 in Cambridge Stud. in Adv. Math. Cambridge Univ. Press,
  Cambridge England, 1st pbk. (with corrections) edition, 1994.

\bibitem{Brenti}
A.~Bj{\"o}rner and F.~Brenti.
\newblock {\em Combinatorics of {C}oxeter groups}.
\newblock Number 231 in Grad. Texts Math. Springer, New York, NY, 2005.

\bibitem{Stembridge}
J.~R. Stembridge.
\newblock Tight quotients and double quotients in the {B}ruhat order.
\newblock {\em Electron. J. Combin.}, 11(2):1--41, 2005.

\bibitem{DLZ19}
R.~Devyatov, M.~Lanini, and K.~Zainoulline.
\newblock Oriented cohomology sheaves on double moment graphs.
\newblock {\em Documenta Math.}, 24:563--608, 2019.

\bibitem{AppD}
H.~H. Andersen, J.~C. Jantzen, and W.~Soergel.
\newblock {\em Representations of quantum groups at a $p$-th root of unity and
  of semisimple groups in characteristic $p$: independence of $p$}.
\newblock Soci{\'e}t{\'e} math{\'e}matique de France, Paris, 1994.

\bibitem{Billey97}
S.~Billey.
\newblock Kostant polynomials and the cohomology ring for {$G/B$}.
\newblock {\em Proc. of Nat. Acad. of Sci. - PNAS}, 94(1):29--32, 1997.

\bibitem{Hahn}
A.~J. Hahn.
\newblock {\em Quadratic Algebras, {C}lifford Algebras, and Arithmetic {W}itt
  Groups}.
\newblock Universitext. Springer US, New York, NY.

\end{thebibliography}
\end{document}